\documentclass[10pt,amsfonts, epsfig]{amsart}
\usepackage{amsmath, amscd, amssymb}
\usepackage{graphicx, psfrag}

\usepackage{graphpap, color}
\usepackage{mathrsfs}
\usepackage{pstricks}
\usepackage{color}
\usepackage{cancel}
\usepackage[mathscr]{eucal}
\usepackage{pstricks}
\usepackage{color}
\usepackage{cancel}
\usepackage{verbatim}
\usepackage{latexsym}
\usepackage[all]{xy}

\usepackage{latexsym}
\usepackage{amsxtra}
\usepackage{verbatim}
\usepackage{color}
\usepackage{amsthm,amsmath,amsfonts,amssymb,mathrsfs,amscd,graphics,amscd}
\usepackage{amssymb}
\usepackage{appendix}
\usepackage{hyperref,supertabular}
\usepackage{ifpdf}
\usepackage{enumerate}
\usepackage{fancyhdr}
\usepackage[margin=3cm]{geometry}
\allowdisplaybreaks

\newtheorem{theorem}{Theorem}[section]
\newtheorem{lemm}[theorem]{Lemma}

\def\loc{_1^\circ}

\def\sO{{\mathscr O}}
\def\sC{{\mathscr C}}

\def\sN{{\mathscr N}}
\def\sL{{\mathscr L}}
 \def\sZ{{\mathscr Z}}

\def\sO{\mathscr{O}}

\def\sF{\mathscr{F}}

\def\sR{\mathscr{R}}

\def\sZ{\mathscr{Z}}

\newcommand{\CC}{\mathbb{C}}
\newcommand{\NN}{\mathbb{N}}
\newcommand{\EE}{\mathbb{E}}

\newcommand{\PP}{\mathbb{P}}
\def\P5{{\PP^5}}
\newcommand{\QQ}{\mathbb{Q}}
\newcommand{\RR}{\mathbb{R}}
\newcommand{\ZZ}{\mathbb{Z}}

\def\sK{{\mathscr K}}

\newcommand{\bL}{\mathbf{L}}

\def\Gm{{\bG_m}}
 
\def\upmo{^{-1}}

\newcommand{\vir}{ {\mathrm{vir}} }
\newcommand{\ev}{ \mathrm{ev} }

\newcommand{\vGa}{\Ga} 

\newcommand{\cal}{\mathcal}
\def\ee{\mathfrak e}

\def\cC{{\cal C}}

\def\cL{{\cal L}}
\def\cM{{\cal M}}
\def\cN{{\cal N}}
\def\cO{{\cal O}}

\def\cW{{\cal W}}

\def\cZ{{\cal Z}}

\def\fB{\mathfrak{B}}

\def\fD{\mathfrak{D}}

\def\fX{\mathfrak{X}}

\def\ff{\mathfrak{f}}

\def\ft{\mathfrak{t}}

\def\v1{{\vec{1}}}

\def\sP{{\mathscr P}}

\newcommand{\Mbar}{\overline{\cM}}

\newcommand{\vd}{\vec{d}}

\def\sta{^\ast}

\def\virt{^{\mathrm{vir}}}
\def\upmo{^{-1}}
\def\sta{^{\ast}}

\def\sta{\sta}

\def\lsta{_{\ast}}

\newcommand{\Si}{\Sigma}
\newcommand{\Ga}{\Gamma}

\newcommand{\lam}{\lambda}

\def\sQ{\mathscr Q}

\def\begeq{\begin{equation}}
\def\endeq{\end{equation}}
\def\and{\quad{\rm and}\quad}
\def\bl{\bigl(}
\def\br{\bigr)}

\def\sub{\subset}

\def\and{\quad\text{and}\quad}

 \DeclareMathOperator{\Aut}{Aut}

 \DeclareMathOperator{\rank}{rank}

\def\lggd{_{g,\gamma,\bd}}

\def\loc{_{\mathrm{loc}}}

\def\ev{\text{ev}}

\def\sta{^\ast}

\def\ti{\tilde}

\def\sO{{\mathscr O}}
\def\sW{{\mathscr W}}

\def\sR{{\mathscr R}}

 \def\ga{{\Gamma}}

\def\beq{\begin{equation}}
\def\eeq{\end{equation}}

\def\Pf{{\PP^4}}

\def\bee{\begin{equation}}
\def\eeq{\end{equation}}

\def\sC{{\mathscr C}}

\def\fA{{\mathfrak A}}
\def\fR{{\mathfrak R}}

\def\bd{{\mathbf d}}

\def\ti{\tilde}

\def\barM{{\overline{M}}}

\def\Gm{T}
\def\mT{{\mathbb T}}

\let\ga=\Ga

\def\lorho{_{^{(1,\rho)}}}

\def\1rho{1_\rho}

\def\virt{^{\vir}}
\def\virtloc{\virt\loc}

\def\sO{{\mathscr O}}
\def\sC{{\mathscr C}}

\def\sN{{\mathscr N}}
\def\sL{{\mathscr L}}
 \def\sZ{{\mathscr Z}}

\def\sO{\mathscr{O}}

\def\sF{\mathscr{F}}

\def\sR{\mathscr{R}}

\def\sZ{\mathscr{Z}}

\def\licolor{\color{cyan}}

\newtheorem{prop}{Proposition}[section]
\newtheorem{theo}[prop]{Theorem}
\newtheorem{coro}[prop]{Corollary}
\newtheorem{rema}[prop]{Remark}

\newtheorem{defi}[prop]{Definition}

\newtheorem{defi-prop}[prop]{Definition-Proposition}
\newtheorem{defi-theo}[prop]{Definition-Theorem}
\begin{document}

\title[genus one]{Genus one GW invariants of quintic threefolds via MSP localization}

\author[Huai-Liang Chang]{Huai-Liang Chang$^1$}
\address{Mathematics Department, Hong Kong University of Science and Technology}
 \email{mahlchang@ust.hk}
\thanks{${}^1$Partially supported by   Hong Kong GRF Grant 6301515}
\author[Shuai Guo]{Shuai Guo$^2$}
\address{School of Mathematical Sciences and Beijing International Center for Mathematical Research, Peking University}
\email{guoshuai@math.pku.edu.cn}
\thanks{${}^2$Partially supported by NSFC grants 11431001 and 11501013}
\author[Wei-Ping Li]{Wei-Ping Li$^3$}
\address{Mathematics Department, Hong Kong University of Science and Technology}
\email{mawpli@ust.hk}
\thanks{${}^3$Partially supported by   Hong Kong GRF Grant 6301515}
\author[Jie Zhou]{Jie Zhou$^4$}
\address{Mathematical Institute, University of Cologne, Weyertal  86-90, 50931 Cologne, Germany}
\email{zhouj@math.uni-koeln.de}
\thanks{${}^4$Partially supported by   German Research Foundation Grant CRC/TRR 191.}


\maketitle

\begin{abstract}  
 The moduli stack of Mixed Spin P-fields (MSP)   provides an effective algorithm  to evaluate all genus 
Gromov-Witten invariants of quintic Calabi-Yau threefolds.
This paper is to apply the algorithm in genus one case. We use the localization formula, the proposed algorithm in \cite{CLLL1, CLLL2}, and Zinger's packaging technique to compute the genus one Gromov-Witten invariants of quintic Calabi-Yau threefolds. 
New hypergeometric series identities are also discovered in the process. 

\end{abstract}

\section{Introduction} 
The genus one Gromov-Witten invariants of Calabi-Yau hypersurfaces were calculated by Zinger in \cite{Zi2}. The method consists of two steps. The first step is the detailed study of the moduli space of genus one stable maps to CY hypersurfaces in \cite{LZ, VZ}. The second step is the  computation part using torus localization and some packaging techniques in \cite{Zi2}.

In this paper, we will use the MSP set-up to calculate the genus one Gromov-Witten invariants for quintic CY hypersurfaces in $\mathbb P^4$. As a by product, we find integral relations for 
certain differential polynomials of $I$ functions (e.g. section \ref{diffrelation}). \black

Mixed-spin $P$-fields (MSP)  are defined in \cite{CLLL1} (see \cite{CLLL3} for a survey). An MSP field   is given  by $$
\xi=(\sC,\Si^\sC,\cL,\cN, \varphi,\rho,\nu)
$$
with   stability conditions,
where $\sC$ is a twisted nodal curve (or an orbifold curve) with markings $\Si^\sC$, $\cL$ and $\cN$ are invertible sheaves on $\sC$, and $\varphi,\rho,\nu$ are sections of $\cL^{\oplus 5}, \cL^{\vee5}\otimes \omega^{\log}_{\sC}, \cL\otimes \cN\oplus \cN$ respectively. For simplicity, let's consider the case without markings. Then $\xi$ has the numerical data $(g, \bd)$ where $g$ is the genus of the curve $\sC$ and $\bd=(d_0, d_\infty)$ with $d_0=\deg(\cL\otimes \cN)$ and $d_\infty=\deg( \cN)$. We use $\cW_{g,\bd}$ to denote the moduli stack of MSP fields with numerical data $(g, \bd)$. The virtual dimension of $\cW_{g,\bd}$ is $d_0+d_\infty-g+1$. 

  MSP moduli stacks can be regarded as a platform interpolating  the moduli of Gromov-Witten theory  and  the moduli of FJRW theory (see \cite{FJR1, FJR2}). Using the $P$-field theory of \cite{CL} and the cosection localization technique of \cite{KL}, a virtual cycle can be defined  and  the properness can be proved in \cite{CLLL1}. The moduli stack admits a $\mathbb C^*$-action. Using this torus action, if the virtual dimension of $\cW_{g,\bd}$ is bigger than zero, the virtual cycle $[\cW_{g,\bd}]^{\virt}$ can give us the identity:
  \begin{eqnarray}
  \label{PI}
  0 =\int_{[\cW_{g,\bd}]\virt} \ft \ c_{\rm top-1}\big(R\pi_{\cW\ast}(\cL_\cW \cN_\cW \cL_\ft)\big),
  \end{eqnarray}
   where $\cL_\cW$ and $\cN_\cW$ are the universal line bundles on the universal curve $\cC$ over the moduli   $\cW:=\cW_{g,\bd}$ with a projection $\pi_\cW\colon\cC\to \cW$,  top is the virtual rank of  $R\pi_{\cW\ast}(\cL_\cW \cN_\cW)$, $\ft$ is the generator of the equivariant cohomology $H_{\mathbb C^*}^*(B\mathbb C^*, \mathbb Q)=\mathbb Q[\ft]$, and $\cL_\ft$ is the trivial line bundle with the weight-$1$ $\mathbb C^*$-action on fibers .
  
  The identify \eqref{PI} provides lots of equations among GW invariants, Hodge integrals, and FJRW invariants. This  will potentially provide an effective algorithm to calculate GW invariants as well as FJRW invariants. In this paper, we will consider the virtual cycles $[\cW_{1,(0, n)}]^{\virt} $ for all $n$ to compute genus one GW invariants of Fermat quintic CY three-folds. 
  
  Comparing with Zinger's work, MSP moduli replaces the detailed study of moduli spaces of stable  maps, especially the complicated issue of the ghost component,  by FJRW invariants.   Since the study of the ghost components for higher genus is very complicated  as it requires separation of components with different $P$-field (ghost) rank (c.f. \cite{CL1}),  \black MSP moduli provides a good way  to work on higher genus GW invariants  (see \cite{MP} for another approach).  For example, our method does not need genus zero two point functions. 
   The $g=1$ package 
  in this paper is expected to be directly extended to higher genus package via MSP moduli.  
    
  The $\mathbb C^*$-action gives four types of graphs from the standard localization procedure, called type A, B, C and D (see Figure \ref{FigABCD}). The contribution from type A graphs involves genus one GW invariants of the quintic, while that from type D graphs involves FJRW invariants. Our packaging uses Zinger's  technique (\cite{Zi2}). We also discover (c.f. Section C.1.2) new hypergeometric series identities in the package of loop type (B) graphs.
  
   The method of quasimaps also provides another way to calculate genus one GW  invariants of CY hypersurfaces \cite{CK1, CK2, KLh}.

  The organization of the paper is as follows. In \S2, we review the
  graphs notations and their contributions in the localization of MSP theory in \cite{CLLL1, CLLL2}. In \S3, we review the mirror theorem for genus zero GW invariants of  quintic threefolds \cite{Gi1, Gi2} (also see \cite{LLY}) and discuss how to transform Givental's set-up to our MSP set-up. In \S4, we define $\sK$-series and $\sP$-series, which are essential for the calculations in later sections. We also calculate some $\sK$-series. In \S5, we  compute the contribution of type A graphs to \eqref{PI}. In \S6, we  compute the contribution of type B graphs to \eqref{PI}. In \S7, we  compute the contribution of type C graphs to \eqref{PI}. In \S8, we  compute the contribution of type D graphs to \eqref{PI}.  Many computations in the last four sections are contained in the Appendix.  
  
  \medskip
  
  {\sl Acknowledgement}. We would  like to thank Professor Jun Li for all  his help and many discussions about the computations. We also thank Professor Melissa Liu for all her help.
      
  Part of the J.~Z.'s work is done while he was a postdoc whose research was supported by
  Perimeter Institute for Theoretical Physics. Research at Perimeter Institute is supported by the Government of Canada through Innovation, Science and Economic Development Canada and by the Province of Ontario through the Ministry of Research, Innovation and Science.
      
  \medskip

Notations:  Throughout the  paper we fix the notations  $T=\CC\sta$, $\ft =-c_1(\sO_{\CC\PP^\infty}(1))$ for $BT=\CC\PP^\infty$, $q=e^t$, and $h\in H^2(\Pf,\ZZ)$ always denotes the hyperplane class.

\black

 \section{MSP theory and localization set-up}
 
 In this section, we will review the localization scheme of the moduli stacks of 
 mixed-spin P-fields   in \cite{CLLL1,CLLL2}.  There are four types of graphs in genus one localization, labeled as type A, B, C and D. 
 
 \subsection{Review of MSP theory}

  The MSP moduli $\cW=\cW_{g,(1,\rho)^n,\bd}$ is constructed in \cite{CLLL1}. It is indexed
  by genus $g$, the number $n$ of markings labelled by $(1,\rho)$, and two integers $(d_0,d_\infty)=\bd$. It is a separated DM stack
  locally of finite type equipped with a cosection $\sigma$ and a 
  $T=\CC\sta$ action. The cosection's zero locus $\cW^-_{g,(1,\rho)^n,\bd}$
  is shown to be proper \cite{CLLL1}  and the torus action gives rise to 
  an equivariant virtual cycle $[\cW\lggd]\virtloc\in A^{\Gm}\lsta (\cW\lggd^-)^T
$. It has a localization formula (for $\delta:=d_0+d_\infty+1-g+n$)
\beq\label{loc-form}[\cW]\virtloc=
\sum_{\ga \black 
} \jmath^-_{\ga\ast}\Bigl( \frac{[\cW_{\Ga}]\virtloc}{e(N_\ga\virt)}\Bigr)\in
A_{\delta}^\Gm (\cW^{-})^T [\ft\upmo],
\eeq
where each connected component of fixed locus $\cW^T$ is indexed by a graph $\ga$ and $\jmath^-$ is the natural inclusion from $\cW_\ga$ to $\cW^-$. Each graph $\ga$ is connected, and has all vertices labelled by $0,1$ or $\infty$, namely $V(\ga)=V_0(\ga)\cup V_1(\ga)\cup V_\infty(\ga)$ according to $\nu_1=0, \nu_1=1=\nu_2$ and $\nu_2=0$ respectively,  and edges 
$E(\ga)=E_0(\ga)\cup E_\infty(\ga)$
where edges in $E_0(\ga)$ connect vertices $V_0(\ga)$ with $V_1(\ga)$, and  
edges in $E_\infty(\ga)$ connect vertices $V_0(\ga)$ with $V_1(\ga)$.
 Denote the set $F(\Ga)$ of flags  to be the set of  $(e,v)\in E(\ga)\times V(\ga)$ such that $e,v$ are adjacent.  The graphs is required to possess decorations as follows:

\begin{itemize}
\item[(a)] (genus)  $g_\cdot: V(\Ga)\to \ZZ_{\geq 0}$; 
\item[(b)] (degree)  $\vd: E(\Ga)\cup V(\Ga)\to \QQ^{\oplus 2}$ via $\vd(a)=(d_{0a},d_{\infty a})$,
where $d_{0a}= \deg \sL\otimes\sN|_{\sC_a}$ and $d_{\infty a}= \deg\sN|_{\sC_a}$  with $\sC_a$ being the curve associated to $a$. 
\item[(c)] (marking)  $S_\cdot: V(\Ga)\to 2^{L(\Ga)}$ via
$v\mapsto S_v\subset L(\Ga)$, where $S_v$ is the subset of markings $\Si^{\sC}_i\in \sC_v$. 
\end{itemize}
 We call a vertex $v\in V(\ga)$ unstable if $g_v=0$, $d_v=\deg(\sL|_{\sC_v})=0$  and $|E_v\cup S_v|
<3$, otherwise stable. Thus $V(\ga)=V^S(\ga)\cup V^{US}(\ga)$. For any unstable vertex $v\in V^{US}(\ga)$,  it must lie in $V^{0,1}$, $V^{1,1}$ or $V^{0,2}$, where
\beq\label{VV0}
V^{a,b} (\Ga):=\{v\in V(\Ga)-V^S(\Ga)\, |\,  |S_v|=a, |E_v|=b\}.
\eeq

Let  $V_c^{a,b}(\Ga)=V_c(\Ga)\cap V^{a,b}(\Ga)$ and 
$\Delta\lggd$ denote the set of all (regular) graphs, up to graph isomorphisms preserving decorations. For each graph $\ga\in\Delta\lggd$ we use $(\ga)$ to denote its
class in $\Delta\lggd/\sim$, where $\sim$ is by automorphisms of graphs.
Then $\cW\lggd^{T}$ is a disjoint union of $\cW_{(\ga)}$, the locus whose curves are associated to the graph $\ga$ (c.f. \cite[Def 2.5]{CLLL2}), over $\ga$'s in $\Delta\lggd$. In \cite{CLLL2} the fixed locus of $\cW\lggd$ decorated by $\ga$ is denoted by $\cW_{\ga}$. In this paper we will no distinguish $\cW_{\ga}$ from $\cW_{(\ga)}$, and we always use $\pi_\ga:\cC_{\cW_\ga}\to\cW_\ga, \cL_\ga$ to denote its universal curve and universal line bundle.

 For   each $v\in V^S(\Ga)$, there is   a moduli stack $\cW_v$ (\cite[Cor. 2.37]{CLLL2}), where
$$\cW_{v}\cong \Mbar_{g_v, E_v\cup S_v}(\PP^4,d_v)^p,\quad
\cW_{v}\cong \Mbar_{g_v,E_v\cup S_v},\quad \cW_v\cong \barM_{g_v,\gamma}^{1/5.5p},$$
when $v\in V^S_0(\vGa), V_1^S(\ga)$ and $V_\infty^S(\ga)$  respectively. Here $\Mbar_{g_v, E_v\cup S_v}(\PP^4,d_v)^p$ is the moduli stack of degree $d_v$ stable maps from genus $g_v$ nodal curves  with markings in $E_v\cup S_v$ to $\PP^4$ with  P-fields studied in \cite{CL}, and  $\barM_{g_v,\gamma}^{1/5.5p}$ is the moduli stack
of genus $g_v$ $\mu_5$-twisted curve with a five-spin structure, markings labeled by $\gamma$ and five $P$-fields studied in \cite{CLL}.
\black

   For $e\in E_0$, there is an MSP moduli $\cW_e=\sqrt[d_e]{\cO_{\PP^4}(1)/\PP^4}$  with $\pi_e: \cW_e\to \PP^4$  the map to its coarse moduli space. Let $h\in H^2(\PP^4;\QQ)$ be the hyperplane class,    $h_e=\pi_e^*h \in H^2(\cW_e;\QQ)$, and   $d_e$ be the degree of $\cL|_{\sC_e}$.

 For unstable $v\in V_0-V_0^S$, we let $[\cW_v]\virtloc=-[Q_5]$, where $Q_5\sub \Pf$ is the Fermat quintic.
\black


\begin{prop} For each graph $\ga$ appearing in localizing $\cW_{g=1,(1,\rho)^n,(d,0)}$,  we have $$[\cW_{(\Ga)}]\virtloc= \frac{1}{|\Aut(\Ga)|} \frac{1}{\prod_{e\in E_0} d_e} 
\frac{1}{\prod_{e\in E_\infty} |5d_e|}
\prod_{v\in V_0\cup V_1^S\cup V_\infty^S} [\cW_v]\virtloc
$$
\end{prop}

 \black
 
\begin{theo}
	$$
	\frac{1}{e_T(N^\vir_{\vGa})} =\prod_{v\in V(\Ga)}B_v \prod_{e\in E(\Ga)} A_e,
	$$
	where 
	$$
	B_v=\begin{cases}
	A_v\cdot \displaystyle{ \prod_{e\in E_v}\frac{1}{w_{(e,v)}-\psi_{(e,v)}} } , &  v\in V^S;\\
	\displaystyle{ \frac{A_v}{w_{(e,v)}+w_{(e',v)}} }, & v\in V^{0,2}, E_v=\{e,e'\};\\
	A_v, & v\in V^{1,1};\\
	A_v\cdot w_{(e,v)}, & v\in V^{0,1}, E_v=\{e\}.
	\end{cases}
	$$
\end{theo}

The formulae  for  $A_v, A_e$ and $w_{(e, v)}$ are provided below. 

\begin{lemm}\label{wev} Let $v$ and $v'$ be the two vertices of an edge $e$.
\begin{enumerate}
\item When $v\in V_0$, then $w_{(e,v)}=\frac{h_e +\ft}{d_e}$ and $ w_{(e,v')}=-\frac{h_e+\ft}{d_e}$.
\item When $v\in V_\infty\setminus V_\infty^{0,1}$, then $w_{(e,v)}=\frac{\ft}{r_e d_e}$ and $ w_{(e,v')}=-\frac{\ft}{d_e}$.
\end{enumerate}
\end{lemm}

\begin{lemm}[Contribution from  edges]\label{lem:all-edges} For the two cases $(1)$\, $e\in E_0$, (2)\, $e\in E_\infty,\ (e,v)\in F,\ v\in V_\infty \setminus V_\infty^{0,1}$,
we have respectively
$$ A_e=\displaystyle{ \frac{\prod\limits_{j=1}^{5d_e-1}(-5 h_e +\frac{j(h_e+\ft)}{d_e})}{
			\prod\limits_{j=1}^{d_e}(h_e-\frac{j(h_e+\ft)}{d_e})^5 \cdot \prod\limits_{j=1}^{d_e}\frac{j(h_e+\ft)}{d_e} }}, \qquad \qquad
		\displaystyle{\frac{\prod\limits_{j=1}^{\lceil-d_e\rceil-1}(-\ft-\frac{j\ft}{d_e} )^5 }{
			\prod\limits_{j=1}^{{-5d_e}}(-\frac{j\ft}{d_e}) 
			\prod\limits_{j=1}^{\lfloor-d_e \rfloor}(\frac{j\ft}{d_e})}
	}. 
		$$

 	\end{lemm}

\begin{lemm}[Contribution from unstable  vertices]\label{lem:unstable-vertex}
	If $v\in V^{0,2}$, then
	$$
	A_v= \begin{cases}
	h_e +\ft = h_{e'}+\ft, & v\in V_0^{0,2}\ \text{and}\ E_v=\{e,e'\},\\
	-5\ft^6, & v\in V_1^{0,2}, 
	\end{cases} 
	$$

	If $v\in V_1^{1,1},\ S_v\subset \Si^{\sC}\lorho$, then we have 
	$
	A_v= 
 	5\ft. 
	$
	If $v\in V_0^{1,1}$, we have  $A_v=1$. 

\end{lemm} 
 
Let's introduce some notations. 
\begin{itemize}
\item Given a stable vertex $v\in V^S$, let $\pi_v: \cC_v\to \cW_v$ be the universal curve,  $\cL_v$  be
the universal line bundle over $\cC_v$.
\item Given $v\in V^S_1$, let $\EE_v:= \pi_*\omega_{\pi_v}$ be the Hodge bundle, where
$\omega_{\pi_v}\to \cC_v$ is the relative dualizing sheaf. Then $\EE^\vee_v=R^1\pi_{v*}\cO_{\cC_v}$.
\item Let $\bL_k$ be the one-dimensional weight $k$ $T$-representation.
\end{itemize}

\begin{lemm}[Contribution from all stable vertices]\label{lem:stable-vertex}
 For $v\in V^S$, define
\begin{equation}\label{eqn:Av}
	\begin{aligned}
		 A_v:= \begin{cases}
			\displaystyle{ \frac{1}{e_T(R\pi_{v*} {\phi}_v^* \cO_{\PP^4}(1)\otimes \bL_1)}
				\prod_{e\in E_v}(h_e+\ft) },
			& v\in V^S_0; \\
			\displaystyle{\Big(\frac{e_T(\EE_v^\vee \otimes \bL_{-1})}{-\ft}\Big)^5
				\cdot \frac{5\ft}{e_T(\EE_v\otimes \bL_5)}\cdot (-\ft^5)^{|E_v|}\cdot \big(\frac{-\ft^4}{5}\big)^{|S_v^{(1,\varphi)}|}  , } 
			& v\in V_1^S;\\
			\displaystyle{ \frac{1}{e_T(R\pi_{v*}\cL_v^\vee \otimes \bL_{-1})} 
				 }, & v\in V^S_\infty.
		\end{cases}
	\end{aligned}
\end{equation}
\end{lemm}

\subsection{Types A, B, C and D} When the genus is one, we can divide the type of regular graphs associated to $T$-fixed points of $\cW$ into four types A, B, C and D. They are given as follows.

 Type A: This graph has one stable vertex $v\in V_0$ with genus $g_v=1$.  
 
 Type B: This graph has a loop with $\sC_0$ consisting of possibly many rational curves and/or points and $\sC_1$ being  rational curves and/or  points. $\sC_\infty$ is empty. Here $\sC_0=\sC\cap (\nu_1=0)$ and similarly for $\sC_1$ and $\sC_\infty$. 
 
 Type C: This graph has one stable vertex $v\in V_1$ with genus $g_v=1$. 
 
Type D: This graph has one stable vertex $v\in V_\infty$ with genus $g_v=1$. All edges $e$ in $E_{\infty}$  has $\deg(\cL|_{\sC_e})=-\displaystyle\frac{1}{5}$.

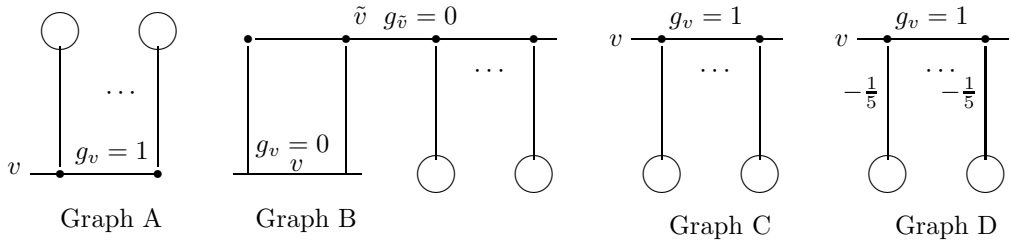
\label{FigABCD}\begin{center}
\begin{picture}(100,30)

\put(-11,24){\circle{5}}
 \put(-11,22){\line(0,-1){16}}
\put(-11,5){\circle*{1}}

\put(2,24){\circle{5}}
 \put(2,22){\line(0,-1){16}}
\put(2,5){\circle*{1}}

\put(-5, 15){$\cdots$}

 \put(-15,5){\line(1,0){17}}

\put(-18, 5){$v$} 
\put(-9,7){$g_v=1$}

    \put (-11, -2){Graph A}


\put(14,23){\circle*{1}}

 \put(14,22){\line(0,-1){17}}

\put(27,23){\circle*{1}}
 \put(27,22){\line(0,-1){17}}

\put(12,5){\line(1,0){17}}
\put(19.5, 5.5){$v$}

\put(39,23){\circle*{1}}
\put(44,18){$\cdots$}
 \put(39,22){\line(0,-1){15}}
\put(39,5){\circle{5}}

\put(52,23){\circle*{1}}
 \put(52,22){\line(0,-1){15}}
\put(52,5){\circle{5}}

 \put(15,23){\line(1,0){40}}

\put(28, 25){$\ti v$}
\put(32,25){$g_{\ti v}=0$}

 \put(15,23){\line(1,0){40}}

\put(15,8){$g_{v}=0$}

\put(15, -2){Graph B}

\put(69,23){\circle*{1}}
\put(74,18){$\cdots$}
 \put(69,22){\line(0,-1){15}}
\put(69,5){\circle{5}}

\put(82,23){\circle*{1}}
 \put(82,22){\line(0,-1){15}}
\put(82,5){\circle{5}}

 \put(65,23){\line(1,0){20}}

\put(62, 22){$v$} 
\put(70,25){$g_v=1$}

\put(70,-3){Graph C}

\put(99,23){\circle*{1}}
\put(104,18){$\cdots$}
 \put(99,22){\line(0,-1){15}}
\put(93,15){$-\frac{1}{5}$}
\put(99,5){\circle{5}}

\put(112,23){\circle*{1}}
 \put(112,22){\line(0,-1){15}}
\put(106,15){$-\frac{1}{5}$}
\put(112,5){\circle{5}}

 \put(95,23){\line(1,0){20}}

\put(92, 22){$v$} 
\put(100,25){$g_v=1$}

    \put(100,-3){Graph D}
    
\end{picture}
\end{center}       
\bigskip

  Each   circle represents    all possible graphs (with one marking) for  genus zero and $d_\infty=0$ MSP moduli, where the marking lies on $V_0$ or $V_1$ accordingly. \black
 
\section{Equivariant Mirror Theorem}

\subsection{Equivariant Mirror Theorem} 


	
	
	
   Let $G=\CC\sta$ act on $\P5$ as follows:
	 $t\cdot [x_0,\ldots, x_{5}]= [x_0,\ldots, tx_{5}].$ 
Over
 $$
H\sta_{G}(\P5;\QQ)=\QQ[p,\ft]/\langle p^{5}(p+\ft)\rangle
$$
  there is a  $G$-equivariant pairing  given by $( p^i, p^j) = 
(-\ft)^{i+j-5}$ if  $i+j\geq 5$ and $0$ otherwise.  The ordered basis
$$
(\mT^0,\mT^1,\mT^2,\mT^3,\mT^4,\mT^5):=(p^4(p+\ft), p^3(p+\ft), p^2(p+\ft), p(p+\ft), (p+\ft), 1)
$$
is dual to the ordered basis $$(\mT_0,\mT_1,\mT_2,\mT_3,\mT_4,\mT_5):=(1,p,p^2,p^3,p^4,p^5)$$ with respect to the $G$-equivariant Poincar\'{e} pairing,  namely $(\mT^i,\mT_j)=\delta_{i,j}$.

Let $ G'=\CC\sta$ acts on $\P5$ trivially. Consider the $G\times G'$ equivariant bundle over $\P5$
\beq\label{defV} V=L_H^{\otimes 5}\oplus (L_H\otimes L_{\ft}\otimes L_{\ft'}),\eeq
 where $\ft'=c_1\big(\sO_{BG'}(-1)\big)$ (with $BG'\cong \CC\mathbb P^\infty$) denotes the first Chern class of $L_{\ft'}$ and similarly for $\ft$ with the group $G'$ replaced by $G$. Let  $(\pi_d,\ff_d):\cC_d\to \barM_{0,1}(\P5,d)\times \P5$ be the universal family and $E_d=\pi_{d\ast}\ff_d^* V$.


  Recall the conventions $q=e^t$,
 $  f(q) =\sum_{d=0}^\infty \frac{(5d)!}{(d!)^5} q^d =I_0(t), \  g_l(q)=\sum_{d=1}^\infty q^d \frac{(5d)!}{(d!)^5}\Big(\sum_{m=1}^{ld}\frac{1}{m}\Big), $
$$\  I_1(t):=5[g_5-g_1](q)+tI_0(t),  \quad T(t):=I_1(t)/I_0(t).$$
We set 
$
T_0 
= \hbar \log f +\ft'\frac{g_1}{f}.
$  \black
 Then 
the equivariant mirror theorem \cite{Gi2} states that

 $$ 
I(q)= q^{  p/\hbar} 5p(p+\ft+\ft') \sum_{d=0}^\infty q^d
\frac{ \prod_{m=1}^{5d}(5 p +m\hbar)\prod_{m=1}^d(p+\ft+\ft'+m\hbar)}
{\prod_{m=1}^d(p+ m\hbar)^5(p+\ft+m\hbar)} 
 $$
 is equal to, under the mirror map $\sQ(q)=e^{I_1(t)/I_0(t)}=e^T$,
 \begin{eqnarray*}J(\sQ)&=& e^{(T_0+p\log \sQ)/\hbar} \bigg( 5p(p+\ft+\ft') 
+\sum_{d=1}^\infty \sQ^d   \sum_{k=0}^{5} \mT_k \int_{\Mbar_{0,1}(\P5,d)} \ev^*\mT^k\frac{e(E_d)}{\hbar(\hbar-\psi)}    \bigg).
\end{eqnarray*}

 \subsection{Application to the  MSP set-up}

 Recall $\cW_d:=\cW_{g=0,(1,\rho),(d,0)}=\barM_{0,1}(\P5,d)$. 
Let  $$F_{d}=\pi_{d\ast}\ff_d^* L_H^{\otimes 5}, \qquad R_d=\pi_{d\ast}\ff_d^* (L_H\otimes L_\ft).$$  
 \black
  Then  $\rank R_d=d+1$ and $E_d=F_d\oplus (R_d\otimes L_{\ft'})$. Thus
   $$e(E_d)=e(F_d)e(R_d\otimes L_{\ft'})=e(F_d) \sum_{k=0}^\infty c_{\rank R_d-k}(R_d) (\ft')^k.$$

 With $[\Mbar_{0,1}(\P5,d)]\cap  e(F_d) = (-1)^{d+1}[\cW_{g=0,(1,\rho),(d,0)}]\virt$, the mirror theorem is rephrased as follows.
 \begin{coro}\label{closer}  
 The series $e^{-(T_0+p\log \sQ)/\hbar} I(q) $ is equal to 

\begin{eqnarray*} 5p(p+\ft+\ft')
-\sum_{d=1}^\infty (-\sQ)^d   \int_{[\cW_{g=0,(1,\rho),(d,0)}]\virt} \bigg( \sum_{k=0}^{4} p^k \frac{e(R_d\otimes L_{\ft'})}{\hbar(\hbar-\psi)} \cup  {\rm ev}^*_1 [p^{4-k}(p+\ft)]
+ p^5 \frac{e(R_d\otimes L_{\ft'})}{\hbar(\hbar-\psi)}  \bigg). 
\end{eqnarray*}
 \end{coro}
 
  \black

 For each $j\geq 0$, define
 $$U_j:=-\sum_{d=1}^\infty (-\sQ)^d  \int_{[\cW_{g=0,(1,\rho),(d,0)}]\virt} \frac{e(R_d\otimes L_{\ft'})}{\hbar(\hbar-\psi)} \cup \ev_1^* p^j . $$ 
 Then $Y:=Y(p,q):=e^{-(T_0+p\log \sQ)/\hbar} I(q)$ in Corollary \ref{closer} equals
 $$ 5p(p+\ft+\ft') +p^0(U_5+\ft U_4)+p^1(U_4+\ft U_3)+p^2(U_3+\ft U_2)+p^3(U_2+\ft U_1)+p^4(U_1+\ft U_0)+p^5 U_0.$$
We also write  $Y(p)=\sum_{j=0}^5 ({\rm Coe}_{p^j}Y)  \  p^j$. 
 We define 
$$  \ti\sZ_{1,5} :=   - \sum_{d=1}^\infty (-\sQ)^d \sum_{(\ga)\in \Xi^{1}_{d}} \int_{[\cW_{(\ga)}]\virt} \frac{1}{\hbar(\hbar-\psi)}   \frac{e(R\pi_{\ga\ast}(\cL_\ga \cL_{\ft} \cL_{\ft'}(-D_\ga)))}{e(N_{\cW_{(\ga)}/\cW}^{vir})}$$
 $$\quad = - \frac{1}{\ft'} \sum_{d=1}^\infty (-\sQ)^d \sum_{(\ga)\in \Xi^{1}_{d}} \int_{[\cW_{(\ga)}]\virt} \frac{1}{\hbar(\hbar-\psi)}   \frac{e(R\pi_{\ga\ast}(\cL_\ga \cL_{\ft} \cL_{\ft'} ))}{e(N_{\cW_{(\ga)}/\cW}^{vir})},$$
 where  $\Xi^1_{n}$ denotes localization graphs with the first marking mapped to $O$ and $\cL_\Gamma$ is the universal line bundle. 
 Then, as $p^4|_{P_i}=0$ for $0\le i\le 4$ and $p|_O=-\ft$,  by the localization formula, one has $U_j=(-\ft)^j \cdot \ft' \cdot \ti\sZ_{1,5}$ whenever $j\geq 4$.  Thus ${\rm Coe}_{p^0}Y=0$ and
 \begin{eqnarray*}
 \ft^4\ft'\ti\sZ_{1,5} &=& U_4= -5(\ft+\ft')+ {\rm Coe}_{p^1}Y-\ft U_3 
 = -5(\ft+\ft')+ {\rm Coe}_{p^1}Y-\ft ({\rm Coe}_{p^2} Y -\ft U_2 -5)\\
                &=& -5 \ft'+ {\rm Coe}_{p^1}Y-\ft {\rm Coe}_{p^2} Y +\ft^2 {\rm Coe}_{p^3}Y-\ft^3 {\rm Coe}_{p^4}Y +\ft^4 {\rm Coe}_{p^5} Y
 \end{eqnarray*}
 implies $-\ft^5\ft' \ti\sZ_{1,5}= 5 \ft\ft' + Y|_{p\mapsto -\ft}.$ 
   Thus we have 
  \begin{theo}\label{package-MSP-e}For MSP moduli $\cW_{g=0,(1,\rho),(d,0)}$,
  $$  \sum_{d=1}^\infty (-\sQ)^{d} \sum_{(\Ga)\in \Xi^1_{d}} \int_{[\cW_{(\ga)}]\virt}  \frac{1}{\hbar(\hbar-\psi_1)}  \frac{ e(\pi_{\ga\ast}(\cL_\Gamma\otimes \cL_{\ft} \otimes\cL_{\ft'}(-D_\ga))  )\black}{e(N_{\cW_\ga/\cW}^{vir})} $$
(here $-D_\ga\sub \cC_{\cW_{\ga}}$ is from the marking) is  equal to 
$$\frac{5}{\ft^4}+\frac{-5}{\ft^4}\frac{1}{I_0}  {\rm exp}(\frac{\ft (T-\ft)}{\hbar}-\frac{\ft'g_1}{I_0\hbar} )\cdot    \sum_{d=0}^\infty q^d \frac{  \prod_{m=1}^{5d}(-5\ft+m\hbar)  \prod_{m=1}^d(\ft'+m\hbar)}{d!\hbar^d\prod_{m=1}^d(-\ft+m\hbar)^5}. $$
\end{theo}

We then have the following package for $c_i$-capped invariants.

 \begin{theo} \label{package-MSP-c} For MSP moduli $\cW_{g=0,(1,\rho),(d,0)}$,    $\forall m\in\NN_{\geq 1}$
  $$  \sum_{d=1}^\infty (-\sQ)^{d} \sum_{(\Ga)\in \Xi^1_{d}} \int_{[\cW_{(\ga)}]\virt}  \frac{1}{\hbar(\hbar-\psi_1)}  \frac{ c_{{\rm top}_\ga-m}(\pi_{\ga\ast}(\cL_\Gamma\otimes \cL_{\ft}))  \black}{e(N_{\cW_\ga/\cW}^{vir})} $$
is the $(\ft')^{m-1}$ coefficient of 
$$\frac{5}{\ft^4}+\frac{-5}{\ft^4}\frac{1}{I_0}  {\rm exp}(\frac{\ft (T-\ft)}{\hbar}-\frac{\ft'g_1}{I_0\hbar} )\cdot    \sum_{d=0}^\infty q^d \frac{  \prod_{m=1}^{5d}(-5\ft+m\hbar)  \prod_{m=1}^d(\ft'+m\hbar)}{d!\hbar^d\prod_{m=1}^d(-\ft+m\hbar)^5}. $$
 \end{theo}

 In the special case $\ft'=-\ft$ of Corollary \ref{closer},  one has the following.

 \begin{coro}\label{closer0} Write $\fA_d=\pi_{d\ast} \ff_d\sta L_H$. It is  a rank $d+1$ bundle over $\cW_{0,(1,\rho),(d,0)}$.
 The series
 \begin{eqnarray*} 5p^2 
-\sum_{d=1}^\infty (-\sQ)^d \sum_{k=0}^{4} \int_{[\cW_{g=0,(1,\rho),(d,0)}]\virt}  \bigg( p^k\frac{e( \fA_d )}{\hbar(\hbar-\psi)} \cup {\rm ev}_1^*[p^{4-k}(p+\ft)]
+ p^5 \frac{e( \fA_d )}{\hbar(\hbar-\psi)}  \bigg) 
\end{eqnarray*}
   is equal to  
   \beq\label{Gosh} \frac{5p^2}{I_0} e^{\frac{1}{\hbar}( \frac{g_1}{I_0} \ft+p(t-T))} 
    \sum_{d=0}^\infty q^d
\frac{ \prod_{m=1}^{5d}(5 p +m\hbar) }
{\prod_{m=1}^d(p+ m\hbar)^4(p+\ft+m\hbar)}. \eeq
 \end{coro}
 

Taking coefficients of $\ft'$ of the formula in Corollary \ref{closer},  we obtain the following identity.
 
 \begin{coro}\label{closer1}
 The series
$$   5p-\sum_{d=1}^\infty (-\sQ)^d  \sum_{k=0}^{4} \int_{[\cW_{g=0,(1,\rho),(d,0)}]\virt} \bigg( p^k\frac{ c_{d}(R_d)}{\hbar(\hbar-\psi)} \cup {\rm ev}_1^*[p^{4-k}(p+\ft)]\\
+ p^5  \frac{c_{d}(R_d)}{\hbar(\hbar-\psi)}  \bigg)  $$
     is equal to the  coefficient of $\ft'$ in $e^{-(T_0+p\log \sQ)/\hbar} I(0,q)$,
   which is  
$$ \frac{  5p (p+\ft)}{I_0} e^{ \frac{p  }{\hbar}(t-T) }
 \bigg[   (\frac{1}{p+\ft}
-\frac{  g_1 }{\hbar I_0} )  \sum_{d=0}^\infty q^d
\frac{ \prod_{m=1}^{5d}(5 p +m\hbar)}
{\prod_{m=1}^d(p+ m\hbar)^5}  + 
  \sum_{d=1}^\infty q^d
\frac{ \prod_{m=1}^{5d}(5 p +m\hbar)  (\sum_{m=1}^d \frac{1}{p+\ft+m\hbar})}
{\prod_{m=1}^d(p+ m\hbar)^5}   \bigg]. 
 $$
 \end{coro}
 
 \subsection{From MSP to Givental I-functions}

Consider 
\begin{eqnarray}\label{bi}
B_d(h):=\frac{\prod_{k=1}^{5d}(-5h+\frac{k(h+\ft)}{d})}{\prod_{k=1}^d (h-\frac{k(h+\ft)}{d})^5}=b_0+b_1\frac{h}{\ft}+b_2 \frac{h^2}{\ft^2}+b_3\frac{h^3}{\ft^3}\in \QQ[\ft,\ft^{-1}][h]/(h^4) 
\end{eqnarray}
\begin{eqnarray*}
A_d(h):=\frac{\prod_{k=1}^{5d}(5h+k)}{\prod_{k=1}^d(h+k)^5}
=a_0+a_1h+a_2h^2+a_3h^3\in \QQ[h]/(h^4)
\end{eqnarray*}
where both $a_i=a_i(d)$,  $b_i=b_i(d) $ are rational numbers. Apply the  map $\theta:\QQ[\ft,\ft^{-1}][h]/(h^4)\to \QQ[h]/(h^4)$ which sends $\ft$ to $(-1-h)$. Since $(-1-h)$ is invertible in $\QQ[h]/(h^4)$, this defines a ring homomorphism. Then
 $\theta(B_d(h))=(-1)^{5d} A_d(dh)$ implies
 $$b_0=(-1)^da_0, \quad b_1=(-1)^{d+1}a_1d, \quad  b_2=(-1)^d[a_2d^2+a_1d], \quad b_3=(-1)^{d+1}[a_3d^3+2a_2d^2+a_1d].$$

Recall the definition of Givental's I function 
 \beq\label{Ifun-a}I_0(t)+I_1(t)h+I_2h^2+ \cdots+I_3h^3=e^{ht}\sum_{d=0}^\infty e^{dt} \frac{\prod_{r=1}^{5d}(5h+r)}{\prod_{r=1}^d(h+r)^5}=e^{ht}\sum_{d=0}^\infty e^{dt} A_d(h) \, ( \mod h^4\  \ ) .\eeq

 
 Recall $D_t:=q\frac{\partial }{\partial q}= \frac{\partial}{\partial t}$ , and denote   $D_t I_k$ as $I_k'$.  The relations in \eqref{Ifun-a} imply
\begin{eqnarray}\label{MSPtoGiv}
1+\sum_{d=1}^\infty  b_0 (-q)^d&=& 1+\sum_{d=1}^\infty  a_0(d) q^d
=I_0, \\
 \sum_{d=1}^\infty  b_1 (-q)^d &=&- \sum_{d=1}^\infty   d a_1(d) q^d
 =-D_tI_1+   D_t (tI_0) ,  \nonumber\\ 
  \sum_{d=1}^\infty  b_2 (-q)^d &=& \sum_{d=1}^\infty [a_2d^2+a_1d]q^d 
=  D_t^2   I_2 -  D_t (tI'_1)+  D_t  (\frac{t^2}{2}I'_0 )    ,\nonumber\\
\sum_{d=1}^\infty  b_3 (-q)^d&=&- \sum_{d=1}^\infty [a_3d^3+2a_2d^2+a_1d]q^d  
 = - D_t^3 I_3  + D_t ( tI_2'' ) -D_t (  \frac{t^2}{2} I_1'')  
+D_t (  \frac{t^3}{6}I_0''  ). \nonumber 
  \end{eqnarray}


\section{$\sK$-series and $\sP$-series}

\subsection{Definitions of $\sK$-series and $\sP$-series}
  Consider the  localization of $\cW_{g=0,\1rho,(n,0)}$.
  Let $\Xi_{n}$ be the set of flat $g=0$ graphs with one leg decorated by $\1rho$ and $(d_0,d_\infty)=(n,0)$, and 
 let $\Xi_{n}^i\sub\Xi_n$, $i=0$ or $1$, be the subset where the marking lies on vertices in $V_i$.

 For   $c>0$ \black and $k = 0,\cdots,5$, we set a formal power  series in $\QQ(\ft)[[\sQ]]$
 \beq\label{sK} \sK_{c,k}(\ft):=   \sum_{d=1}^\infty (-\sQ)^d \sum_{(\Ga)\in \hat\Xi^{0}_{d}} \int_{[\cW_{(\ga)}]\virt} \frac{1}{\psi^c}\ev\sta \mT^k  \frac{c_{{\rm top}_\ga-1}(R\pi_{\ga\ast}(\cL_\Gamma\otimes \cL_{\ft}(-D_\ga)))}{e(N_{\cW_{(\ga)}/\cW}^{vir})} \in \QQ[[\sQ]] \ft^{2-k-c}, \eeq
where $ \hat\Xi^{0}_{d}$ denotes the subset of graphs in $\Xi^{0}_{d}$ whose vertices contain an unstable $\1rho$-leg, $D_\ga\sub\cC_{\cW_\ga}$ is the divisor given by the marking,
and  ${\rm top}_\ga$ is the rank of the bundle $R\pi_{\ga\ast}(\cL_\Gamma\otimes \cL_{\ft}(-D_\ga))$.

 For convenience, we also consider  formal power  series in $\QQ(\ft)[[\ft']][[\sQ]]$
$$ \ti\sK_{c,k}( \ft, \ft'):=\sum_{d=1}^\infty (-\sQ)^d \sum_{(\Ga)\in \hat\Xi^{0}_{d}} \int_{[\cW_{(\ga)}]\virt} \frac{1}{\psi^c}ev\sta \mT^k  \frac{e(\pi_{\ga \ast}\cL_{\ga}\cL_\ft \cL_{\ft'}(-D_\ga)) }{e(N_{\cW_{(\ga)}/\cW}^{vir})},$$
 
 \begin{eqnarray*}  &&\hat\sK_k(w):=\sum_{c=1}^\infty \ti\sK_{c,k} (-w)^{c-1}\\
 &= &  \sum_{d=1}^\infty (-\sQ)^d \sum_{(\ga)\in \hat\Xi^{0}_{d}} \int_{[\cW_{(\ga)}]\virt} \frac{1}{\psi_{\ga}+w}\ev\sta \mT^k  \frac{e(R\pi_{\ga\ast}(\cL_\Gamma\otimes \cL_{\ft}\otimes \cL_{\ft'}(-D_\ga)))}{e(N_{\cW_{(\ga)}/\cW}^{vir})},
\end{eqnarray*}
which is an expansion near $w=0$. Note that $$e(\pi_{\ga \ast}\cL_{\ga}\cL_\ft \cL_{\ft'}(-D_\ga))=\sum_{\ell=0}^\infty c_{{\rm top}_\ga-\ell}(R\pi_{\ga\ast}(\cL_\Gamma\otimes \cL_{\ft}(-D_\ga) ))(\ft')^\ell.$$

 By definition, $\sK_{c,k}$ is the coefficient of $ (\ft' ) ^1$ in $\ti\sK_{c,k}$. Observe that   the
 constant term of $\ft'$-expansion of  $\ti\sK_{c,k}$ vanishes because $\nu_1$ makes $e\bl R \pi_{\ga\ast}(\cL_\ga\cL_\ft(-D_\ga)) \br $ vanishing.
 We call this  \textit{cap geometry}, and write the expansion  
\beq\label{tic}  \ti\sK_{c,k}= \sK_{c,k}\ft' + \sK_{c,k,1} (\ft')^2+ \cdots+ \sK_{c,k,\ell}(\ft')^{\ell+1}+\cdots.\eeq


  We  define  $\sP$-series as follows. For   $c>0$,  \black  we set  a formal power  series in $\QQ(\ft)[[\sQ]]$
 \begin{align*} \sP_{c}(\ft):=   \sum_{d=1}^\infty (-\sQ)^d \sum_{(\ga)\in \hat\Xi^{1}_{d}} \int_{[\cW_{(\ga)}]\virt} \frac{1}{\psi^c}  \frac{c_{{\rm top}_\ga}(R\pi_{\ga\ast}(\cL_\Gamma\otimes \cL_{\ft}(-D_\ga)))}{e(N_{\cW_{(\ga)}/\cW}^{vir})} , \end{align*}

 $$ \ti\sP_{c}( \ft.\ft'):=\sum_{d=1}^\infty (-\sQ)^d \sum_{(\ga)\in \hat\Xi^{1}_{d}} \int_{[\cW_{(\ga)}]\virt} \frac{1}{\psi^c}  \frac{e(\pi_{\ga \ast}\cL_{\ga}\cL_\ft \cL_{\ft'}(-D_\ga)) }{e(N_{\cW_{(\ga)}/\cW}^{vir})} .$$
 Let $ \hat\Xi^{1}_{d}$ be  the subset of graphs in $\Xi^{1}_{d}$ whose vertices contain the unstable $\1rho$-leg.  We also define  \beq\label{hatP} \hat \sP (w):= \sum_{d=1}^\infty (-\sQ)^d \sum_{(\ga)\in \hat\Xi^{1}_{d}} \int_{[\cW_{(\ga)}]\virt} \frac{1}{ w + \psi_\ga }  \frac{e(\pi_{\ga \ast}\cL_{\ga}\cL_\ft \cL_{\ft'}(-D_\ga)) }{e(N_{\cW_{(\ga)}/\cW}^{vir})} =  \ti\sP_1 - w \ti\sP_2 +w^2 \ti\sP_3 - \cdots, \eeq
 $$\ti\sZ_{j,k}:=  - \sum_{d=1}^\infty (-\sQ)^d \sum_{(\ga)\in \Xi^{j}_{d}} \int_{[\cW_{(\ga)}]\virt} \frac{1}{\hbar(\hbar-\psi)}\ev\sta \mT^k  \frac{e(R\pi_{\ga\ast}(\cL_\ga \cL_{\ft} \cL_{\ft'}(-D_\ga)))}{e(N_{\cW_{(\ga)}/\cW}^{vir})},  \qquad \text{for}\ j=0,1,$$




  \beq\label{Zjk}\sZ_{j,k}:=  - \sum_{d=1}^\infty (-\sQ)^d \sum_{(\ga)\in \Xi^{j}_{d}} \int_{[\cW_{(\ga)}]\virt} \frac{1}{\hbar(\hbar-\psi)}\ev\sta \mT^k  \frac{c_{{\rm top}-1}(R\pi_{\ga\ast}(\cL_\ga \cL_{\ft} (-D_\ga)))}{e(N_{\cW_{(\ga)}/\cW}^{vir})} ,   \qquad \text{for}\ j=0,1.\eeq

\begin{lemm}\label{res} For $c\geq 1$, we have
$$\ti\sK_{c,k}=  {\rm Res}_{\hbar=0}  \frac{\ti\sZ_{0,k}}{\hbar^{c-1}}, \quad \sK_{c,k}=  {\rm Res}_{\hbar=0}  \frac{\sZ_{0,k}}{\hbar^{c-1}},\quad 
 \ti \sP_c=  {\rm Res}_{\hbar=0}  \frac{\ti\sZ_{1,5}}{\hbar^{c-1}} ,   \quad \sP_c =   {\rm Res}_{\hbar=0}  \frac{\sZ_{1,5}}{\hbar^{c-1}}.$$
\end{lemm}
\begin{proof}
 We prove the first two identities. In the case  $(\ga)\in\hat\Xi^0_d$, its $\psi$ is   invertible and the expansion near $\hbar=0$ gives
 $\frac{1}{\hbar-\psi}= \frac{-1}{\psi} (1+\frac{\hbar}{\psi} +\frac{\hbar^2}{\psi^2}+\cdots ).$
And in the case $(\ga)\in \Xi^0_d-\hat\Xi^0_d$, its $\psi^r$ will vanish on $[\cW_\ga]\virt$ for some $r$. Thus we can write 
$\frac{1}{\hbar-\psi}=\frac{1}{\hbar} (1+ \frac{\psi}{\hbar}+\dots+\frac{\psi^r}{\hbar^r})$, $\forall \hbar\neq 0$.  Therefore  such $\ga$ doesn't  contribute to the residue.
\end{proof}

\subsection{Calculations of $\ti\sZ_{1, k}$}
 For $k<5$, one has $\mT^k=p^{4-k}(p+\ft)$ and thus $\mT^k|_O=0$ implies $\ti\sZ_{1,k}=0$. 
 Thus 
 \begin{eqnarray*}
  && - \sum_{d=1}^\infty (-\sQ)^d \sum_{(\ga)\in \Xi^{1}_{d}} \int_{[\cW_{(\ga)}]\virt} \frac{1}{\hbar(\hbar-\psi)}   \frac{e(R\pi_{\ga\ast}(\cL_\ga \cL_{\ft} \cL_{\ft'}(-D_\ga)))}{e(N_{\cW_{(\ga)}/\cW}^{vir})}\ev^*p^k=(-\ft)^k\cdot \tilde \sZ_{1, 5}.\nonumber
   \end{eqnarray*}
  When $k=5$, by Theorem \ref{package-MSP-e}, we have  \begin{eqnarray}\label{Z1,5}
 \ti\sZ_{1,5} &= &   - \sum_{d=1}^\infty (-\sQ)^d \sum_{(\ga)\in \Xi^{1}_{d}} \int_{[\cW_{(\ga)}]\virt} \frac{1}{\hbar(\hbar-\psi)}   \frac{e(R\pi_{\ga\ast}(\cL_\ga \cL_{\ft} \cL_{\ft'}(-D_\ga)))}{e(N_{\cW_{(\ga)}/\cW}^{vir})} \nonumber\\
&=&-\frac{5}{\ft^4}+\frac{5}{\ft^4 I_0}   {\rm exp}(\frac{\ft (T-t)}{\hbar}-\frac{\ft'g_1}{I_0\hbar} )\cdot    \sum_{d=0}^\infty q^d \frac{  \prod_{m=1}^{5d}(-5\ft+m\hbar)  \prod_{m=1}^d(\ft'+m\hbar)}{d!\hbar^d\prod_{m=1}^d(-\ft+m\hbar)^5}.
 \end{eqnarray}

We write  $$ w:=-\frac{\ft}{\hbar}   , \  \sR(w,t):=\sum_{d=0}^\infty q^d  \frac{\prod_{k=1}^{5d}(5w+k)}{\prod_{k=1}^d(w+k)^5}
 \in \QQ[[q,w]], \    \Lambda:=  \frac{1}{I_0} e^{(t-T)w} \sR(w,t).$$

 \begin{coro}\label{coro0}

  \beq\label{Z15} \sZ_{1,5}={\rm Coe}_{(\ft')^0}  \ti\sZ_{1,5}= -\frac{5}{\ft^4}+\frac{5}{\ft^4} \Lambda. \eeq
  
  \begin{eqnarray}
  &&  \sum_{d=1}^\infty (-\sQ)^d \sum_{(\ga)\in \Xi^{1}_{d}} \int_{[\cW_{(\ga)}]\virt} \frac{1}{\hbar(\hbar-\psi)}   \frac{c_d(R\pi_{\ga\ast}(\cL_\ga \cL_{\ft} (-D_\ga)))}{e(N_{\cW_{(\ga)}/\cW}^{vir})}\ev^*p^k= 5(-\ft)^{k-4}\big(1-\Lambda\big).\nonumber
   \end{eqnarray}
  \end{coro}

\subsection{Calculation of  $\sZ_{0,k}$}

Assume $k<5$, then $\mT^k=p^{4-k}(p+\ft)$. We have for each $(\ga)\in\Xi^0_d$, 
  $$\pi_{\ga\ast}(\cL_\ga\cL_\ft)=\pi_{\ga\ast}(\cL_\ga\cL_\ft(-D_\ga)) + \pi_{\ga\ast} (\cL_\ga\cL_\ft|_{D_\ga})=\pi_{\ga\ast}(\cL_\ga\cL_\ft(-D_\ga)) + \ev\sta L_HL_\ft,$$
  as classes in $K$-groups, and thus
  $$c_d(\pi_{\ga\ast}(\cL_\ga\cL_\ft))=e\bl \pi_{\ga\ast}(\cL_\ga\cL_\ft(-D_\ga)) \br+
  c_{d-1} \bl \pi_{\ga\ast}(\cL_\ga\cL_\ft(-D_\ga)) \br \cdot   c_1 ( \ev\sta L_HL_\ft ) ,$$
 \begin{eqnarray} \label{z0k}
\sZ_{0,k} &=&  - \sum_{d=1}^\infty (-\sQ)^d    \sum_{(\ga)\in \Xi^{0}_{d}} \int_{[\cW_{(\ga)}]\virt} \frac{1}{\hbar(\hbar-\psi)}\ev\sta p^{4-k}(p+\ft )  \frac{c_{d-1}(R\pi_{\ga\ast}(\cL_\ga \cL_{\ft} (-D_\ga)))}{e_G(N_{\cW_{(\ga)}/\cW}^{vir})} \nonumber \\
 &=&  - \sum_{d=1}^\infty (-\sQ)^d    \sum_{(\ga)\in \Xi^{0}_{d}} \int_{[\cW_{(\ga)}]\virt} \frac{1}{\hbar(\hbar-\psi)}\ev\sta p^{4-k}   \frac{c_d(\pi_{\ga\ast}(\cL_\ga \cL_{\ft}))   }{e_G(N_{\cW_{(\ga)}/\cW}^{vir})}  \\
 &&  + \sum_{d=1}^\infty (-\sQ)^d     \sum_{(\ga)\in \Xi^{0}_{d}} \int_{[\cW_{(\ga)}]\virt} \frac{1}{\hbar(\hbar-\psi)}\ev\sta p^{4-k}   \frac{ e\bl \pi_{\ga\ast}(\cL_\ga\cL_\ft(-D_\ga)) \br }{e_G(N_{\cW_{(\ga)}/\cW}^{vir})}\nonumber
                   \end{eqnarray}

           The two integrals in \eqref{z0k} can be separately reduced as follows.

\noindent{\bf [i].}
     The first integral in \eqref{z0k} can be calculated as 
           \begin{eqnarray}\label{z0kf1} &&- \sum_{d=1}^\infty (-\sQ)^d    \sum_{(\ga)\in \Xi^{0}_{d}} \int_{[\cW_{(\ga)}]\virt} \frac{1}{\hbar(\hbar-\psi)}ev\sta p^{4-k}   \frac{c_d(\pi_{\ga\ast}(\cL_\ga \cL_{\ft}))   }{e_G(N_{\cW_{(\ga)}/\cW}^{vir})}\nonumber \\
             &=&- \sum_{d=1}^\infty (-\sQ)^d  {\rm Coe}_{\ft'}   \int_{[\cW_{0,\1rho,(d,0)}]\virt}  \frac{1}{\hbar(\hbar-\psi)}\ev\sta p^{4-k}   \cdot e(\pi_{\cW\ast}(\cL_\cW \cL_{\ft}\cL_{\ft'}))  \\
             &+&  \sum_{d=1}^\infty (-\sQ)^d     \sum_{(\ga)\in   \Xi^{1}_d} \int_{[\cW_{(\ga)}]\virt} \frac{1}{\hbar(\hbar-\psi)}\ev\sta p^{4-k}   \frac{c_d(\pi_{\ga\ast}(\cL_\ga \cL_{\ft} ))   }{e_G(N_{\cW_{(\ga)}/\cW}^{vir})},\nonumber
           \end{eqnarray}
           where $\pi_\cW$ is the projection from the universal family  $\sC_\cW$ to $\sW$ and $\cL_\cW$ is the pull-back of  $\sO_{\mathbb P^5}(1)$ via the universal family map 
           $\sC_\cW \to \mathbb P^5$.    Denote the first term in \eqref{z0kf1}  as           \begin{eqnarray*}    \fB_k     \black  :=  -  {\rm Coe}_{\ft'} \sum_{d=1}^\infty (-\sQ)^d     \int_{[\cW_{0,\1rho,(d,0)}]\virt}  \frac{1}{\hbar(\hbar-\psi)}\ev\sta p^{4-k}   \cdot e(\pi_{\cW\ast}(\cL_\cW \cL_{\ft}\cL_{\ft'})).      \end{eqnarray*}
           The second term in \eqref{z0kf1} is, by Corollary \ref{coro0},
           \begin{eqnarray*}
          &&\sum_{d=1}^\infty (-\sQ)^d     \sum_{(\ga)\in   \Xi^{1}_d} \int_{[\cW_{(\ga)}]\virt} \frac{1}{\hbar(\hbar-\psi)}\ev\sta p^{4-k}   \frac{c_d(\pi_{\ga\ast}(\cL_\ga \cL_{\ft} ))   }{e_G(N_{\cW_{(\ga)}/\cW}^{vir})} \\
&= &5(-\ft)^{-k} \Bigg(1 -  \frac{1}{I_0} e^{\frac{\ft}{\hbar}\cdot (T-t)}  \sum_{d=0}^\infty q^d \frac{  \prod_{m=1}^{5d}(-5\ft+m\hbar)  }{ \prod_{m=1}^d(-\ft+m\hbar)^5}  \Bigg).
           \end{eqnarray*}



\noindent {\bf [ii].}   The second integral in \eqref{z0k}, denoted by $\fR_k$,  is only contributed by those $\ga\in \Xi^0_d$ which have only one vertex $v$ \black in $V_0$ (denote such graph as $\ga_v$),  otherwise,  $\nu_1$ makes $ e\bl \pi_{\ga\ast}(\cL_\ga\cL_\ft(-D_\ga)) \br $ vanish.                  More precisely     \begin{eqnarray*}
   \fR_k 
&=&\sum_{d=1}^\infty (-\sQ)^d      \int_{(-1)^{d+1}[\barM_{0,1}(Q,d)]\virt}    \frac{1}{\hbar^2} (1+ \frac{\psi}{\hbar})\cdot \ev\sta h^{4-k}   \cdot    \frac{1}{ e\bl \pi_{\ga\ast}(\cL_\ga\cL_\ft|_{D_\ga}) \br  } \\
&=& - \sum_{d=1}^\infty \sQ^d      \int_{[\barM_{0,1}(Q,d)]\virt}    \frac{1}{\hbar^2} (1+ \frac{\psi}{\hbar}) \cdot \ev\sta h^{4-k}   \cdot   \ev\sta (1-\frac{h}{\ft})  \frac{1}{\ft},
         \end{eqnarray*}
         where $h$ is the hyperplane class in $\mathbb P^4$. 
         We calculate $ \fR_k$ for different $k$'s.

                  \begin{eqnarray*}
      \fR_4      &=&         - \sum_{d=1}^\infty \sQ^d      \int_{[\barM_{0,1}(Q,d)]\virt}    \frac{1}{\hbar^2}    \cdot   \ev\sta (-\frac{h}{\ft})  \frac{1}{\ft}
               - \sum_{d=1}^\infty \sQ^d      \int_{[\barM_{0,1}(Q,d)]\virt}    \frac{1}{\hbar^2}  \frac{\psi}{\hbar}  \cdot     \frac{1}{\ft} \\
               &=& \sum_{d=1}^\infty \sQ^d \bl  \frac{1}{\hbar^2\ft^2}  dN_{0,d} -\frac{1}{\hbar^3\ft} (-2) N_{0,d} \br.
                  \end{eqnarray*}

                  Let $k<4$. The integral $ \fR_k$ becomes
                  \begin{eqnarray*}
                  \begin{aligned}
                  &  - \sum_{d=1}^\infty \sQ^d      \int_{[\barM_{0,1}(Q,d)]\virt}    \frac{1}{\hbar^2}   \ev\sta  h^{4-k}   \cdot     \frac{1}{\ft}  
            =    \begin{cases}
			\displaystyle{  - \sum_{d=1}^\infty \sQ^d  \frac{1}{\hbar^2\ft} d N_{0,d}   },
			& k=3; \\
			\displaystyle{0 }
				& k=0,1,2.
		\end{cases}
		\end{aligned}
		\end{eqnarray*}

                    Therefore we proved the following.

                    \begin{lemm}\label{main} For $k=0,1,\cdots,4$,  we have
                     \begin{eqnarray*} \sZ_{0,k}= \fB_k +   5(-\ft)^{-k} \Big(1- \Lambda  \Big)   + \fR_k,
                     \end{eqnarray*}
 where $\fR_k=0$ for $k=0,1,2$ and
 $$\fR_3=   - \sum_{d=1}^\infty \sQ^d  \frac{1}{\hbar^2\ft} d N_{0,d} 
   , \qquad \qquad \fR_4=\sum_{d=1}^\infty \sQ^d \bl  \frac{1}{\hbar^2\ft^2}  dN_{0,d} -\frac{1}{\hbar^3\ft} (-2) N_{0,d} \br. $$
                     \end{lemm}
                  We now calculate $\fB_k$ for $k=0,1,2$.
                       \begin{eqnarray*}   \fB_0     &=&  - \frac{1}{\ft} {\rm Coe}_{\ft'} \sum_{d=1}^\infty (-\sQ)^d     \int_{[\cW_{0,\1rho,(d,0)}]\virt}  \frac{1}{\hbar(\hbar-\psi)}\ev\sta p^{4} (p+\ft)  \cdot e(\pi_{\cW\ast}(\cL_\cW \cL_{\ft}\cL_{\ft'}))  \\
                      &+&  \frac{1}{\ft}  {\rm Coe}_{\ft'} \sum_{d=1}^\infty (-\sQ)^d     \int_{[\cW_{0,\1rho,(d,0)}]\virt}  \frac{1}{\hbar(\hbar-\psi)}\ev\sta p^5  \cdot e(\pi_{\cW\ast}(\cL_\cW \cL_{\ft}\cL_{\ft'})) \\
                       &=&     \frac{-\ft^5}{\ft}   \sum_{d=1}^\infty (-\sQ)^d   \sum_{(\ga)\in\Xi_d^1}  \int_{[\cW_{(\ga)}]\virt}  \frac{1}{\hbar(\hbar-\psi)}   \cdot \frac{c_{{\rm top-}1}(\pi_{\ga\ast}(\cL_\ga \cL_{\ft} ))}{e(N_{\cW_\ga/\cW}\virt)} \\
                       &=& -\ft^4 \cdot \bigg(   \frac{5}{\ft^4}+ \frac{-5}{\ft^4}\frac{1}{I_0}  {\rm exp}(\frac{\ft (T-t)}{\hbar})\cdot    \sum_{d=0}^\infty q^d \frac{  \prod_{m=1}^{5d}(-5\ft+m\hbar)   }{ \prod_{m=1}^d(-\ft+m\hbar)^5}\bigg)
                 =5\Lambda(w,t)-5,      
                           \end{eqnarray*}
                    where the third  identity is by $G$-localization and  the last line comes from Theorem \ref{package-MSP-c}.

                    Taking ${\rm Coe}_{p}$ of the Corollary \ref{closer1},     one has   
                    \begin{eqnarray*}   \fB_1       &=&  -  \frac{1}{\ft}    \sum_{d=1}^\infty (-\sQ)^d     \int_{[\cW_{0,\1rho,(d,0)}]\virt}  \frac{1}{\hbar(\hbar-\psi)}\ev\sta p^{3}(p+\ft)   \cdot c_{{\rm top}-1}(\pi_{\cW\ast}(\cL_\cW \cL_{\ft}))    \\
           &+&        \frac{1}{\ft}    \sum_{d=1}^\infty (-\sQ)^d     \int_{[\cW_{0,\1rho,(d,0)}]\virt}  \frac{1}{\hbar(\hbar-\psi)}\ev\sta p^4   \cdot c_{{\rm top-}1}(\pi_{\cW\ast}(\cL_\cW \cL_{\ft}))    \\
           &=&   - \frac{5}{\ft}\Lambda(w,t)   +         \frac{ - 5w }{\ft I_0}   \bigg[ (-\frac{I_0}{w}
- g_1)      +   \sum_{d=1}^\infty q^d \frac{ (5d)!   }{ (d!)^5}(\sum_{m=1}^d \frac{1}{m -w })   \bigg]   . 
                    \end{eqnarray*}

              Taking  ${\rm Coe}_{p^2}$ of  Corollary \ref{closer1}, we get

                            \begin{eqnarray*}   \fB_2   &=& -  \frac{1}{\ft} \sum_{d=1}^\infty (-\sQ)^d     \int_{[\cW_{0,\1rho,(d,0)}]\virt}  \frac{1}{\hbar(\hbar-\psi)}\ev\sta p^{2}(p+\ft)   \cdot c_{{\rm top}-1}(\pi_{\cW\ast}(\cL_\cW \cL_{\ft}))  \\
                           && +\frac{1}{\ft} \sum_{d=1}^\infty (-\sQ)^d     \int_{[\cW_{0,\1rho,(d,0)}]\virt}  \frac{1}{\hbar(\hbar-\psi)}\ev\sta p^{3}   \cdot c_{{\rm top}-1}(\pi_{\cW\ast}(\cL_\cW \cL_{\ft}))  \\
       &=&   -\frac{1}{\ft}\fB_1  -\frac{5}{I_0}\cdot         \frac{ g_1}{\hbar \ft}   \\
       && +  \frac{5 }{I_0}
          \sum_{d=1}^\infty q^d \frac{(5d)! }{(d!)^{5} }
   \bigg(  \big(  \frac{1}{\ft}-\frac{T-t}{\hbar} +\sum_{m=d+1}^{5d}\frac{5}{m\hbar} \big) \sum_{m=1}^d \frac{1}{\ft+m\hbar}   -    \sum_{m=1}^d \frac{ 1}{(\ft+m\hbar)^2}  \bigg).     \\
 \end{eqnarray*}

 Here we omit    detailed calculations obtaining the second equality above.

                    \subsection{Formula of $\sK_{1,2}$ and $\sK_{2,1}$}

    For $k=1,2$ ,    we can now apply $  \sK_{c,k}=  {\rm Res}_{\hbar=0}  \frac{\sZ_{0,k}}{\hbar^{c-1}} $.

   Now we look at the case $k=1$.
The residue at $\hbar=0$ of $5\ft^{-1} \big(1- \Lambda  \big)   + \fR_1$ divided by $\hbar$ is 
 \begin{eqnarray*}    \frac{5}{\ft}\bigg( 1- {\rm Res}_{\hbar=0} \frac{\Lambda}{\hbar} d\hbar \bigg)
   =      \frac{5}{\ft}\bigg( 1- {\rm Res}_{w=\infty} \Lambda \cdot  (-\ft)\cdot  \frac{-dw}{w^2}\cdot \frac{-w}{\ft} \bigg) = \frac{5}{\ft}\bigg( 1+ {\rm Res}_{w=\infty} \Lambda \frac{dw}{w}   \bigg).
 \end{eqnarray*}   
 
From
\begin{eqnarray*}&&\fB_1
 =   - \frac{5}{\ft}\Lambda(w,t)+ \frac{ - 5w }{\ft I_0}   \bigg[
- g_1(q)      +   \sum_{d=1}^\infty q^d
\frac{ (5d)!   } { (d!)^5}(\sum_{m=1}^d \frac{1}{m -w })   \bigg]   +\frac{5}{\ft} ,
\end{eqnarray*}
one calculates ${\rm Res}_{w=\infty} \frac{dw}{m-w}=1$ for all $m\in\ZZ_{\geq 0}$ and thus
    $${\rm Res}_{w=\infty}  \fB_1 \frac{dw}{w}=  - \frac{5}{\ft} \bl {\rm Res}_{w=\infty}  \Lambda(w,t) \frac{dw}{w} + \frac{1}{I_0}\frac{d I_0}{dt} \br -\frac{5}{\ft}.$$
 Thus we have
   \begin{eqnarray*} &&  \sK_{2,1}=     {\rm Res}_{\hbar=0}  \frac{\fB_1}{\hbar}- \frac{5}{\ft}\bigg( 1+ {\rm Res}_{w=\infty} \Lambda \frac{dw}{w}   \bigg)\nonumber 
    = \frac{5}{\ft}  \frac{1}{I_0}\frac{d I_0}{dt}.  
   \end{eqnarray*}
 
 For later purpose, we also use ${\rm Res}_{w=\infty} \frac{1}{m-w}\frac{dw}{w}=0$ for all $m\in\ZZ_{\geq 0}$ to calculate
 \begin{eqnarray*}
 && {\rm Res}_{\hbar=0} \fB_1={\rm Res}_{\hbar=0} \fB_1 d\hbar   =\ft {\rm Res}_{w=\infty} \fB_1 \frac{dw}{w^2}\\
 &=&-  5   {\rm Res}_{w=\infty}\Lambda \frac{dw}{w^2} -\frac{5g_1}{  I_0}
 - \frac{   5  }{  I_0}  \sum_{d=0}^\infty q^d
\frac{ (5d)!   } { (d!)^5}(\sum_{m=0}^d  {\rm Res}_{w=\infty}  \frac{1}{m -w }\frac{dw}{w })  \\
& =& -  5  {\rm Res}_{w=\infty} \Lambda \frac{dw}{w^2} -\frac{5g_1}{  I_0} . \end{eqnarray*}

   \black

 We now calculate $\sK_{1,2}$.  Using $d\hbar=\ft\frac{dw}{w^2}$,
  \begin{eqnarray*}
&& {\rm Res}_{\hbar=0} \fB_2\\
 &=&
  \frac{1}{\ft}(- 5  {\rm Res}_{w=\infty} \Lambda \frac{dw}{w^2} +\frac{5g_1}{  I_0}) - \frac{5}{I_0}\cdot   \frac{ g_1}{  \ft}      +  \frac{5 }{I_0}        \sum_{d=1}^\infty q^d \frac{(5d)! }{(d!)^{5} }
   \bigg(  \frac{(t-T)d}{\ft} +  \sum_{m=d+1}^{5d} \frac{5d}{m\ft}   \bigg) \\
   &=&
   \frac{5}{\ft}       {\rm Res}_{w=\infty} \Lambda \frac{dw}{w^2}       +  \frac{5  }{I_0\ft}        \sum_{d=1}^\infty d q^d \frac{(5d)! }{(d!)^{5} }
   \bigg(   -(T-t)+ \sum_{m=d+1}^{5d} \frac{5 }{m }   \bigg). \\
 \end{eqnarray*}

 By Lemma \ref{main},  $ \sZ_{0,2}=\fB_2 +   5\ft^{-2 }  (1- \Lambda  )$. Thus $\sK_{1,2}=    {\rm Res}_{\hbar=0}  \sZ_{0,2} $ gives
\begin{eqnarray} \sK_{1,2} 
 &=&  {\rm Res}_{\hbar=0} \fB_2+\frac{-5}{\ft}\bigg(  {\rm Res}_{w=\infty} \Lambda \cdot     \frac{dw}{w^2} \bigg) \nonumber= \frac{5  }{\ft I_0 }        \sum_{d=1}^\infty d q^d \frac{(5d)! }{(d!)^{5} }
   \bigg(    t-T+ \sum_{m=d+1}^{5d} \frac{5 }{m }   \bigg) \nonumber \\
   &=&\frac{5  }{\ft I_0 } \big[ q\frac{d}{dq}(I_1- t I_0)  - (T-t) \cdot q\frac{d}{dq}I_0 \big].
 \end{eqnarray}
  \begin{lemm}
  \beq\label{GOD}   \sK_{1,2}= \frac{5}{\ft} q\frac{d}{dq}(T-t)=\frac{5}{\ft} \frac{d}{dt}(T-t) \and  \sK_{2,1}=   \frac{5}{\ft}  \frac{1}{I_0}\frac{d I_0}{dt}.  \eeq
\end{lemm}

 \subsection{Relations obtained by substituting $\ft'$ by $-\ft$}

   The trick used in previous calulation of $\sZ_{0,k}$ may be used to calculate $\ti\sZ_{0,k}(\ft,-\ft)$. For each $(\ga)\in\Xi^0_d$, we have 
$$\pi_{\ga\ast}(\cL_\ga )=\pi_{\ga\ast}(\cL_\ga (-D_\ga)) + \pi_{\ga\ast} (\cL_\ga |_{D_\ga})=\pi_{\ga\ast}(\cL_\ga (-D_\ga)) + \ev\sta L_H ,$$
 and thus
  $$e(\pi_{\ga\ast}(\cL_\ga ))=
  e \bl \pi_{\ga\ast}(\cL_\ga (-D_\ga)) \br \cdot \ev\sta H . $$
 
  Suppose $k<4$ and thus $\mT^k=p^{4-k}(p+\ft)$. Set
 \begin{eqnarray*}
 \ti\sZ_{0,k}(\ft,-\ft) &:=&  - \sum_{d=1}^\infty (-\sQ)^d \sum_{(\ga)\in \Xi^{0}_{d}} \int_{[\cW_{(\ga)}]\virt} \frac{1}{\hbar(\hbar-\psi)}ev\sta \mT^k  \frac{e(R\pi_{\ga\ast}(\cL_\ga  (-D_\ga)))}{e(N_{\cW_{(\ga)}/\cW}^{vir})} \\
  &=&  - \sum_{d=1}^\infty (-\sQ)^d \sum_{(\ga)\in \Xi^{0}_{d}\cup \Xi^{1}_{d}} \int_{[\cW_{(\ga)}]\virt} \frac{1}{\hbar(\hbar-\psi)}ev\sta \mT^{k+1}  \frac{e(R\pi_{\ga\ast} \cL_\ga )}{e(N_{\cW_{(\ga)}/\cW}^{vir})} \\
    &=&  - \sum_{d=1}^\infty (-\sQ)^d   \int_{[\cW_{0,\1rho,(d,0)}]\virt} \frac{1}{\hbar(\hbar-\psi)}ev\sta \mT^{k+1}   e(R\pi_{\ga\ast} \cL_\ga ). 
        \end{eqnarray*}
 In the case $k=1$, $\mT^2=p^2(p+\ft)$.  Taking coefficient of $p^2$ in \eqref{Gosh},  Corollary \ref{closer0} implies
  \begin{eqnarray}\label{ghost0}
 \ti\sZ_{0,1}(\ft,-\ft) 
&=&    -5+   \bigg[ \frac{5 }{I_0} e^{\frac{\ft}{\hbar}( \frac{g_1}{I_0}  )}
    \sum_{d=0}^\infty q^d
\frac{(5d)! \hbar^{d} }
{(d!)^4 \prod_{m=1}^d (\ft+m\hbar)}  \bigg]. 
    \end{eqnarray}

 \section{Graph Type A}
 
  In genus one cap-MSP,  we apply localization formula to
   \beq\label{zero0} 0=\int_{[\cW_{g=1,\emptyset,(n,0)}]\virt} \ft \ c_{n-1}(R\pi_{\cW\ast}\cL_\cW \cN_\cW \cL_\ft)\eeq for all $n$.
 We say a graph $\Ga$  is of type $A$ if for some $v\in V_0(\ga)$ we have  $g_v=1$.

  \subsection{Packaging graph $A_+$ and $A_{\rm single}$ using $\sK$ series}

  
  If $\ga$ is of type $A$ with $d_v>0$, one can assume $\ga$ has a single vertex at $V_0$ without edges, or $\ga$ is as Figure 1 below. 
   All other type of graphs have zero contribution to \eqref{zero0}. This is because if $\ga$ has more than one vertices at $V_1$, 
   one removes $v$ from $\ga$ to obtain connected graphs   $\ga_i$'s for $i=1,\cdots,\ell$, with $\ell>1$. This expresses $\cC_{\cW_\ga}$ as a union of $\cC_{\cW_v}$ and a family of rational curves   $\cC_{\cW_{\ga_i}}$.
      Then $\{\nu_1|_{\cC_{\ga_i}}\}_{i=1}^\ell$ are nowhere zero sections  of the locally free sheaves $\{R\pi_{\cW_{\ga_i}\ast}( \cL_\cW\cN_\cW\cL_\ft)|_{\cC_{\cW_{\ga_i}}}\}_{i=1}^\ell$. Since
$$        c(R\pi_{\cW_v\ast}(\cL_\cW \cN_\cW \cL_\ft))|_{\cC_{\cW_v}}\cap  [\cW_v]\virt  \cong  c(\CC^{d_v}\ft) \cap  [\cW_v]\virt $$
 are chern classes of a bundle,    the fact $n=\rank R\pi_{\cW\ast}\cL_\cW \cN_\cW \cL_\ft$ and the  condition     $\ell>1$  imply $c_{n-1}(R\pi_{\cW\ast}\cL_\cW\cN_\cW\cL_\ft)$ vanishes on $\cW_\ga$.

\begin{figure}[h]
 \begin{center}
\begin{picture}(-20,12)

\put(-36,14){\circle*{1}}

 \put(-36,13){\line(0,-1){10}}
 \put(-35,7){$e$}

\put(-23,14){\circle*{1}}
 \put(-23,13){\line(0,-1){10}}
 \put(-22,7){$e_1$}
\put(-24,0){$v_1$}

\put(-37, 0){$v$}

\put(-11,14){\circle*{1}}
\put(-6,9){$\cdots$}
 \put(-11,13){\line(0,-1){10}}
\put(-10,7){$e_2$}
\put(-12,0){$v_2$}

\put(2,14){\circle*{1}}
 \put(2,13){\line(0,-1){10}}
\put(3,7){$e_k$}
\put(1,0){$v_k$}

 \put(-35,14){\line(1,0){40}}

\put(7, 13){$ v_0$}
\put(-50, 12){$1$}


\put(-50,2){$0$}

\end{picture}
\end{center}

\caption{The labelled graph $\Gamma_A(d_v,d_e,d_1,\ldots, d_k)$: $g_v=1$; $g_{v_0}=0$;
for $i=1,\ldots,k$, $d_i:=d_{e_i}$; $v_i\in V_0$,  $g_{v_i}=0$ and $d_{v_i}\in\ZZ_{\geq 0}$.}
\end{figure}
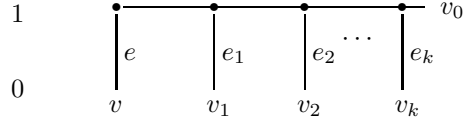



  
  
 In this subsection  we calculate contributions of such graphs with $d_v>0$. Let $A_{\rm single}$ be the graph type that no edges are attached to $v$, then ${\rm Contri}(A_{\rm single})$ is
\begin{eqnarray}\label{a0single} 
 &&\sum_{d_v=1}^\infty \sQ^{d_v}\int_{[\barM_1(Q,d_v)]\virt} 
\frac{ \ft c_{d_v-1}(R\pi_{v\ast}\cL_v  \cL_\ft)}{ e_T(R\pi_{v\ast}\cL_v  \cL_\ft)}\nonumber= \sum_{d_v=1}^\infty \sQ^{d_v}\int_{[\barM_1(Q,d_v)]\virt} \frac{\ft (d_v\ft^{d_v-1})}{\ft^{d_v}}\\
&= &\sum_{d_v=1}^\infty N_{1,d_v}d_v\sQ^{d_v}
= \frac{d}{dT}\ti F_1(\sQ). 
\end{eqnarray}

Let $A_+$ be  the graph type with $d_v>0$ and $\sum_j (d_{e_j}+d_{v_j})>0$. For each  such $\ga$ we replace $v$ by a $(1,\rho)$-marking to obtain $\ga'\in \hat\Xi^0_m$ (see the sentences below \eqref{sK}) for $m=d_\ga-d_v$. Using the Poincar\'e dual of the quintic diagonal $PD:=5^{-1}\sum_{j=0}^3   h^j \otimes h^{3-j}  
$ and $\deg(\psi_{(e,v)}\cap \barM_{1,1}(Q,d_v))=0$,
   the contribution we want   is   \black
\begin{eqnarray*}\label{A+f}
 &&  \sum_{d_v,m=1}^\infty  (-\sQ)^{d_v+m}   \sum_{(\ga')\in \hat\Xi^0_m}\int_{ [\barM_{1,1}(\Pf,d_v)^p]^v\times [\cW_{(\ga')}]^v}   \frac{  \ft\cdot e_T(R\pi_{v\ast} \cL_v\otimes \cL_\ft) }{e_T(R\pi_{v\ast} \cL_v\otimes \cL_\ft)}      \frac{(-1) PD(h_e+\ft)}{w_{(e,v)}-\psi_{(e,v)}}    \\
 &&\cdot  \frac{c_{{\rm top}-1}(\pi_{\ga' \ast}\cL_{\ga'}\cL_\ft  (-D_{\ga'}))}{e(N_{\cW_{\ga'}/\cW}\virt)}  \nonumber\\
  &=&\frac{1}{5}    \sum_{j=0}^3     \sum_{d_v=1}^\infty  (-\sQ)^{d_v} \int_{ [\barM_{1,1}(\Pf,d_v)^p]^v}   \ft     \ev\sta h^j     \\
  &&\cdot     \sum_{m=1}^\infty (-\sQ)^m   \sum_{(\ga')\in \hat\Xi^0_m} \int_{[\cW_{(\ga')}]\virt}   \frac{-  \ev\sta   h^{3-j}(h+\ft)  }{w_{(e,v)} }        \black  \frac{c_{{\rm top}-1}(\pi_{{\ga'} \ast}\cL_{\ga'}\cL_\ft  (-D_{\ga'}))}{e(N_{\cW_{\ga'}/\cW}\virt)}\nonumber\\
  &=&  \frac{1}{5}    \bigg(   \sum_{d_v=1}^\infty d_v N_{1,d_v} \sQ^{d_v} \bigg) 
  \sum_{m=1}^\infty (-\sQ)^m   \sum_{(\ga')\in \hat\Xi^0_m} \int_{[\cW_{(\ga')}]\virt}   \frac{\ft}{\psi'_\ga }   \black     \ev\sta  \bl h^{2}(h+\ft) \br    \black  \frac{c_{{\rm top}-1}(\pi_{\ga' \ast}\cL_{\ga'}\cL_\ft  (-D_{\ga'}))}{e(N_{\cW_{\ga'}/\cW}\virt)}  \\
&=& \frac{1}{5}    \bigg(   \sum_{d_v=1}^\infty d_v N_{1,d_v} \sQ^{d_v} \bigg) \cdot  \ft \sK_{1,2} = \frac{\ft}{5}  \cdot     \sK_{1,2}  \cdot \frac{d}{dT}  \sF_1.  \end{eqnarray*}
where $\sF_1:=\sum_{d=1}^\infty N_{1,d}\sQ^d$ with $\sQ=e^T$. \black

 \subsection{Contribution from type $A$ graphs $\ga$ with $d_v=0,v\in V_0(\ga),g_v=1$}

Let $A_+^0$ be the graph type  of  $d_v=0,v\in V_0(\ga),g_v=1$. Such graphs are of  the following shape.
 
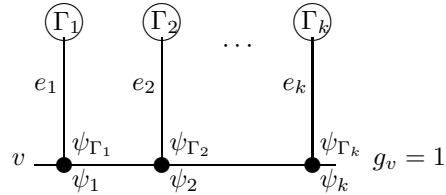
\begin{center}
\begin{picture}(20,28)
 
 \put(-1,24){\circle{5}}
\put(-2.5,23){$\Gamma_1$}
  \put(-1,22){\line(0,-1){16}}
\put(-5,15){$e_1$}
\put(-1,5){\circle*{2}}
\put(0,7){$\psi_{\ga_1}$}
\put(0,2){$\psi_{1}$}

 \put(12,24){\circle{5}}
\put(10.5,23){$\Gamma_2$}
  \put(12,22){\line(0,-1){16}}
\put(8,15){$e_2$}
\put(12,5){\circle*{2}}
\put(13,7){$\psi_{\ga_2}$}
\put(13,2){$\psi_{2}$}

\put(20, 20){$\cdots$}

 \put(32,24){\circle{5}}
\put(30.5,23){$\Gamma_k$}
  \put(32,22){\line(0,-1){16}}
\put(28,15){$e_k$}
\put(32,5){\circle*{2}}
\put(33,7){$\psi_{\ga_k}$}
\put(33,2){$\psi_{k}$}

 \put(-5,5){\line(1,0){40}}

\put(-8, 5){$v$}
\put(40,5){$g_v=1$}

\end{picture}
\end{center}

 In the above graph,   $\ga $ is decomposed uniquely to (i) a $\ga_0=\{v\}\in \Xi_0$ with $d_v=0$, and  (ii) a sequence $\ga_1,\cdots,\ga_k\in \hat\Xi^0$  where each edge containing the  marking is denoted by $e_i$  (for $i=1,\cdots,k$).

 Denote $\psi_{(e_i,v)}=\psi_i$ and $w_{(e_i,v)}:=-\psi_{\ga_i}$, where $\psi_{\ga_i}$ is the $\psi$-class of $\sC_{e_i}$. Then the contribution we want  is
\begin{eqnarray*}&& \ft \sum_{m,k=1}^\infty  \frac{(-\sQ)^{m} } {k!} \sum_{\substack{ (\ga_i) \in \hat\Xi^0_{m_i}    \\  \sum m_i=m }}  \sum_{\substack{a +1=\sum a_i\\ a,a_i\geq 0}}
\int_{[\cW_{\ga_0}]^v\times \prod[\cW_{(\ga_i)}]^v}   \frac{ c_{a}(R\pi_{\ga_0 \ast}\cL_{\ga_0}\cL_\ft)   }{e_T(R\pi_{\ga_0\ast} \cL_{\ga_0}\otimes \cL_\ft)} \\
&&\cdot\prod_{i=1}^k    (-PD_i)   \frac{h_{e_i}+\ft}{w_{(e_i,v)}-\psi_i}    \black   \frac{c_{m_i-a_i}(\pi_{\ga_i \ast}\cL_{\ga_i}\cL_\ft  (-D_{\ga_i}))}{e(N_{\cW_{\ga_i}/\cW}\virt)}  \\
 &=& \ft \sum_{m,k=1}^\infty \frac{(-\sQ)^{m} } {k!} \sum_{\substack{ (\ga_i) \in \hat\Xi^0_{m_i}    \\  \sum m_i=m  }} \sum_{a=0}^\infty {\rm Coe}_{(\ft')^{a+1}}
\int_{[\cW_{\ga_0}]^v\times \prod[\cW_{(\ga_i)}]^v}   \frac{ c_{a}(R\pi_{\ga_0 \ast}\cL_{\ga_0}\cL_\ft)   }{e_T(R\pi_{\ga_0\ast} \cL_{\ga_0}\otimes \cL_\ft)}\\
&&\cdot \prod_{i=1}^k    PD_i   \frac{h_{e_i}+\ft}{\psi_{\ga_i}+\psi_i}    \black   \frac{e(\pi_{\ga_i \ast}\cL_{\ga_i}\cL_\ft\cL_{\ft'}  (-D_{\ga_i}))}{e(N_{\cW_{\ga_i}/\cW}\virt)}, 
 \end{eqnarray*}
where  $PD_i=\frac{1}{5}\sum_{j=0}^3 h^j\otimes h_{e_i}^{3-j}$. \black
Using $\lam^2=h^2=\lam h=0$ over 
$[\cW_{\ga_0}]\virt=(\barM_{1,k}\times Q ) \cap  [ -40h^3-10\lam h^2 ]$, 
$$  \frac{c(R\pi_{\ga_0 \ast}\cL_{\ga_0}\cL_\ft)}{e(R\pi_{\ga_0 \ast}\cL_{\ga_0}\cL_\ft)} = \frac{1+h+\ft}{h+\ft}   \frac{\ft+h -\lam}{1+\ft+h -\lam} 
= 1- \frac{\lam}{\ft(1+\ft)}  = 1-\frac{\lam}{\ft} +\lam -\lam \ft +\lam \ft^2-\lam \ft^3 +\cdots  + (-1)^\ell \lam  \ft^\ell+\cdots. $$
Hence we get 
 $$ \frac{c_a(R\pi_{\ga_0 \ast}\cL_{\ga_0}\cL_\ft)}{e(R\pi_{\ga_0 \ast}\cL_{\ga_0}\cL_\ft)}= 
 \begin{cases}  (-1)^{a-1} \lam  \ft^{a-1}  &\quad \text{for}\ a\in\NN,\\
   1-\lam{\ft}^{-1} &\quad \text{for}\ a=0.
   \end{cases}$$

 We  express   the contribution    as the sum  over $a=0,1,2,3,\cdots$ as follows.

Case $a\in\NN$:  As $ [ -40h^3-10\lam h^2 ]((-1)^{a-1} \lam  \ft^{a-1} )=(-1)^{a} 40h^3 \lam     \ft^{a-1} $, the contribution is
 \begin{eqnarray*}
&&\ft \sum_{k,m_i=1}^\infty  \frac{(-\sQ)^{m_1+\cdots m_k} } {k!} \sum_{\substack{ (\ga_i) \in \hat\Xi^0_{m_i}      }}  {\rm Coe}_{(\ft')^{a+1}}
\int_{\barM_{1,k} \times Q\times \prod[\cW_{(\ga_i)}]^v}   (-1)^{a} 40h^3 \lam     \ft^{a-1}\\
&&\cdot \prod_{i=1}^k    \frac{h^3_{e_i}}{5}   \frac{h_{e_i}+\ft}{\psi_{\ga_i}+\psi_i}    \black   \frac{e(\pi_{\ga_i \ast}\cL_{\ga_i}\cL_\ft\cL_{\ft'}  (-D_{\ga_i}))}{e(N_{\cW_{\ga_i}/\cW}\virt)} \\
&=& \ft \sum_{k=1}^\infty  \frac{1} {k!}   {\rm Coe}_{(\ft')^{a+1}}
\int_{\barM_{1,k} \times Q}   \frac{(-1)^{a}}{5^k} 40h^3 \lam     \ft^{a-1}   \hat\sK_1(\psi_1)\hat\sK_1(\psi_2)\cdots\hat\sK_1(\psi_k)\\
&=& (-1)^{a}\ft^{a}   \sum_{k=1}^\infty  \frac{200} {5^k k!}     {\rm Coe}_{(\ft')^{a+1}}
\int_{\barM_{1,k} }      \lam       \hat\sK_1(\psi_1)\hat\sK_1(\psi_2)\cdots\hat\sK_1(\psi_k).
     \end{eqnarray*}

 Case $a=0$: As $ [ -40h^3-10\lam h^2 ](1-\frac{\lam}{\ft})=-40h^3+40h^3\frac{\lam}{\ft}-10\lam h^2$, the contribution is the sum of:
  \begin{eqnarray*}
 &(i)& \ft \sum_{k,m_i=1}^\infty  \frac{(-\sQ)^{m_1+\cdots m_k} } {k!} \sum_{\substack{ (\ga_i) \in \hat\Xi^0_{m_i}    }}   {\rm Coe}_{(\ft')^{1}}
\int_{ \barM_{1,k} \times Q\times \prod[\cW_{(\ga_i)}]^v}   40 (h^3\frac{\lam}{\ft}-h^3 ) \\
&&\cdot   \prod_{i=1}^k    PD_i   \frac{h_{e_i}+\ft}{\psi_{\ga_i}+\psi_i}    \black   \frac{e(\pi_{\ga_i \ast}\cL_{\ga_i}\cL_\ft\cL_{\ft'}  (-D_{\ga_i}))}{e(N_{\cW_{\ga_i}/\cW}\virt)} \\
&=&\ft \sum_{k=1 }^\infty  \frac{1} {5^k k!}   {\rm Coe}_{(\ft')^{1}}
\int_{ \barM_{1,k} \times Q }    (-40h^3+40h^3\frac{\lam}{\ft} )    \hat\sK_1(\psi_1)\cdots\ti\sK_1(\psi_k)\\
&=& \ft \sum_{k=1 }^\infty  \frac{1} {5^k k!}  {\rm Coe}_{(\ft')^{1}}
\int_{ \barM_{1,k} }    200(\frac{\lam}{\ft}-1 )    \hat\sK_1(\psi_1)\cdots\ti\sK_1(\psi_k), 
   \end{eqnarray*}  \black
    and $(ii)$, recalling $PD_i=\frac{1}{5}\sum_{j=0}^3 h^j\otimes h_{e_i}^{3-j}$, 
    \begin{eqnarray*}
&&\ft   \sum_{k,m_i=1}^\infty  \frac{(-\sQ)^{m_1+\cdots m_k} } {k!} \sum_{\substack{ (\ga_i) \in \hat\Xi^0_{m_i}    }}   {\rm Coe}_{(\ft')^{1}}
\int_{ \barM_{1,k} \times Q\times \prod[\cW_{(\ga_i)}]^v}    ( -10\lam h^2 )  \\
&&\cdot  \prod_{i=1}^k    PD_i \cdot  \frac{h_{e_i}+\ft}{\psi_{\ga_i}+\psi_i}    \black   \frac{e(\pi_{\ga_i \ast}\cL_{\ga_i}\cL_\ft\cL_{\ft'}  (-D_{\ga_i}))}{e(N_{\cW_{\ga_i}/\cW}\virt)} \\
&=& \ft \sum_{k,m_i=1}^\infty  \frac{(-\sQ)^{m_1+\cdots m_k} } {5^kk!} \sum_{\substack{ (\ga_i) \in \hat\Xi^0_{m_i}    }}  {\rm Coe}_{(\ft')^{1}}
\int_{ \barM_{1,k} \times Q\times \prod[\cW_{(\ga_i)}]^v}    ( -10\lam h^2 ),  \\
&&  \bigg(  \sum_{j=1}^k
   (h^0\otimes h_{e_1}^3)\cdots (h^0\otimes h_{e_{j-1}}^3) (h^1\otimes h_{e_j}^2)(h^0\otimes h_{e_{j+1}}^3)\cdots (h^0\otimes h_{e_k}^3) \bigg) \\
   &&\cdot  \bigg( \prod_{i=1}^k  \frac{h_{e_i}+\ft}{\psi_{\ga_i}+\psi_i}    \black   \frac{e(\pi_{\ga_i \ast}\cL_{\ga_i}\cL_\ft\cL_{\ft'}  (-D_{\ga_i}))}{e(N_{\cW_{\ga_i}/\cW}\virt)}  \bigg) \\
&=& \ft \sum_{k=1}^\infty  \frac{1 } {5^kk!}  {\rm Coe}_{(\ft')^{1}}
\int_{ \barM_{1,k}  }    (-50) \lam     \sum_{j=1}^k  \hat\sK_1(\psi_1) \hat\sK_1(\psi_2)\cdots \hat\sK_{1}(\psi_{j-1})\hat\sK_{2}(\psi_j)\hat \sK_{1}(\psi_{j+1})\cdots \hat \sK_1(\psi_k).     \end{eqnarray*}

 Therefore the   contribution of graph type $A_+^0$ 
is
\begin{eqnarray*}
&& \ft  \sum_{k=1}^\infty  \frac{(-50) } {5^kk!}  {\rm Coe}_{(\ft')^{1}}
\int_{ \barM_{1,k}  }     \lam     \sum_{j=1}^k  \hat\sK_1(\psi_1) \hat\sK_1(\psi_2)\cdots \hat \sK_{1}(\psi_{j-1})\hat \sK_{2}(\psi_j)\hat \sK_{1}(\psi_{j+1})\cdots \hat \sK_1(\psi_k)  \nonumber\\
&&- \ft \sum_{k=1 }^\infty  \frac{200} {5^k k!}   {\rm Coe}_{(\ft')^{1}}
\int_{ \barM_{1,k} }        \hat\sK_1(\psi_1)\cdots\hat \sK_1(\psi_k)\\
&&+\sum_{a=0}^\infty  (-1)^{a}\ft^{a}   \sum_{k=1}^\infty  \frac{200} {5^k k!}     {\rm Coe}_{(\ft')^{a+1}}
\int_{\barM_{1,k} }      \lam       \hat\sK_1(\psi_1)\hat\sK_1(\psi_2)\cdots\hat\sK_1(\psi_k) \nonumber\\
&=&    \frac{(-50) } {5} \ft  {\rm Coe}_{(\ft')^{1}}
\int_{ \barM_{1,1}  }     \lam     \hat \sK_{2}(\psi_1)   -   \frac{200} {5 } \ft   {\rm Coe}_{(\ft')^{1}}
\int_{ \barM_{1,1} }        \hat\sK_1(\psi_1)   \nonumber \\
&&
-\frac{1}{\ft}  \sum_{k=1}^\infty  \frac{200} {5^k k!}  \sum_{a=0}^\infty  (-\ft)^{a+1}     {\rm Coe}_{(\ft')^{a+1}}
\int_{\barM_{1,k} }      \lam       \hat\sK_1(\psi_1)\hat\sK_1(\psi_2)\cdots\hat\sK_1(\psi_k) \nonumber \\
&=& -\frac{10}{24} \ft  \sK_{1,2}   +   \frac{40} {24 } \ft \sK_{2,1}  -\frac{1}{\ft}         \sum_{k=1}^\infty  \frac{200} {5^k k!}    
\int_{\barM_{1,k} }      \lam       \hat\sK_1|_{\ft'=-\ft}(\psi_1)\hat\sK_1|_{\ft'=-\ft}(\psi_2)\cdots\hat\sK_1|_{\ft'=-\ft}(\psi_k), 
\end{eqnarray*}
where the first equality comes from the fact that  both $\ti \sK_{c,k}$ and $ \hat\sK_k$ for any $c,k$ are formal power series of $\ft'$ with coefficients of $(\ft')^0$ vanishing.

  Using  \eqref{tic} and Lemma \ref{res}, the last term in  the above formula equals
 \begin{eqnarray}\label{sKK}
 -\frac{200}{\ft}\cdot\Big(\int_{\barM_{1,1}}\lambda \Big)\cdot\sum_{k=1}^\infty\sum_{i_1+ \ldots+ i_k=k-1\atop i_1 , \ldots, i_k\ge 0}\frac{1}{k}
\cdot  \prod_{j=1}^k\bigg(\frac{(-1)^{i_j}}{i_j!}{\rm Res}_{\hbar=0}\{\hbar^{-i_j}\frac{\tilde\sZ_{0, 1}(\ft, -\ft)}{5}\}\bigg), \end{eqnarray}
where we used $\int_{\barM_{1,k}}\lambda\psi_1^{i_1}\cdots\psi_k^{i_k}=\frac{(k-1)!}{i_1!\cdots i_k!}\int_{\barM_{1,1}}\lambda$   whenever $i_1+\ldots+i_k=k-1$.
 By \eqref{formula1} \black  in the Appendix, we have \eqref{sKK} equals 
$-\frac{200}{24}\ft^{-1} \eta=-\frac{25}{3}\displaystyle\frac{g_1}{I_0}$. Thus  the   contribution of graph type $A_+^0$  equals 
\begin{eqnarray*}
\frac{40\ft }{24}\cdot\sK_{2, 1}-\frac{10\ft}{24}\cdot\sK_{1, 2}
-\frac{25}{3}\cdot\frac{g_1}{I_0}
=\frac{25}{3 I_0}\frac{{\rm d}I_0}{{\rm d}t}-\frac{25}{12}\cdot\frac{d}{dt}(T-t)
-\frac{25}{3}\cdot\frac{g_1}{I_0}.
\end{eqnarray*}

Thus type A contribution is, by \eqref{GOD},
  \begin{eqnarray*}&& (1+\frac{\ft}{5}      \sK_{1,2} )\frac{d\sF_1}{dT}
 +\frac{25}{3 I_0}\frac{{\rm d}I_0}{{\rm d}t}-\frac{25}{12}\cdot\frac{d(T-t)}{dt}
-\frac{25}{3}\cdot\frac{g_1}{I_0}\nonumber\\
   &=&  \frac{d}{dt}\sF_1  -\frac{25}{12}\cdot  (\frac{dT}{dt}-1)
 \ +  \frac{40} {24 } \cdot 5  \frac{1}{I_0}\frac{d I_0}{dt}    -\frac{25}{3}\cdot\frac{g_1}{I_0}. \end{eqnarray*}
\section{Graph Type B}

The contribution from Graph B is the most involved. It splits into two parts. One part is in this section, another is in the Appendix \ref{sectypeBseriessimplification}.

  \subsection{Notations for $g=0$ two-point functions.}


 We define in  MSP weight  a series
 \begin{align*}
 \ti\sZ_{66}\sta :=  \frac{1}{2\hbar_1\hbar_2}\sum_{d=1}^\infty \sQ^d\int_{\barM_{0,2}(\P5,d)} \frac{e( \ti E_d)}{(\hbar_1-\psi_1)(\hbar_2-\psi_2)}ev_1\sta H^5 ev_2\sta H^5,
 \end{align*}
where $\ti E_d:=\ti\pi\lsta \ti f\sta (L_H^{\otimes 5} \oplus (L_H L_{\ft} L_{\ft'})) $ is a rank $6d+2$ bundle over $\barM_{0,2}(\P5,d)$, with $\ti\pi:\ti\cC_d\to \barM_{0,2}(\P5,d)$ and 
$\ti f:\cC_d\to \barM_{0,2}(\P5,d)$ the universal family. Then
 \begin{align*}
  \ti\sZ_{66}\sta =    \frac{-1}{2\hbar_1\hbar_2}\sum_{d=1}^\infty (-\sQ)^d\int_{[\cW_{0,(1,\rho)^2,(d,0)}]\virt} \frac{e\bl  \ti\pi\lsta \ti f\sta  (L_HL_{\ft}L_{\ft'})  \br }{(\hbar_1-\psi_1)(\hbar_2-\psi_2)}ev_1\sta H^5 ev_2\sta H^5
\end{align*}
$$\qquad =    \frac{-1}{2\hbar_1\hbar_2}\sum_{d=1}^\infty (-\sQ)^d  \sum_{(\ti\ga)\in \ti\Delta^{11}_d}\int_{[\cW_{(\ti\ga)}]\virt} \frac{e\bl  \ti\pi\lsta \ti f\sta  (L_{\ti\ga} L_{\ft} L_{\ft'})  \br }{e(N_{\cW_{(\ti\ga)}/\cW}\virt)}\frac{\ft^{10}}{(\hbar_1-\psi_1)(\hbar_2-\psi_2)},  $$
where  $\ti\Delta^{ij}_d$  denotes the set of all (regular) graphs appearing in localizing $ \CC\sta$-action on $\cW_{0,(1,\rho)^2,(d,0)}=\barM_{0,2}(\P5,d)$ which have the  first marking mapped to $V_i$ and the second marking mapped to $V_j$, for $i,j\in\{0,1\}$.\\

 Let $\ti\Delta^{11}_{d,Sep}$  denote the set of all graphs in $\ti\Delta^{11}_d$ with two markings not   on the same vertex.  

 \begin{align*}&   \ti\sZ_{66,Sep}\sta :=    \frac{-1}{2\hbar_1\hbar_2}\sum_{d=1}^\infty (-\sQ)^d  \sum_{(\ti\ga)\in \ti\Delta^{11}_{d,Sep}}\int_{[\cW_{(\ti\ga)}]\virt} \frac{e\bl  \ti\pi\lsta \ti f\sta  (L_{\ti\ga} L_{\ft} L_{\ft'})  \br }{e(N_{\cW_{(\ti\ga)}/\cW}\virt)}\frac{\ft^{10}}{(\hbar_1-\psi_1)(\hbar_2-\psi_2)}.  \\
 \end{align*}

 Let $\fD_1,\fD_2\sub \cC_{\ti\ga}\to \cW_{(\ga)}$ be the universal family and divisors given by the two markings. We also have
\begin{align*} \ti\sZ_{66,Sep}\sta = \frac{-1}{2\hbar_1\hbar_2}\sum_{d=1}^\infty (-\sQ)^d  \sum_{(\ti\ga)\in \ti\Delta^{11}_{d,Sep}}\int_{[\cW_{(\ti\ga)}]\virt}   (\ft')^2 \black \cdot  \frac{ e\bl  \ti\pi\lsta \ti f\sta  (L_{\ti\ga}L_{\ft} L_{\ft'}(-\fD_1-\fD_2)) \br }{e(N_{\cW_{(\ti\ga)}/\cW}\virt)}\frac{  \ft^{10} }{(\hbar_1-\psi_1)(\hbar_2-\psi_2)} . \end{align*}
 \begin{eqnarray}\label{star}
 &&Coe_{(\ft')^2} \ti\sZ_{66,Sep}\sta \nonumber\\
 &=& \frac{-1}{2\hbar_1\hbar_2}\sum_{d=1}^\infty (-\sQ)^d  \sum_{(\ti\ga)\in \ti\Delta^{11}_{d,Sep}}\int_{[\cW_{(\ti\ga)}]\virt}   \frac{ e\bl  \ti\pi\lsta \ti f\sta  (L_{\ti\ga}L_{\ft} (-\fD_1-\fD_2)) \br }{e(N_{\cW_{(\ti\ga)}/\cW}\virt)}\frac{\ft^{10} }{(\hbar_1-\psi_1)(\hbar_2-\psi_2)}.
 \end{eqnarray}

\black

 \subsection{Two types of contributions}


  Any  graph $\ga$ without positive genus vertex (called loop graph) that gives nonzero contribution
  to \eqref{zero0}  necessarily has $|V_1(\ga)|=1$, because of the integrand in \eqref{zero0}. We call it \textit{the geometry of cap}.  We say such a graph is of type $B$, as associated  to the left graph below.  Let $\ti \ga$ be the subgraph of $\ga$ shown below on the right. 

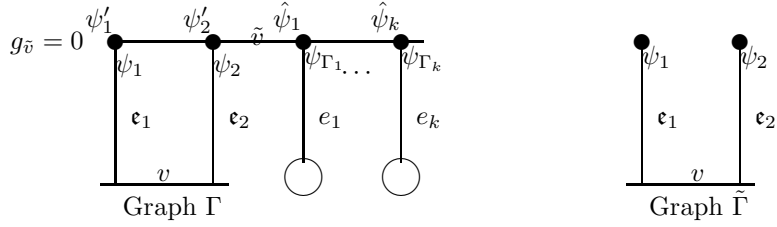
\begin{center}
\begin{picture}(30,30)

\put(-36,23){\circle*{2}}

\put(-40,25){$\psi'_1$}
 \put(-36,22){\line(0,-1){18}}
\put(-36,19){$\psi_1$}
\put(-34,12){$ \ee_1$}

\put(-23,23){\circle*{2}}
\put(-23,19){$\psi_2$}
\put(-21,12){$\ee_2$}
\put(-27,25){$\psi'_2$}
 \put(-23,22){\line(0,-1){18}}

\put(-38,4){\line(1,0){17}}
\put(-30.5, 4.5){$v$}

\put(-35, 0){Graph $\Gamma$}

\put(-11,23){\circle*{2}}
\put(-6,18){$\cdots$}
\put(-15,25){$\hat\psi_1$}
 \put(-11,22){\line(0,-1){15}}
\put(-11,20){$\psi_{\ga_1}$}
\put(-9,12){$e_1$}
\put(-11,5){\circle{5}}

\put(2,23){\circle*{2}}
\put(2,20){$\psi_{\ga_k}$}
\put(4,12){$e_k$}
\put(-2,25){$\hat\psi_k$}
 \put(2,22){\line(0,-1){15}}
\put(2,5){\circle{5}}

 \put(-35,23){\line(1,0){40}}

\put(-18, 22.5){$\ti v$}
\put(-50,22){$g_{\ti v}=0$}

\put(34,23){\circle*{2}}
 \put(34,22){\line(0,-1){18}}
\put(34,20){$\psi_1$}
\put(36,12){$\ee_1$}

\put(47,23){\circle*{2}}
\put(47,20){$\psi_2$}
\put(49,12){$\ee_2$}
  \put(47,22){\line(0,-1){18}}

\put(32,4){\line(1,0){17}}
\put(40.5, 4.5){$v$}

\put(35,0 ){Graph $\ti \Gamma$}

\end{picture}
\end{center}
 
 Here all vertices are of genus zero. Let $\ga_0$ be the graph of single vertex $\ti v\in V_1$, $g_{\ti v}=0$, no edges, and with
$k+2$ many $(1,\rho)$ markings. 
The tail of $\ti v$ that contains $e_i$ is denoted by $\ga_i$. 
 Let the total degree of $\cL$ over objects in $\ga_i$ be $d_i$.   
 Then $d=d_{\ti\ga}+d_1+d_2+\cdots + d_k$ is the degree $\sL$. We divide into two cases.
\begin{enumerate}
\item the first  case is  when  $k=0$, $\ti v$ unstable.   We leave its $
Cont_I$ to next few subsections.

 \item  the other  case is when $k>0$ and $\ti v$ stable whose contribution is denoted by 
 $Cont_{II}$.
\end{enumerate} \black

Let's consider the second case. 
 By \cite{CLLL2}, $$[\cW_{(\ga)}]\virt= |Aut\ti \ga| \cdot  \frac{  \prod|Aut\ga_i|}{|Aut\ga|}  \black [\cW_{(\ga_0)}]\virt\times  [\cW_{(\ti\ga)}]\virt\times \prod[\cW_{(\ga_i)}]\virt. $$  
  Thus
the contribution of $\ga$ 
 is
 \begin{eqnarray*}
  &&Contri_{(\ga)}:=\int_{[\cW_{(\ga)}]^v}\frac{\ft c_{d-1}R\pi_{\cW_{\ga}\ast}(\cL_{\cW_{\ga}} \cN_{\cW_{\ga}} \cL_\ft)}{e(N\virt_{\cW_{(\ga)}})} \\
  &=& \frac{\ft|Aut\ti \ga| \prod|Aut\ga_i|}{|Aut\ga|}  \sum_{\ell+\ell_0+\sum \ell_i=d-1}  \int_{[\cW_{(\ga_0)}]^v} \frac{c_\ell \big( R\pi_{\cW_{\ga_0}\ast}(\cL_{\cW_{\ga_0}}   \cL_\ft) \big) }{e(N\virt_{\cW_{(\ga_0)}})}    \\
&&\cdot \int_{[\cW_{(\ti\ga)}]^v}   \frac{\ft^8c_{\ell_0}\big( R\pi_{\cW_{\ti\ga}\ast}(\cL_{\cW_{\ti\ga}}   \cL_\ft (-\fD_1-\fD_2)) \big) }{5^2(\psi_1+\psi'_1)(\psi_2+\psi'_2)e(N\virt_{\cW_{(\ti\ga)}})}  
    \prod \int_{ [\cW_{(\ga_i)}]^v}\frac{-\ft^4c_{\ell_i}\big( R\pi_{\cW_{\ga_i}\ast}(\cL_{\cW_{\ga_i}}   \cL_\ft (-D_i) )\big)}{5(-\psi_{\ga_i}-\hat\psi_i)e(N\virt_{\cW_{(\ga_i)} })}    .
\end{eqnarray*}

 From the shape of the graph, one has $ R\pi_{\cW_{\ga_0}\ast}(\cL_{\cW_{\ga_0}}   \cL_\ft)\cong \sO_{\cW_{\ga_0}}.$
 Both $R\pi_{\cW_{\ti\ga}\ast}(\cL_{\cW_{\ti\ga}}   \cL_\ft (-\fD_1-\fD_2))$
 and $\oplus_i R\pi_{\cW_{\ga_i}\ast}(\cL_{\cW_{\ga_i}}   \cL_\ft (-D_i))$ are bundles, and the sum of  their ranks is $d-1$. We see the only possible contribution to $Contri_{(\ga)}$
 is from $\ell=0$ and all $c_{\ell_i}$ are top Chern classes.
 Therefore, with $ \frac{1 }{e(N\virt_{\cW_{(\ga_0)}})}=\frac{5\ft}{-\ft^5}$, $Contri_{(\ga)}$ equals
  \begin{eqnarray*}
&&\frac{- |Aut\ti \ga| \prod|Aut\ga_i|}{ 5\ft^5 |Aut\ga|}  \sum_{a,b,b_1,\cdots,b_k=0}^\infty \int_{\barM_{0,2+k}} (-\psi'_1)^{a}(-\psi'_2)^{b}\prod_i (-\hat\psi_i)^{b_i}    \bigg(\int_{[\cW_{(\ti\ga)}]\virt}   \frac{ \ft^{10} }{(\psi_1)^{a+1} (\psi_2)^{b+1}}   \\
&&        \cdot \frac{e\big( \pi_{\cW_{\ti\ga}\ast}(\cL_{\cW_{\ti\ga}}   \cL_\ft (-\fD_1-\fD_2)) \big) }{e(N\virt_{\cW_{(\ti\ga)}})}\bigg)
     \prod \int_{ [\cW_{(\ga_i)}]\virt}\frac{\ft^4}{5 }       \frac{1}{(\psi_{\ga_i})^{b_i+1}}        \frac{e \big( \pi_{\cW_{\ga_i}\ast}(\cL_{\cW_{\ga_i}}   \cL_\ft (-D_i) )\big)}{e(N\virt_{\cW_{(\ga_i)} })}    \\
 &=&\frac{- |Aut\ti \ga| \prod|Aut\ga_i|}{ 5\ft^5 |Aut\ga|} (k-1)! \sum_{a+b+b_1+\cdots b_k=k-1}^\infty   \frac{(-1)^{a+b}}{a!b!}   \bigg( \int_{[\cW_{(\ti\ga)}]\virt} \frac{ \ft^{10} }{(\psi_1)^{a+1} (\psi_2)^{b+1}}        \\
&&  \cdot \frac{e\big( \pi_{\cW_{\ti\ga}\ast}(\cL_{\cW_{\ti\ga}}   \cL_\ft (-\fD_1-\fD_2)) \big) }{e(N\virt_{\cW_{(\ti\ga)}})} \bigg)   \prod_{i=1}^k    \frac{(-1)^{b_i}}{b_i!}  \frac{\ft^4}{5 }    \int_{ [\cW_{(\ga_i)}]\virt}       \frac{1}{(\psi_{\ga_i})^{b_i+1}}        \frac{e \big( \pi_{\cW_{\ga_i}\ast}(\cL_{\cW_{\ga_i}}   \cL_\ft (-D_i) )\big)}{e(N\virt_{\cW_{(\ga_i)} })}.
\end{eqnarray*}

 Let  $\ti\Delta^{11,\circ}_d$  denote the graphs    in $\ti\Delta^{11}_d$ which has only one vertex at $V_0$ and whose both markings  are unstable. \black Using \eqref{star} and   the geometry of cap,   the sum of all above contributions is
    \begin{eqnarray*}
&&Cont_{II}
= \sum_{k=1}^\infty \frac{-  (k-1)!}{ 5\ft^5 \cdot 2  \cdot  k!}  \sum_{\substack{a+b+b_1+\cdots b_k=k-1\\  a,b,b_i \geq 0}}^\infty   \frac{(-1)^{a+b}}{a!b!}
 \sum_{\substack{(\ti\ga)\in \ti\Delta^{11,\circ}_d} }  (-\sQ)^{d_{\ti\ga}} \bigg(\int_{[\cW_{(\ti\ga)}]\virt}         \frac{ \ft^{10} }{(\psi_1)^{a+1} (\psi_2)^{b+1}}
    \\
&&   \cdot \frac{e\big( \pi_{\cW_{\ti\ga}\ast}(\cL_{\cW_{\ti\ga}}   \cL_\ft (-\fD_1-\fD_2)) \big) }{e(N\virt_{\cW_{\ti\ga}})} \bigg)\prod_{i=1}^k       \frac{(-1)^{b_i}}{b_i!}    \frac{\ft^4}{5 }   \sum_{(\ga_i)\in\hat\Xi^1} (-\sQ)^{d_{\ga_i}}       \int_{[\cW_{(\ga_i)}]\virt}   \frac{e(\pi_{\ga_i \ast}\cL_{\ga_i}\cL_\ft  (-D_i))}{(\psi_{\ga_i})^{b_i+1}e(N_{\cW_{(\ga_i)}/\cW}^{vir})}    \\
 &=& \sum_{k=1}^\infty \frac{1 }{ 5\ft^5   k }   \sum_{\substack{a+b+b_1+\cdots b_k=k-1\\  a,b,b_i \geq 0}}^\infty
 \frac{(-1)^{a+b}}{a!b!}
 \fR_{\hbar_2=0}\fR_{\hbar_1=0}  \big\{ \hbar_1^{-a}\hbar_2^{-b}Coe_{(\ft')^2}   \ti\sZ_{66,Sep}\sta  \big \}  \prod_{i=1}^k       \frac{(-1)^{b_i}}{b_i!}    \frac{\ft^4}{5 }   \fR_{\hbar=0}  \frac{  \sZ_{1,5}}{\hbar^{b_i}} ,
  \end{eqnarray*}
 where we used $ \sP_c =   {\rm Res}_{\hbar=0}  \frac{\sZ_{1,5}}{\hbar^{c-1}}$  in Lemma \ref{res}.  By \eqref{Zjk}, $\sZ_{1,5}=\ti\sZ_{1,5}|_{\ft'=0}$ where $\ti\sZ_{1,5}$ is regular in $\ft'$.  
   Applying   \eqref{formula1},  the above becomes  ($\sZ\sta_6:=  \frac{\ft^4}{5}   \ti\sZ_{1,5} $) \black
   \begin{eqnarray*}
&&Cont_{II}
   =   \frac{1 }{ 5\ft^5     }   \sum_{  a,b \geq 0}^\infty
   \frac{(-1)^{a+b}}{a!b!}
 \fR_{\hbar_2=0}\fR_{\hbar_1=0}  \big\{  \hbar_1^{-a}\hbar_2^{-b}Coe_{(\ft')^2}  \ti\sZ_{66,Sep}\sta    \big \}     \\
 &&\cdot \sum_{m=1}^\infty \frac{1 }{  m }     \bigg(     \sum_{\substack{a+b+\sum b_\ell=m-1\\   b_i \geq 0}}^\infty  \prod_{\ell=1}^m       \frac{(-1)^{b_\ell}}{b_\ell!}      \fR_{\hbar=0}  \frac{  \sZ\sta_6}{\hbar^{b_\ell}}  \bigg)\bigg|_{\ft'=0}\\
  &=&   \frac{1 }{ 5\ft^5     }   \sum_{  a,b \geq 0}^\infty
   \frac{(-1)^{a+b}}{a!b!}
 \fR_{\hbar_2=0}\fR_{\hbar_1=0}  \big\{  \hbar_1^{-a}\hbar_2^{-b}Coe_{(\ft')^2}  \ti\sZ_{66,Sep}\sta     \big \}
\frac{\eta_0^{a+b+1}}{a+b+1} ,\\
   \end{eqnarray*}
where we used $\sum_{a,b\geq 0}\frac{y^{a+b+1}}{(a+b+1)a!b!\hbar_1^a\hbar_2^b}=\frac{\hbar_1\hbar_2}{\hbar_1+\hbar_2} \big( e^{y (\frac{1}{\hbar_1}+\frac{1}{\hbar_2}) }  -1 \big)$, and
we  denote   $\eta_0:=\eta|_{\ft'=0}$. 


 For $Cont_{I}$, using $[\cW_{(\ga)}]\virt=  \frac{1}{2} \black [\cW_{(\ti\ga)}]\virt  $, MSP II   gives   
 us
  \begin{align*} Cont_{I}=& \frac{ 1}{2 }\sum_{d=1}^\infty (-\sQ)^d  \sum_{(\ti\ga)\in \ti\Delta^{11,\circ}_d}\int_{[\cW_{(\ti\ga)}]\virt}   \frac{ e\bl  \ti\pi\lsta \ti f\sta  (L_{\ti\ga}L_{\ft} (-\fD_1-\fD_2)) \br }{e(N_{\cW_{(\ti\ga)}/\cW}\virt)}\cdot  \frac{\ft}{(5\ft)^2}\frac{-5\ft^6}{-\psi_1-\psi_2}. \black
 \end{align*}
 
     \subsection{Computations of Contri II }
 Let  $\ti\Delta^{11,\circ}_d$  denote the graphs    in $\ti\Delta^{11}_d$ which has only one vertex at $V_0$ and whose both markings  are unstable. \black Using \eqref{star} and   the geometry of cap, 
 we can argue to obtain
 \begin{align} \label{capred}
\fX:=& \fR_{\hbar_2=0}\fR_{\hbar_1=0}  \big\{ \hbar_1^{-a}\hbar_2^{-b}Coe_{(\ft')^2} \ti\sZ_{66,Sep}\sta \big \}  \nonumber  \\
 =&  \frac{-1}{2 }\sum_{d=1}^\infty (-\sQ)^d  \sum_{(\ti\ga)\in \ti\Delta^{11,\circ}_d}\int_{[\cW_{(\ti\ga)}]\virt}   \frac{ e\bl  \ti\pi\lsta \ti f\sta  (L_{\ti\ga}L_{\ft} (-\fD_1-\fD_2)) \br }{e(N_{\cW_{(\ti\ga)}/\cW}\virt)}\cdot \ft^{10}  \frac{-1}{\psi_1^{a+1}}  \cdot \frac{-1}{\psi_2^{b +1}}.
 \end{align}  For each $\ti\ga\in\ti\Delta^{11,\circ}_d$, let $\bar{\psi}_i$ be the $\psi$-class of the marking (seperated from the node) on $\cC_v$. 
 For convenience,  we introduce  $W_{\ee_i}:=Q$ with $[W_{\ee_i}]^{\virt}=-Q/d_{\ee_i}$, and we let $d_i=d_{\ee_i}$,  and $h_{\ee_i},h_i$ be the classes on $W_{\ee_i},\barM_{0,2}(\Pf,d_v)^p$ respectively induced by the hyperplane of $\Pf$ . Also let  $D_1,D_2$ be  divisors of $\cC_{e_1},\cC_{e_2}$ given by the nodes. \black
  There are two cases, depending on $d_v>0$ or $d_v=0$, which we denote by $   Cont_{II}^+$ and $Cont_{II}^0$ respectively. Then $Cont_{II}=Cont_{II}^++Cont_{II}^0$. \black

(1)   Case one: $d_{v}>0$.  We have an exact sequence of sheaves
$$  0 \to  L_{\ti\ga}L_{\ft} (-\fD_1)|_{\cC_{\ee_1}} (-D_1) \oplus   L_{\ti\ga}L_{\ft} (-\fD_2)|_{\cC_{\ee_2}} (-D_2)\to  L_{\ti\ga}L_{\ft} (-\fD_1-\fD_2) \to    L_{\ti\ga}L_{\ft}|_{\cC_{v}}\to 0.$$
 Thus, up to capping with $[\cW_{\ti\ga}]\virt$, one has
\begin{eqnarray}\label{e12} e\bl  \ti\pi\lsta \ti f\sta  (L_{\ti\ga}L_{\ft} (-\fD_1-\fD_2)) \br&=&\prod_{a=1}^2  e\bl  \ti\pi\lsta \ti f\sta \big(  L_{\cC_{\ee_a}}L_{\ft} (-\fD_a-D_a) \big) \br  \ \ \cdot \ \  e\bl  \ti\pi\lsta \ti f\sta    L_{\ti\ga}L_{\ft}|_{\cC_{v}} \br \\
&=& \prod_{j=1}^{d_1-1} \frac{j(h_{\ee_1}+\ft)}{d_1} \  \ \cdot \ \   \prod_{j=1}^{d_2-1} \frac{j(h_{\ee_2}+\ft)}{d_2} \ \ \  \cdot \   \ft^{d_{v}+1}. \nonumber
\end{eqnarray}
 Recall $\frac{1}{5}\sum_{j=0}^3 h^j\otimes h_{\ee_i}^{3-j}$ is the Poincar\'e dual of the diagonal   in $Q\times Q$.
By \cite{CLLL2} 
, $e_T(R\pi_{v\ast} \cL_{v}\otimes L_\ft)=\ft^{d_{v}+1}$,  $$[\cW_{(\ti\ga)}]\virt= \big[ \prod_{i=1}^2\bl \frac{1}{5}\sum_{j=0}^3 h_i^j\otimes h_{\ee_i}^{3-j} \br \big] \cap \bigg( [\barM_{0,2}(\Pf,d_{v})^p]\virt \times \bl [W_{\ee_1}]\virt \times [W_{\ee_2}]\virt \br \bigg),
$$
\begin{align*}\frac{1}{e(N_{\cW_{(\ti\ga)}/\cW}\virt)}= \frac{ (5\ft)^2 (h_{\ee_1}+\ft)(h_{\ee_2}+\ft) }{e_T(R\pi_{v\ast} \cL_{v}\otimes L_\ft)}   \frac{ A_{\ee_1} \cdot A_{\ee_2}}{(\psi_1-\bar{\psi}_1)(\psi_2-\bar{\psi}_2)}   .
\end{align*}
 And for $i=1,2$, $\psi_i=\frac{h_i+\ft}{d_i}$ and $\bar{\psi_i}^3=0$ give
  $ \frac{1}{\psi_i-\bar{\psi}_i}=\sum_{\ell_i=0}^2 \frac{\bar\psi_i^{\ell_i}}{\psi_i^{\ell_i+1}}.$
 The above relation and \eqref{e12}    gives
\begin{align}\label{ee}\frac{ e\bl  \ti\pi\lsta \ti f\sta  (L_{\ti\ga}L_{\ft} (-\fD_1-\fD_2)) \br }{e(N_{\cW_{(\ti\ga)}/\cW}\virt)}=        \frac{B_{d_1}(h_{\ee_1})B_{d_2}(h_{\ee_2}) }{(\psi_1-\bar{\psi}_1)(\psi_2-\bar{\psi}_2)}.\end{align}

Thus contribution of case $d_v>0$ to     \eqref{capred}  is 
   \begin{eqnarray*}
  \fX^+ &=& 
  \frac{-1}{2 } \sum_{j_1,j_2=0}^3 \sum_{\ell_1,\ell_2=0}^2   \bigg(  \sum_{d_v =1}^\infty (-\sQ)^{d_v } \int_{[\barM_{0,2}(\Pf,d_{v})^p]\virt}  \bar{\psi}_1^{\ell_1} ev_1\sta h_{1}^{j_1}  \bar{\psi}_2^{\ell_2}  ev_2\sta h_{2}^{j_2} \frac{1}{\ft^{d_v+1}}  \ft^{d_{v}+1} \bigg)\\
 &&\cdot  \bigg(   \frac{5\ft}{5} \black \sum_{d_1=1}^\infty (-\sQ)^{d_1} \int_{ [W_{\ee_1}]\virt }    h_{\ee_1}^{3-j_1}    \frac{h_{\ee_1}+\ft}{ \psi_1^{\ell_1+1}  } A_{\ee_1}  \bl   \prod_{j=1}^{d_1-1} \frac{j(h_{\ee_1}+\ft)}{d_1} \br \      \frac{\ft^{5} }{\psi_1^{a+1}}\bigg)  \\
  &&\cdot \bigg(  \frac{5\ft}{5} \black \sum_{d_2=1}^\infty (-\sQ)^{d_2} \int_{ [W_{\ee_2}]\virt }    h_{\ee_2}^{3-j_2}    \frac{h_{\ee_2}+\ft}{\psi_2^{\ell_2+1}} A_{\ee_2}    \bl  \prod_{j=1}^{d_2-1} \frac{j(h_{\ee_2}+\ft)}{d_2}  \br\    \frac{\ft^{5} }{\psi_2^{b+1}}\bigg)\\  
   &=&\frac{1}{2 } \sum_{j_1,j_2=0}^1 \sum_{\ell_1,\ell_2=0}^2   \bigg(    \sum_{d_v =1}^\infty \sQ^{d_v } <\tau_{\ell_1}(h^{j_1}) \tau_{\ell_2}(h^{j_2})>_{g=0,d_v}     \bigg) \\ 
   && \cdot  \bigg( \frac{\ft^5}{5} \sum_{d_1=1}^\infty (-\sQ)^{d_1} \int_{ [W_{\ee_1}]\virt }       \frac{ h_{\ee_1}^{3-j_1} }{ \psi_1^{\ell_1+2+a}  }  B_{d_1}(h_{\ee_1})  \bigg)  \frac{\ft^5}{5} \sum_{d2=1}^\infty (-\sQ)^{d_2} \int_{ [W_{\ee_2}]\virt }       \frac{ h_{\ee_2}^{3-j_2} }{ \psi_2^{\ell_2+2+b}  }  B_{d_2}(h_{\ee_2})  .\\
 \end{eqnarray*}

  
Thus the contribution from Case one  is 
  \begin{align*} 
&Cont_{II}^+
 =   \frac{1 }{ 5\ft^5     }   \sum_{  a,b \geq 0}^\infty
   \frac{(-1)^{a+b}}{a!b!}
\fX^+\black
\frac{\eta_0^{a+b+1}}{a+b+1} \nonumber  
= \frac{1}{2 }  \frac{\ft^5 }{ 5}
 \sum_{j_1,j_2=0}^1 \sum_{\ell_1,\ell_2=0}^2   \sum_{d_v =1}^\infty \sQ^{d_v } <\tau_{\ell_1}(h^{j_1}) \tau_{\ell_2}(h^{j_2})>_{g=0,d_v}    \\
   & \cdot  \bigg( \frac{1}{5^2} \sum_{d_1,d_2=1}^\infty (-\sQ)^{d_1+d_2} \int_{ [W_{\ee_1}]\virt \times  [W_{\ee_2}]\virt}       \frac{ h_{\ee_1}^{3-j_1} }{ \psi_1^{\ell_1+1}  }  \frac{ h_{\ee_2}^{3-j_2} }{ \psi_2^{\ell_2+1}  } B_{d_1}(h_{\ee_1})    B_{d_2}(h_{\ee_2})     \frac{ e^{(-\eta_0) (\frac{1}{\psi_1}+\frac{1}{\psi_2})}-1 }{\psi_1+\psi_2}    \black   \bigg).   \nonumber
  \end{align*}

(2) Case  two: $d_{v}=0$. Then $v$ is unstable. Similarly by \cite{CLLL2}, $[\cW_{(\ti\ga)}]\virt=  \frac{[-Q]}{d_1d_2}  , \ \psi_i=\frac{h_{e_i}+\ft}{d_i}$, 
\begin{eqnarray*}&&\frac{1}{e(N_{\cW_{(\ti\ga)}/\cW}\virt)}= (5\ft)^2  \frac{(h+\ft)A_{\ee_1}A_{\ee_2}}{\psi_1+\psi_2},        
\,
 e\bl  \ti\pi\lsta \ti f\sta  (L_{\ti\ga}L_{\ft} (-\fD_1-\fD_2)) \br\\
 &=&(h+\ft)\prod_{j=1}^{d_1-1} \frac{j(h_{\ee_1}+\ft)}{d_1}  \prod_{j=1}^{d_2-1} \frac{j(h_{\ee_2}+\ft)}{d_2}  .\end{eqnarray*}
 Thus contribution of Case two  to  \eqref{capred} is,   
   \begin{eqnarray*}
 &&\fX^0:=  \frac{-1}{2 }          \sum_{d_1,d_2=1}^\infty \frac{(-\sQ)^{d_1+d_2}}{d_1d_2} \bigg(\int_{ [-Q] }    (5\ft)^2  \frac{h+\ft}{\psi_1+\psi_2}    A_{\ee_1} A_{\ee_2}   \\
 &&\cdot \bl   \prod_{j=1}^{d_1-1} \frac{j(h_{\ee_1}+\ft)}{d_1} \br   \bl  \prod_{j=1}^{d_2-1} \frac{j(h_{\ee_2}+\ft)}{d_2}  \br   (h+\ft) \black \ft^{10}   \frac{1}{\psi_1^{a+1}}
           \frac{1}{\psi_2^{b+1}} \bigg)\\
           & =&  \frac{1}{2 }          \sum_{d_1,d_2=1}^\infty \frac{(-\sQ)^{d_1+d_2}}{d_1+d_2} \int_{ Q }    (5\ft)^2     A_{\ee_1} A_{\ee_2}   \bl   \prod_{j=1}^{d_1-1} \frac{j(h_{\ee_1}+\ft)}{d_1} \br   \bl  \prod_{j=1}^{d_2-1} \frac{j(h_{\ee_2}+\ft)}{d_2}  \br   (h+\ft) \black \ft^{10}   \frac{1}{\psi_1^{a+1}}          \frac{1}{\psi_2^{b+1}} . 
                           \end{eqnarray*}
 

 We see that the contribution  of Case 2  to $Contri_{II}$ is
   \begin{eqnarray*}
  Cont_{II}^0
 &=&   \frac{1 }{ 5\ft^5     }   \sum_{  a,b = 0}^\infty  \frac{(-1)^{a+b}}{a!b!} \fX^0\black \frac{\eta_0^{a+b+1}}{a+b+1}.
   \end{eqnarray*}


     \subsection{Computations of Contri I}  The graph of the contribution type $I$ is a type $B$ graph $\ga$ with the unique vertex $\ti v\in V_1(\ga)$ unstable.  Let $\ti \ga$ be the  graph obtained by resolving the node at $V_1$ of $\ga$ shown in the Figure Graph $\ti \Gamma$. Then    $\ti\ga$ is a  graph in $\ti\Delta^{11,\circ}_d$.  We have $[\cW_{(\ga)}]\virt=  \frac{1}{2} \black [\cW_{(\ti\ga)}]\virt . $    The contribution of $\ga$ 
 is
  \begin{align*}  \frac{ 1}{2 }\sum_{d=1}^\infty (-\sQ)^d  \sum_{(\ti\ga)\in \ti\Delta^{11,\circ}_d}\int_{[\cW_{(\ti\ga)}]\virt}   \frac{ e\bl  \ti\pi\lsta \ti f\sta  (L_{\ti\ga}L_{\ft} (-\fD_1-\fD_2)) \br }{e(N_{\cW_{(\ti\ga)}/\cW}\virt)}\cdot  \frac{\ft}{(5\ft)^2}\frac{-5\ft^6}{-\psi_1-\psi_2}. \black
 \end{align*}

  There are two cases, depending on $d_v>0$ or $d_v=0$.

\begin{enumerate}

\item   Case one: $d_{v}>0$.  The  factors   are the same as those in $Cont_{II}^+$.  By \eqref{ee}, one has
  \begin{eqnarray*}
  Cont_{I}^+
=&&  \frac{-1}{2 }  \frac{\ft^5}{5}    \sum_{j_1,j_2=0}^1 \sum_{\ell_1,\ell_2=0}^2   \bigg(    \sum_{d_v =1}^\infty \sQ^{d_v } <\tau_{\ell_1}(h^{j_1}) \tau_{\ell_2}(h^{j_2})>_{g=0,d_v}     \bigg)  \\
 &&  \bigg( \frac{1}{5^2} \sum_{d_1,d_2=1}^\infty (-\sQ)^{d_1+d_2} \int_{ [W_{\ee_1}]\virt\times [W_{\ee_2}]\virt  }       \frac{ h_{\ee_1}^{3-j_1} }{ \psi_1^{\ell_1+1}  }  B_{d_1}(h_{\ee_1})           \frac{  h_{\ee_2}^{3-j_2}  }{\psi_2^{\ell_2+1}}  B_{d_2}(h_{\ee_2})    \frac{1}{\psi_1+\psi_2} \black   \bigg). \\
 \end{eqnarray*} 




\item   Case two: $d_{v}=0$.  The factors are the same as those  in 
  $Cont_{II}^0$. 
Thus contribution of this case     is,   noting $h_{\ee_1}=h_{\ee_2}=h$ over the diagonal $Q\sub Q\times Q$,
\black
    \begin{eqnarray*}
    Cont_{I}^0
  &=&   \frac{1}{2 }          \sum_{d_1,d_2=1}^\infty \frac{(-\sQ)^{d_1+d_2}}{d_1d_2}  \int_{ [-Q] }    (5\ft)^2  \frac{(h+\ft) A_{\ee_1} A_{\ee_2}}{\psi_1+\psi_2}       \bl   \prod_{j=1}^{d_1-1} \frac{j(h_{\ee_1}+\ft)}{d_1} \br   \bl  \prod_{j=1}^{d_2-1} \frac{j(h_{\ee_2}+\ft)}{d_2}    \br   \frac{ \ft  (h+\ft) }{ \psi_1+\psi_2} \black    \frac{  \ft^4}{5}  \nonumber\\
    &=& - \frac{1}{2 }       \frac{  \ft^5}{5}      \sum_{d_1,d_2=1}^\infty \frac{(-\sQ)^{d_1+d_2}}{d_1+d_2} \int_{ Q }    (5\ft)^2      A_{\ee_1} A_{\ee_2}   \bl   \prod_{j=1}^{d_1-1} \frac{j(h_{\ee_1}+\ft)}{d_1} \br   \bl  \prod_{j=1}^{d_2-1} \frac{j(h_{\ee_2}+\ft)}{d_2}  \br     \frac{ (h+\ft)\black   }{ \psi_1+\psi_2}. \nonumber
             \end{eqnarray*}
\black
\end{enumerate} 
\subsection{Summations $Cont_{II}^++Cont_I^+$ and   $Cont_{II}^0+Cont_I^0$} 
 Using $$  \sum_{  a,b = 0}^\infty
   \frac{(-1)^{a+b}}{a!b!}    \frac{1}{\psi_1^{a+1}}          \frac{1}{\psi_2^{b+1}}   \black
\frac{\eta_0^{a+b+1}}{a+b+1}
=\frac{-1}{\psi_1+\psi_2}\big(  e^{(-\eta_0)(\frac{1}{\psi_1}+\frac{1}{\psi_2})} -1  \big),$$
 we see the sum $Cont_{II}^++Cont_I^+$ equals 
      \begin{eqnarray}\label{II+I+}
&&Cont^+:=Cont_{II}^+  +Cont_{I}^+
 = \frac{1}{2 }  \frac{\ft^5 }{ 5}
 \sum_{j_1,j_2=0}^1 \sum_{\ell_1,\ell_2=0}^2   \bigg(    \sum_{d_v =1}^\infty \sQ^{d_v } <\tau_{\ell_1}(h^{j_1}) \tau_{\ell_2}(h^{j_2})>_{g=0,d_v}     \bigg)  \nonumber\\
   &&\cdot \bigg(- \displaystyle\frac{1}{5^2} \sum_{d_1,d_2=1}^\infty (-\sQ)^{d_1+d_2} \int_{ [W_{\ee_1}]\virt \times  [W_{\ee_2}]\virt}       \frac{ h_{\ee_1}^{3-j_1} }{ \psi_1^{\ell_1+1}  }  \frac{ h_{\ee_2}^{3-j_2} }{ \psi_2^{\ell_2+1}  } B_{d_1}(h_{\ee_1})    B_{d_2}(h_{\ee_2})     \frac{ e^{(-\eta_0) (\frac{1}{\psi_1}+\frac{1}{\psi_2})} }{\psi_1+\psi_2}     \bigg).  
  \end{eqnarray}
      A similar combination holds for $Cont_{II}^0+Cont_{I}^0$.
Recall  $h_{\ee_1}=h=h_{\ee_2}$ over diagonal $Q\sub Q\times Q$. We have
 \begin{eqnarray}\label{II0I0}
   &&Cont^0:= Cont_{II}^0+Cont_{I}^0 
     = - \frac{1}{2 }       \frac{  \ft^5}{5}      \sum_{d_1,d_2=1}^\infty \frac{(-\sQ)^{d_1+d_2}}{d_1+d_2} \bigg(\int_{ Q }     (5\ft)^2      A_{\ee_1} A_{\ee_2}   \\
     &&\cdot \bl   \prod_{j=1}^{d_1} \frac{j(h_{\ee_1}+\ft)}{d_1} \br   \bl  \prod_{j=1}^{d_2} \frac{j(h_{\ee_2}+\ft)}{d_2}   \br \black   \frac{1}{d_1d_2} \frac{1}{\psi_1} \frac{1}{\psi_2}     \frac{h+\ft}{ \psi_1+\psi_2}  e^{(-\eta_0)(\frac{1}{\psi_1}+\frac{1}{\psi_2})}\bigg).
      \nonumber   
  \end{eqnarray}
 
\subsection{  $S$-series  and $I$-functions}

 In this subsection we introduce a $S$-series which is a generating function of MSP's $A_e$ for all edges $e$ connecting $V_0$ with $V_1$. We express $S$-series using Givental's  $I$ -unctions. This will enable us to express the two series $Cont_{II}^++Cont_{II}^0$ and $Cont_{I}^++Cont_{I}^0$ in previous subsections using Givental's  $I$-functions. 
 
For $j\in \ZZ_{\leq 3}$,  $\psi=\frac{h_\ee+\ft}{d_\ee}$, and for arbitrary   $m \in\ZZ $, \black we formulate
\begin{align*}
 T_{j,m}: =\frac{\ft^5}{5}\sum_{d=1}^\infty (-\sQ)^{d} \int_{ [W_{\ee}]\virt }       \frac{ h^{3-j} }{ \psi^{m}  }  B_{d}(h):=S_{j,m-2}.
\end{align*}
Note that this implies $S_{j,m}=0$ whenever $j<0$. Also, by definition, we have
 \begin{align*}
 \int_{ [W_{\ee}]\virt }   d    \frac{ h^{3-j} }{ \psi^{m+2}  }   B_{d}(h) 
  =  \int_{ [W_{\ee}]\virt }  (h+\ft)  \frac{ h^{3-j} }{ \psi^{m+3}  }  B_{d}(h)
   =  \int_{ [W_{\ee}]\virt }       \frac{ h^{3-(j-1)} }{ \psi^{m+3}  }   B_{d}(h)+ \ft\int_{ [W_{\ee}]\virt }       \frac{ h^{3-j} }{ \psi^{m+3}  }   B_{d}(h).
 \end{align*}
This implies directly
 \begin{align*}
  \partial_T S_{j,m}
  = S_{j-1,m+1}+\ft S_{j,m+1},   \qquad \forall j\in\ZZ_{\leq 3},  m\in\ZZ.
 \end{align*}

 We also have  $$S_{j,m}|_{T=-\infty}=0,\qquad   T^{k}S_{j,m}|_{T=-\infty}=0 ,\forall k\in\RR, $$ \black
where $-\infty$ refers to the limit of the point on the real line with the 
real part approaching the  negative infinity.

 On the other hand, denote $I_{\ell}=0=b_\ell(d)$ whenever $\ell<0$. By \eqref{bi}, for $j=0,\cdots,3$, and $m \in\ZZ$, we have   \black 
$$T_{j,m}=  \ft^{5-j-m} \sum_{d=1}^\infty   d^{m-1}   (-\sQ)^{d} \bigg[ - b_j(d) +  mb_{j-1}(d)     -  \binom{m+1}{2} b_{j-2}(d) + \binom{m+2}{3} b_{j-3}(d)  \bigg]. $$

 We denote $I_j(y)=I_j(t)|_{t\mapsto y}$, and $I_j^{(k)}=\partial_y^kI_j(y)$.
 As an example,  we have 
  $$S_{0,-1}=T_{0,1}
 =-\ft^{5} \sum_{d=1}^\infty (-\sQ)^{d}b_0(d) \frac{1}{\ft}=-\ft^4\big( I_0(y)-1\big)|_{y\mapsto T},$$
 $$S_{1,-1}=T_{1,1}
= \ft^3  I_1'|_{t\mapsto T}  -\ft^3   T I_0'(y)\ |_{y\mapsto T}  - \ft^3. $$

Applying the formulae in \eqref{MSPtoGiv}, one directly computes
 \begin{lemm} Denote $I_{\ell}=0$ whenever $\ell<0$. Then for $m>0$ and $j=0,\cdots,3$,  we have
\begin{eqnarray*}T_{j,m}&=& (-1)^{j+1}\ft^{5-j-m} \Big(  I_j^{(m+j-1)}    - t I_{j-1}^{(m+j-1)}   + \frac{t^2}{2} I_{j-2}^{(m+j-1)}     -\frac{t^3}{6}I_{j-3}^{(m+j-1)}\Big) \ |_{t \mapsto T}  \\
&&+(-1)^j \ft^{5-j-m} \binom{m+j-1}{j} \delta_{m,1.} 
\end{eqnarray*}
\end{lemm}

\begin{lemm}\label{S-series}
    For arbitrary $\ell\in\ZZ_{\geq -1}$, $j=0,\cdots,3$, we have  
\begin{align*}\sum_{  a = 0}^\infty  \frac{(-1)^{a}}{a!}         S_{j,a+\ell}      (-\ft(y-T))^{a}
=& (-1)^{j+1}\ft^{3-j-\ell} \bigg(  I_j^{(\ell+j+1)}(y)  -T I_{\ell+j}^{(\ell+j+1)}(y) +\frac{T^2}{2} I_{\ell+j-1}^{(\ell+j+1)}(y)\\
& -\frac{T^3}{6} I_{\ell+j-2}^{(\ell+j+1)}(y)   \bigg)+(-1)^j\ft^{3-j-\ell}\binom{\ell+j+1}{j}\delta_{\ell,-1}.
\end{align*} 
\end{lemm}

     \subsection{Expression of  $Cont^+$  by  $I$-functions}

\black

  Let us define a function, for $j_1,j_2\in\ZZ_{\leq 3},\ell_1, \ell_2\in\ZZ$,

\begin{eqnarray*} && M_{j_1,j_2,\ell_1,\ell_2}(x,T)\\
&&:=- \frac{1}{2 }       \frac{  \ft^6}{5}      \sum_{d_1,d_2=1}^\infty \frac{(-e^T)^{d_1+d_2}}{d_1+d_2} \int_{ [W_{\ee_1}]\virt \times  [W_{\ee_2}]\virt } \frac{1}{5}  \frac{ h_{\ee_1}^{3-j_1} }{ \psi_1^{\ell_1+1}  }  \frac{ h_{\ee_2}^{3-j_2} }{ \psi_2^{\ell_2+1}  } B_{d_1}(h_{\ee_1})B_{d_2}(h_{\ee_2})       \frac{ e^{(\ft x)(\frac{1}{\psi_1}+\frac{1}{\psi_2})}}{ \psi_1+\psi_2} .  
\end{eqnarray*}

   Then
   \begin{eqnarray*}
  && \partial_TM_{j_1,j_2,\ell_1,\ell_2}(x,T)\\
  &=&- \frac{1}{2 }       \frac{  \ft^6}{5}      \sum_{d_1,d_2=1}^\infty (-e^T)^{d_1+d_2}  \int_{ [W_{\ee_1}]\virt \times  [W_{\ee_2}]\virt } \frac{1}{5}  \frac{ h_{\ee_1}^{3-j_1} }{ \psi_1^{\ell_1+1}  }  \frac{ h_{\ee_2}^{3-j_2} }{ \psi_2^{\ell_2+1}  } B_{d_1}(h_{\ee_1})B_{d_2}(h_{\ee_2})       \frac{ e^{(\ft x)(\frac{1}{\psi_1}+\frac{1}{\psi_2})}}{ \psi_1+\psi_2} .
   \end{eqnarray*}


  We have, when $x$ or $T$ approaches real $-\infty$, 
       \begin{align*}  M_{j_1,j_2,\ell_1,\ell_2}(x,T)|_{x\to-\infty}=0,\quad     \partial_TM_{j_1,j_2,\ell_1,\ell_2}(x,T)|_{x\to-\infty}=0  \black ,\quad  M_{j_1,j_2,\ell_1,\ell_2}(x,T) |_{T\to-\infty}=0.
     \end{align*}

      This implies $$\partial_TM_{j_1,j_2,\ell_1,\ell_2}(z,T)=\int_{-\infty}^z\partial_x\partial_TM_{j_1,j_2,\ell_1,\ell_2}(x,T)dx,$$
      $$   M_{j_1,j_2,\ell_1,\ell_2}(x,u)=\int_{-\infty}^u \partial_w  M_{j_1,j_2,\ell_1,\ell_2}(x,w)dw=\int_{-\infty}^u\int_{-\infty}^x\partial_z\partial_wM_{j_1,j_2,\ell_1\ell_2}(z,w)dzdw .$$

      We have a closed formula
  \begin{align*}
   \partial_x\partial_TM_{j_1,j_2,\ell_1,\ell_2}(x,T)
  &= - \frac{1}{2 }       \frac{ 1}{\ft^3}    \sum_{a,b=0}^\infty    S_{j_1,\ell_1+a} S_{j_2,\ell_2+b}     \frac{ (\ft x)^a}{a!}   \frac{ (\ft x)^b}{b!}.
  \end{align*}

   The definition of $M$ enables us to simplify, by \eqref{II+I+} , \eqref{II0I0}         and the above formula of $  \partial_x\partial_TM$,


   \begin{eqnarray*}
 &&Cont^+:=Cont_{II}^+  +Cont_{I}^+ \\
 &=& \frac{1}{ 5\ft}
 \sum_{j_1,j_2=0}^1 \sum_{\ell_1,\ell_2=0}^2   \bigg(    \sum_{d_v =1}^\infty \sQ^{d_v } <\tau_{\ell_1}(h^{j_1}) \tau_{\ell_2}(h^{j_2})>_{g=0,d_v}     \bigg)      \big(\partial_TM_{j_1,j_2,\ell_1,\ell_2}(x,T) \big)|_{x=t-T} \\
 &=& \frac{1}{ 5\ft}  \partial_T^2 F_0  \     \big(\partial_T M_{1,1,0,0}(x,T) \big)|_{x=t-T}  \\
 &&+ \frac{1}{ 5\ft} 2 \partial_T F_0  \bigg(      \big(\partial_T M_{1,0,1,0}(x,T) \big)|_{x=t-T}  -  \big(\partial_T M_{1,0,0,1}(x,T) \big)|_{x=t-T}    \bigg) \\
 &&+  \frac{1}{ 5\ft} 2 F_0 \bigg(    \big(\partial_T  M_{0,0,1,1}(x,T) \big)|_{x=t-T}      -2   \big(\partial_T  M_{0,0,2,0}(x,T) \big)|_{x=t-T} \bigg)\\
  &=&- \frac{1}{10} \  \partial_T^2 F_0 \  \int_{-\infty}^t (I_1^{(2)}-TI_0^{(2)})^2(y)dy  \\
 &&- \frac{1}{ 10 } 2  \partial_T F_0  \bigg(- \int_{-\infty}^t  (I_1^{(3)}-TI_0^{(3)})(y) I_0^{(1)}(y) dy    +  \int_{-\infty}^t   (I_1^{(2)}-TI_0^{(2)})(y) I_0^{(2)}(y)  \bigg) \\
 &&-  \frac{1}{ 10 } 2 F_0 \bigg(      \int_{-\infty}^t (I_0^{(2)})^2(y)   -2   \int_{-\infty}^t I_0^{(3)}(y) I_0^{(1)}(y)dy\bigg),
  \end{eqnarray*}
  where $I_a^{(k)}(y):= \frac{d^k}{dy^k}(I_a(t)|_{t\mapsto y})$, and in last identity we used Lemma \ref{S-series}. Furthermore,  using   $$    F_0 
  =\frac{5}{2}\big( J_1J_2 - J_3\big)-\frac{5}{6}T^3,
  \qquad 
  \partial_T F_0
 =5 J_2 -5\cdot \frac{T^2}{2}, \qquad\partial_T^2 F_0=5J_2'-5T,  
 $$
  one easily calculates
 \begin{eqnarray*}
 &&Cont^+
   =- \frac{1}{2} \   \partial_T J_2 \  \int_{-\infty}^t (I_1^{(2)}-TI_0^{(2)})^2dy  
 - \frac{1}{ 2} 2   J_2 \bigg(  \int_{-\infty}^t           (I_1^{(2)}-TI_0^{(2)}) I_0^{(2)} -  (I_1^{(3)}-TI_0^{(3)}) I_0' dy    \bigg)      \\
 &&-  \frac{1}{ 2}  \partial_T^{-1}(2J_2)     \int_{-\infty}^t (I_0^{(2)})^2   -2    I_0^{(3)} I_0^{(1)}dy 
+ \frac{1}{2} \   T \  \int_{-\infty}^t (I_1^{(2)})^2  dy  -  \ \frac{1}{ 2 }    T^2       (I_1''(t)-\frac{T}{3}I_0''(t))I_0'(t).
  \end{eqnarray*} \black

\subsection{Expression of   $Cont^0$    by $I$-functions}

 In this subsection   we use $I_r$ to mean Givental's  $I$-functions $I_r(t)$,  which is a power series in $q=e^t$ as usual. The difference of the $I_r$ notation from before is just the variable $y$ changed to variable $t$ here.
  \begin{align*}
   & Cont_{II}^0+Cont_{I}^0 
                                            =   - \frac{1}{2 }       \frac{  \ft^5}{5}      \sum_{d_1,d_2=1}^\infty \frac{(-1)^{d_1+d_2} }{(d_1+d_2)^2} \int_{ Q }     B_{d_1}(h) \cdot B_{d_2}(h)     \frac{1}{\psi_1} \frac{1}{\psi_2}        q^{d_1+d_2} \nonumber\\
                                           &  \cdot\bigg(  1+                                        (d_1+d_2) \frac{ (T-t)h}{h+\ft}  +                                           \frac{1}{2} (d_1+d_2)^2(T-t)^2  \frac{  h^2}{\ft^2}      \bigg) 
                                                +  \frac{ 1}{2} (T-t)^2    (I'_0)^2 - \frac{1}{12 }      (T-t)^3    \partial_t  \big( I'_0\cdot I'_0 \big).
                                                                       \end{align*}
 Let us calculate the three terms appeared above respectively.
 \begin{align*}
& - \frac{1}{2 }       \frac{  \ft^5}{5}      \sum_{d_1,d_2=1}^\infty \frac{(-1)^{d_1+d_2} }{(d_1+d_2)^2} \int_{ Q }     B_{d_1}(h) \cdot B_{d_2}(h)     \frac{1}{\psi_1} \frac{1}{\psi_2}    q^{d_1+d_2} \cdot    \bigg(    \frac{1}{2} (d_1+d_2)^2(T-t)^2  \frac{  h^2}{\ft^2}   \bigg)  \frac{(h+\ft)^2}{(h+\ft)^2}\nonumber \\
=&   -       \frac{  \ft^5}{20}   (T-t)^2    \sum_{d_1,d_2=1}^\infty (-1)^{d_1+d_2} \int_{ Q }         \frac{ B_{d_1}(h) \cdot B_{d_2}(h) }{\psi_1 \psi_2}     \frac{1}{ (h+\ft)^2}    q^{d_1+d_2} \cdot         h^2     \nonumber \\
 -&       \frac{  \ft^5}{10}   (T-t)^2    \sum_{d_1,d_2=1}^\infty (-1)^{d_1+d_2} \int_{ Q }         \frac{ B_{d_1}(h) \cdot B_{d_2}(h) }{\psi_1 \psi_2}     \frac{1}{ (h+\ft)^2}    q^{d_1+d_2} \cdot        \frac{  h^3}{\ft}     \nonumber \\
=&   -       \frac{  \ft^5}{100}    (T-t)^2 \sum_{j=0}^1  \sum_{d_1,d_2=1}^\infty (-1)^{d_1+d_2} \int_{ [W_{\ee_1}]\virt \times [W_{\ee_2}]\virt }   h_1^{j+2}h_2^{3-j}      \frac{ B_{d_1}(h_1) \cdot B_{d_2}(h_2)  }{\psi_1^2 \psi_2^2}       q^{d_1+d_2}         \nonumber \\
 -&      \frac{  1}{2}   (T-t)^2    \sum_{d_1,d_2=1}^\infty (-1)^{d_1+d_2}d_1d_2  b_0(d_1)b_0(d_2)          q^{d_1+d_2}           \nonumber \\
 =&    -       \frac{  \ft^5}{100}    (T-t)^2 \sum_{j=0}^1     \frac{5^2}{\ft^{10}}   (S_{1-j,0} S_{j,0})|_{T\mapsto t}      -         \frac{  1}{2}   (T-t)^2   (I_0')^2 
    =        \frac{ 1}{2}   (T-t)^2    \Big(    \  \big(I_1''-tI_0''\big)I_0'     -      (I_0')^2\Big). 
  \end{align*}
Similar calculations give 
  \begin{align*}
&   - \frac{1}{2 }       \frac{  \ft^5}{5}      \sum_{d_1,d_2=1}^\infty \frac{(-1)^{d_1+d_2} }{(d_1+d_2)^2} \int_{ Q }     B_{d_1}(h) \cdot B_{d_2}(h)     \frac{1}{\psi_1} \frac{1}{\psi_2}         q^{d_1+d_2} \cdot   \bigg(   (d_1+d_2) \frac{ (T-t)h}{h+\ft}     \bigg)  \nonumber \\
=& -    \frac{ 1}{2}          (T-t)    \partial_t^{-1} \bigg(\big(I_1''(t)-tI_0''(t)\big)^2+2\big( I_2'''(t)-tI_1'''(t)+\frac{t^2}{2}I_0'''(t) \big) I_0'(t)  \bigg) \\
&+  (T-t)    \partial_t^{-1}  \  \big(I_1''(t)-tI_0''(t)\big)I_0'(t),\nonumber 
 \end{align*}
\begin{align*}
&  - \frac{1}{2 }       \frac{  \ft^5}{5}      \sum_{d_1,d_2=1}^\infty \frac{(-1)^{d_1+d_2} }{(d_1+d_2)^2} \int_{ Q }     B_{d_1}(h) \cdot B_{d_2}(h)     \frac{1}{\psi_1} \frac{1}{\psi_2}         q^{d_1+d_2} \cdot   \\
  =&\partial_t^{-2}   \bigg( \big(I_1''-tI_0''\big) I_0' \bigg) 
  -      \partial_t^{-2} \bigg(
 \big(I_1''-tI_0''\big)^2
+2\big( I_2'''-tI_1'''+\frac{t^2}{2}I_0''' \big) I_0'
  \bigg)  \nonumber \\
    +&
                \partial_t^{-2}    \bigg(   \big(I_3''''-tI_2''''+\frac{t^2}{2}I_1''''-\frac{t^3}{6}I_0'''' \big)   I_0' +    \big( I_2'''-tI_1'''+\frac{t^2}{2}I_0''' \big) \big( I_1''-tI_0'' \big) \bigg),    \end{align*}

where $I_k,I_k',I_k'',\cdots $ denotes $I_k(t),I_k'(t),I_k''(t),\cdots$.
   We then conclude
     \begin{align*}
 & Cont_{II}^0+Cont_{I}^0 
  =   - \frac{1}{12 }      (T-t)^3    \partial_t  \big( I'_0\cdot I'_0 \big) 
  +      \frac{ 1}{2}    (T-t)^2       \  \big(I_1''-tI_0''\big)I_0'    \\
  & -    \frac{ 1}{2}          (T-t)    \partial_t^{-1} \bigg(
 \big(I_1''-tI_0''\big)^2
+2\big( I_2'''-tI_1'''+\frac{t^2}{2}I_0''' \big) I_0'
  \bigg) +  (T-t)    \partial_t^{-1}    \bigg( \big(I_1''-tI_0''\big)I_0' \bigg) \nonumber \\
  &+ \partial_t^{-2}   \bigg( \big(I_1''-tI_0''\big) I_0' \bigg) 
  -        \partial_t^{-2} \bigg(
 \big(I_1''-tI_0''\big)^2
+2\big( I_2'''-tI_1'''+\frac{t^2}{2}I_0''' \big) I_0'
  \bigg)  \nonumber \\
    &+
                \partial_t^{-2}    \bigg(   \big(I_3''''-tI_2''''+\frac{t^2}{2}I_1''''-\frac{t^3}{6}I_0'''' \big)   I_0' +    \big( I_2'''-tI_1'''+\frac{t^2}{2}I_0''' \big) \big( I_1''-tI_0'' \big) \bigg).  
    \end{align*}

   One can combine a few terms to write in terms of total derivatives,
        \begin{align*}
  Cont^0:=Cont_{II}^0+Cont_{I}^0 
     =  &     - \frac{1}{12 }      (T-t)^3    \partial_t  \big( I'_0\cdot I'_0 \big)+  \frac{ 1}{2}    (T-t)^2       \  \big(I_1''-tI_0''\big)I_0'   \\
    - &    \frac{ 1}{2}          (T-t)    \partial_t^{-1} \bigg(\big(I_1''-tI_0''\big)^2+2\big( I_2''-tI_1''+\frac{t^2}{2}I_0'' \big)' I_0'\bigg)     \nonumber   \\
    +&      \partial_t^{-2}    \bigg(   \big(I_3'''-tI_2'''+\frac{t^2}{2}I_1'''-\frac{t^3}{6}I_0''' \big)'   I_0' +    \big( I_2''-tI_1''+\frac{t^2}{2}I_0'' \big)' \big( I_1'-tI_0' \big)' \bigg)   \nonumber.  
    \end{align*}

    After simplifications one has
            \begin{align*}
         Cont^0      =  
  &            \big( \frac{T^2}{2}I_1'' -\frac{T^3}{6}I_0'' \big)I_0'(t) - \frac{1}{2} T \partial_t^{-1} (I_1'')^2    -(T-t)\partial_t^{-1}(I_2'''I_0' ) \\
  &+ \partial_t^{-2}\Big(  \big(I_3'''-tI_2'''  \big)'   I_0' +     I_2''' \big( I_1'-tI_0' \big)'  \Big). \\  \nonumber
                  \end{align*}

  Combining previous subsection and   using  Proposition \ref{lemmadesiredidentity},   the type $B$ contribution is
     \begin{align}\label{simp}
    & Cont^++Cont^0\\
    =&  - \frac{1}{2} \   \partial_T J_2 \  \int_{-\infty}^t (I_1^{(2)}-TI_0^{(2)})^2dy  -   J_2 \Big(  \int_{-\infty}^t           (I_1^{(2)}-TI_0^{(2)}) I_0^{(2)} -  (I_1^{(3)}-TI_0^{(3)}) I_0^{(1)} dy    \Big)   \nonumber    \\
 &-     \partial_T^{-1}( J_2)     \int_{-\infty}^t (I_0^{(2)})^2   -2    I_0^{(3)} I_0^{(1)}dy \nonumber  + \partial_t^{-2}\Big(  \big(I_3'''-tI_2'''  \big)'   I_0' +     I_2''' \big( I_1'-tI_0' \big)'  \Big)   -(T-t)\partial_t^{-1}(I_2'''I_0' )\nonumber\\
  = & \frac{1}{2}\frac{T''}{T'} + 2 \frac{I_0'}{I_0}+ \frac{\ln(1-5^5e^t)}{5}. \nonumber
    \end{align}

 \section{Graph Type C}

 A type $C$ graph $\ga$ is of the following form:

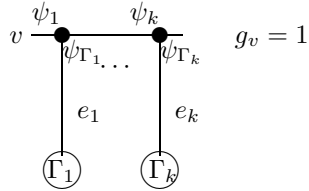
\begin{center}
\begin{picture}(20,28)

\put(-11,23){\circle*{2}}
\put(-6,18){$\cdots$}
\put(-13,4){$\Gamma_1$}
\put(-15,25){$\psi_1$}
 \put(-11,22){\line(0,-1){15}}
\put(-11,20){$\psi_{\ga_1}$}
\put(-9,12){$e_1$}
\put(-11,5){\circle{5}}

\put(2,23){\circle*{2}}
\put(0.5,4){$ \Gamma_k$}
\put(2,20){$\psi_{\ga_k}$}
\put(4,12){$e_k$}
\put(-2,25){$\psi_k$}
 \put(2,22){\line(0,-1){15}}
\put(2,5){\circle{5}}

 \put(-15,23){\line(1,0){20}}

\put(-18, 22){$v$}
\put(12,22){$g_v=1$}


\end{picture}
\end{center}
 
Separating all black nodes, $\ga$ is decomposed into $v$
and rational tails $\{{\ga_i}\}_{i=1}^k$.  Let    $d_i=\deg\ga_i$. Then $d_1+d_2+\cdots + d_k = n .$
 Thus
the contribution of $\ga$ to \eqref{PI} 
 is \black 
  \begin{eqnarray*}
&& \frac{\ft\prod|Aut\ga_i|}{|Aut\ga|}  \sum_{a+\sum a_i=n-1}\int_{[\cW_{v}]\virt} \frac{c_{a}\big( R\pi_{\cW_{v}\ast}(\cL_{\cW_{v}}   \cL_\ft)\big) }{e(N\virt_{\cW_{v}})}    \prod_{i=1}^k \int_{ [\cW_{(\ga_i)}]\virt} \frac{-\ft^4c_{a_i}\big( \pi_{\cW_{\ga_i}\ast}(\cL_{\cW_{\ga_i}}   \cL_\ft (-D_i) )\big)}{5(-\psi_{\ga_i}-\psi_i)e(N\virt_{\cW_{(\ga_i)} })}   \\
 &=&  \frac{ \prod|\Aut\ga_i|}{|\Aut\ga|}   \frac{\ft^{4k+1}}{5^k}  \sum_{a=0}^{n-1} \int_{[\cW_{v}]\virt} \frac{c_{a}\big( \pi_{\cW_{v}\ast}(\cL_{\cW_{v}}   \cL_\ft)\big) }{e(N\virt_{\cW_{v}})}  \\
 &&\cdot  {\rm Coe}_{(\ft')^{a+1}}   \prod \int_{ [\cW_{(\ga_i)}]\virt}  \frac{e\big( \pi_{\cW_{\ga_i}\ast}(\cL_{\cW_{\ga_i}}     \cL_\ft \cL_{\ft'} (-D_i)   )\big)}{ (\psi_{\ga_i}+\psi_i)e(N\virt_{\cW_{(\ga_i)} })} .\end{eqnarray*}


 Since $\cW_{v}=\barM_{1,k} $ and, over its universal curve, the section $\nu_1$ trivializes $ 
 \cL_{\cW_{v}} \cL_\ft $,   $R\pi_{\cW_{v}\ast} ( \cL_{\cW_{v}} \cL_\ft  ) =  R\pi_{\cW_{v}\ast}\sO_{\cW_{v}}$  and 
 $c( R\pi_{\cW_{v}\ast} ( \cL_{\cW_{v}} \cL_\ft  ) )=\frac{1}{1-\lam}=1+\lam $. Also  by   \eqref{eqn:Av}, over $[\cW_{v}]\virt$ one has 
  $$\frac{1}{e(N\virt_{\cW_{(v)} })}=\left(\frac{\lam+\ft}{\ft}\right)^5 \cdot \left( \frac{5\ft}{\lam+5\ft}\right)= \left(1+\frac{5\lam}{\ft}\right)\left(1-\frac{\lam}{5\ft}\right)
  =1+\frac{24}{5}\frac{\lam}{\ft}.$$
  Let $\Delta_C$ be the subset  of  type $C$ graphs with degree $(d,0)$  for $d=1,2,\cdots$. 


      Then the total contribution of type $C$ graphs is, using \eqref{hatP}, 
    \begin{eqnarray}\label{contriC} 
&&   {\rm Contri}(C)
= \sum_{k=1}^\infty  \frac{1}{k!}  \left(\frac{\ft^4}{5}\right)^k    \sum_{a=0}^{1} \int_{\barM_{1,k}}  c_{a}\big( R\pi_{\cW_{v}\ast}(\cL_{\cW_{v}}   \cL_\ft)\big) \left(\ft+\frac{24}{5}\lam\right)   {\rm Coe}_{(\ft')^{a+1}}  \prod_{i=1}^k\hat\sP(\psi_i) \nonumber\\
 &=&
  \sum_{k=1}^\infty  \frac{1}{k!}  \left(\frac{\ft^4}{5}\right)^k    \int_{\barM_{1,k}} \left[    \frac{24}{5}\lam   {\rm Coe}_{(\ft')^{1}} \prod_{i=1}^k\hat\sP(\psi_i) 
  +   \lam   \ft   {\rm Coe}_{(\ft')^{2}} \prod_{i=1}^k\hat\sP(\psi_i)  +    \ft  {\rm Coe}_{(\ft')^{1}}  \prod_{i=1}^k\hat\sP(\psi_i)\right]  \nonumber .
     \end{eqnarray}
   Since   $\cZ_6\sta:=\frac{\ft^4}{5}\ti\sZ_{1,5}(\hbar,\sQ)$      is regularizable(c.f. Lemma \ref{Z6reg}),  by  \eqref{hatP}  we have 
\begin{eqnarray*} &&  \sum_{k=1}^\infty  \frac{1}{k!}  (\frac{\ft^4}{5})^k    \int_{\barM_{1,k}}   \lam   \prod_{i=1}^k\hat\sP(\psi_i)
 = \frac{1}{24} \sum_{k=1}^\infty  \frac{(-1)^{k-1}}{k!}  (\frac{\ft^4}{5})^k    \sum_{\substack{  \sum_{\ell=1}^k  a_\ell
 =k-1\\a_\ell\geq 0}}    \binom{k-1}{a_1,\cdots,a_k}      \prod_{\ell=1}^k    \ti\sP_{1+a_\ell} \\
    &=&     \frac{1}{24} \sum_{k=1}^\infty  \frac{1}{k} \black     \sum_{\substack{ \sum_{\ell=1}^k  a_\ell=k-1\\a_\ell\geq 0 }}         \prod_{\ell=1}^k \frac{(-1)^{a_\ell} }{a_\ell !}
    Res_{\hbar=0} \{ \hbar^{-a_\ell}  \sZ_6\sta(\hbar,\sQ)  \} 
    =  \frac{1}{24} Res_{\hbar=0}\ln\big(1+Z^*_6(\hbar, \sQ)\big).
 \end{eqnarray*}

 Using  \eqref{LemmX} in Appendix \ref{secregularityandsomeidentities},  the last term in \eqref{contriC} becomes
$$\sum_{k=1}^\infty  \frac{1}{k!}  \left(\frac{\ft^4}{5}\right)^k   \int_{\barM_{1,k}}   \ft  {\rm Coe}_{\ft'} \hat\sP(\psi_1)\hat\sP(\psi_2)\cdots\hat\sP(\psi_k)
= -\frac{\ft}{24}{\rm Coe}_{\ft'}{\rm Res}_{\hbar=0} \bl \ln(1+\bar{\sZ}_6\sta)) /\hbar \br.
$$

Thus, the type $C$ contribution is
    \begin{eqnarray*}
  {\rm Contri}(C) =  (\frac{24}{5}{\rm Coe}_{\ft'}+\ft {\rm Coe}_{(\ft')^{2}}) \frac{1}{24} {\rm Res}_{\hbar=0}\ln(1+\sZ_6\sta(\hbar))      -\frac{\ft}{24} {\rm Coe}_{\ft'}  {\rm Res}_{\hbar=0} \bl \hbar^{-1} \ln(1+ {\sZ}_6\sta))  \br.  \end{eqnarray*}

 Using \eqref{4coe}, the type $C$ contribution is
\begin{eqnarray*}
&&(\frac{24}{5}{\rm Coe}_{\ft'}+\ft {\rm Coe}_{(\ft')^{2}})\frac{\eta}{24}      -\frac{\ft}{24} {\rm Coe}_{\ft'}  {\rm Res}_{\hbar=0} \bl \hbar^{-1} \log(1+ {\sZ}_6\sta))  \br \\
&= &  - \frac{1}{5}\frac{  g_1}{  I_0} - \frac{1}{5} \ln(1-5^5q) -\frac{1}{12}  \cdot \frac{5^5q}{1-5^5q}.  \nonumber
\end{eqnarray*}

 \section{Graph Type D}
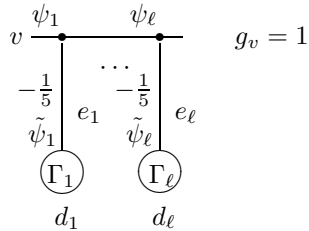
 \begin{center}
\begin{picture}(20,28)

\put(-11,23){\circle*{1}}
\put(-6,18){$\cdots$}
\put(-13,3.5){$\Gamma_1$}
\put(-12,-2){$ d_1$}
\put(-15,25){$\psi_1$}
 \put(-11,22){\line(0,-1){14}}
\put(-17,15){$-\frac{1}{5}$}
\put(-9,12){$e_1$}
\put(-15.5,9){$\ti\psi_1$}
\put(-11,5){\circle{6}}

\put(2,23){\circle*{1}}
\put(0.5,3.5){$\Gamma_\ell$}
\put(1,-2){$ d_\ell$}
\put(4,12){$e_\ell$}
\put(-2,25){$\psi_\ell$}
 \put(2,22){\line(0,-1){14}}
\put(-4,15){$-\frac{1}{5}$}
\put(-2.5,9){$\ti\psi_\ell$}
\put(2,5){\circle{6}}

 \put(-15,23){\line(1,0){20}}

\put(-18, 22){$v$} 
\put(12,22){$g_v=1$}

\end{picture}
\end{center}

\medskip

Let $d_i=\deg\Gamma_i$ and $d=d_1+\ldots+d_\ell$. 
Each $\Gamma_i$ is in $\Xi^1_{d_i}$.  The $\ti\psi_i$ denotes the $\psi$-class of the marking on $\sC_{\Gamma_i}$ and $\psi_i$ is the $\psi$-class of the $i$-th marking on the curve $\sC_v$. \black The contribution from type D graphs is
\begin{eqnarray}\label{integrand}
&&\sum_{\ell=1}\sum_{d_i>0, 1\le i\le \ell} \sum_{(\ga_i) \in \hat\Xi^1_{d_i}  } \frac{(- \sQ)^{d_1+\ldots +d_\ell}}{|\Aut(\{\Gamma_1, \ldots, \Gamma_\ell\})|}\int_{\barM_{1,\ell}^{1/5,5p}\times [\cW_{(\Gamma_1)}]^{vir}\times \cdots \times  [\cW_{(\Gamma_\ell)}]^{vir}}
\ft\cdot \prod_{i=1}^\ell\frac{1}{-\ft-\psi_i} 
\\
&&  \cdot \frac{1}{e_T\big(R\pi_v(\cL_v^\vee \cL_\ft)\big)}\cdot\frac{ (-\ft^5)^\ell}{ (5\ft)^{\ell}}\cdot\prod_{i=1}^\ell\left(\frac{1}{5\ft- \ti\psi_i}\cdot \frac{1}{e_T(N_{\cW_{ \Gamma_1}/\cW}^{vir})}\right)
\cdot\left(\frac{1}{1-\lambda}c \big(R\pi_{\cW_{\ga}\ast}(\cL_{\cW_{\ga}} \cN_{\cW_{\ga}} \cL_\ft)\big)\right)_{d-1},\nonumber
\end{eqnarray}
where $\barM_{1,\ell}^{1/5,5p}$ is the moduli  of   $5$-spin curves of genus $1$ with five $P$-fields and 
 $(\cdots)_{d-1}$ is the degree $d-1$ component. 

Using the variable $\ft'$, we can rewrite  $\left(\displaystyle\frac{1}{1-\lambda}c \big(R\pi_{\cW_{\ga}\ast}(\cL_{\cW_{\ga}} \cN_{\cW_{\ga}} \cL_\ft)\big)\right)_{d-1}$ as

$$({\rm Coe}_{\ft'}+\lambda{\rm Coe}_{\ft^{\prime 2}})
\prod_{i=1}^\ell e\big(R\pi_{\cW_{\ga_i}\ast}(\cL_{\cW_{\ga_i}} \cL_\ft\cL_{\ft'})\big).$$ \black

Let $h_\ell\colon \barM_{1,1^\ell}^{1/5, 5p}\to \barM_{1,1^{\ell-1}}^{1/5, 5p}$ be the contracting map forgetting the last marking point.  Let $D_i$ be the boundary divisor of  $\barM_{1,1^\ell}^{1/5, 5p}$ which is the graph of the section of $h_\ell$ induced by $\ell$-th marking. For $i\le \ell-1$, we have the identity $\psi_i^k=h_\ell\sta\psi_i^k+(D_i/5)h_\ell\sta\psi^{k-1}_i$. This implies
\begin{eqnarray*}
&&\frac{1}{\ft+\psi_i}
=\frac{1}{\ft}\Big(1-\frac{h_\ell^*\psi_i+D_i/5}{\ft}+\frac{h_\ell^*(\psi_i)^2+h_\ell^*\psi_i\cdot D_i/5}{\ft^2}+\ldots\Big)
= \frac{1-D_i/5\ft}{\ft+h_\ell^*\psi_i}.
\end{eqnarray*}
 We denote $ {e_T(R\pi_{v*}ev^*\cL^\vee\cL_\ft)}^{-1}=\ee_{\ell}$. It is direct to check $\ee_\ell=(-\ft-\psi_\ell)h_\ell\sta \ee_{\ell-1}$. Therefore
\begin{eqnarray*}
&&\int_{\barM_{1,1^\ell}^{1/5, 5p}}\frac{1}{-\ft-\psi_1}\cdots\frac{1}{-\ft-\psi_\ell}
\ee_\ell 
=\int_{\barM_{1,1^\ell}^{1/5, 5p}}\frac{1-D_1/5\ft}{-\ft-h_\ell^*\psi_1}\cdots\frac{1-D_{\ell-1}/5\ft}{-\ft-h_\ell^*\psi_{\ell-1}}
 h_\ell\sta\ell_{\ell-1}\\
&=&-\frac{\ell-1}{5\ft}\int_{\barM_{1,1^{\ell-1}}^{1/5, 5p}}\frac{1}{-\ft-\psi_1}\cdots\frac{1}{-\ft-\psi_{\ell-1}}
\ee_{\ell-1}
=(-1)^{\ell-1}\frac{(\ell-1)!}{(5\ft)^{\ell-1}}\int_{\barM_{1,1}^{1/5, 5p}}\frac{1}{-\ft-\psi_1}\ee_1\\
&=&(-1)^{\ell}\frac{(\ell-1)!}{(5\ft)^{\ell}}\cdot \frac{128}{3}.
\end{eqnarray*}
Here we used the fact that $D_i\cdot D_j=0$ if $i\neq j$ and a reduction Lemma about FJRW invariants.

 Using $\lambda=5\psi_1$ and a computation of a FJRW invariant, we obtain similarly \begin{eqnarray*}
\int_{\barM_{1,1^\ell}^{1/5, 5p}}\frac{1}{-\ft-\psi_1}\cdots\frac{1}{-\ft-\psi_\ell}
\ee_\ell
=(-1)^{\ell-1}\frac{(\ell-1)!}{(5\ft)^{\ell-1}}\int_{\barM_{1,1}^{1/5, 5p}}\frac{-1}{\ft}\Big(1-\frac{\psi_1}{\ft}\Big)
 5\psi_1\ee_1=(-1)^{\ell}\frac{(\ell-1)!}{(5\ft)^{\ell-1}}\cdot \frac{25}{3}.
\end{eqnarray*}


Using \eqref{4coe}, the  type D contribution \eqref{integrand}   is
 \begin{eqnarray*}
&&\sum_{\ell=1}\sum_{d_i>0, 1\le i\le \ell}\sum_{(\ga_i) \in \hat\Xi^1_{d_i}  } (- \sQ)^{d_1+\ldots +d_\ell}\frac{\ft}{|\Aut(\{\Gamma_1, \ldots, \Gamma_\ell\})|}\cdot  \frac{(-\ft^5)^\ell}{ (5\ft)^{\ell}}\bigg((-1)^{\ell}\frac{(\ell-1)!}{(5\ft)^{\ell}}\cdot \frac{128}{3}{\rm Coe}_{\ft'}\nonumber\\
&&+(-1)^{\ell}\frac{(\ell-1)!}{(5\ft)^{\ell-1}}\cdot \frac{25}{3}{\rm Coe}_{\ft^{\prime2}}\bigg)
\prod_{i=1}^\ell\Big(\int_{ [\cW_{(\Gamma_i)}]^{vir}}\frac{1}{5\ft- \ti\psi_i}\frac{e(R\pi_{\Gamma_i*}\cL_{\Gamma_i}(-D_{\Gamma_i})\otimes \cL_\ft\cL_{\ft'})}{ e_T(N_{\cW_{ \Gamma_i}/\cW}^{vir})}\Big)\\
&=&\ft\cdot \frac{128}{3}\sum_{\ell=1}^\infty\frac{1}{\ell}\Big(\sum_{d=1}^\infty\sum_{\Gamma\in \Sigma_d^1 }(- \sQ)^{d}
\cdot \frac{\ft^4}{5}\cdot\int_{ [\cW_{(\Gamma)}]^{vir}}\frac{1}{5\ft\cdot (5\ft- \ti\psi)}\frac{e(R\pi_{\ga\ast}(\cL_\ga \cL_{\ft} \cL_{\ft'}(-D_\ga)))}{ e_T(N_{\cW_{ \Gamma}/\cW}^{vir})}\Big)^\ell|_{\ft'}\\
 &&+ \frac{125}{3}\ft^2
\cdot\sum_{\ell=1}^\infty\frac{1}{\ell} \Big(\sum_{d=1}^\infty\sum_{\Gamma\in \Sigma_d^1 }(- \sQ)^{d}
\frac{\ft^4}{5}\cdot\int_{ [\cW_{(\Gamma)}]^{vir}}\frac{1}{5\ft\cdot (5\ft- \ti \psi)}\frac{e(R\pi_{\ga\ast}(\cL_\ga \cL_{\ft} \cL_{\ft'}(-D_\ga)))}{ e_T(N_{\cW_{ \Gamma}/\cW}^{vir})}\Big)^\ell
|_{\ft^{\prime2}}\\
&=&-\ft\cdot \frac{128}{3} {\rm Coe}_{\ft'}\ln(1+\frac{\ft^4}{5}\tilde\sZ^*_{1, 5}(5\ft))-\frac{125}{3}\ft^2 {\rm Coe}_{\ft^{\prime 2}}\ln(1+\frac{\ft^4}{5}\tilde \sZ^*_{1,5}(5\ft))\\
&=&-\frac{128}{3}\ft\left(-\frac{g_1}{5\ft I_0}\right)=\frac{128}{3}\cdot\frac{g_1}{5I_0}.
\end{eqnarray*}

\section{A+B+C+D}

  Let's recall the results from previous sections about the contributions of graph type A, B, C, and D. 
 
 \begin{eqnarray*}\text{Graphs A:  } &&  \sF_1'  -\frac{25}{12}\cdot  (T'-1)
 \ +  \frac{40} {24 } \cdot 5  \frac{I_0' }{I_0}    -\frac{25}{3}\cdot\frac{g_1}{I_0}; \\
 \text{Graphs B:  } &&2 \frac{I_0'}{I_0}          + \frac{ \ln(1-5^5q)}{5} + \frac{1}{2}\frac{T''(t)}{T'(t)};\\
\text{Graphs C:  } && - \frac{1}{5}\frac{  g_1}{  I_0} - \frac{1}{5} \ln(1-5^5e^t) -\frac{1}{12}  \cdot \frac{5^5e^t}{1-5^5e^t};\\ \text{Graphs D:  } && \frac{128}{3} \frac{  g_1}{5 I_0} .  
    \end{eqnarray*}
    Here $\displaystyle{\frac{d}{dt}}$ is replaced by ${}^\prime$.
    
    Taking the summation of the four contributions above, we get  
    $$ \sF_1' -\frac{25}{12}\cdot  (T'-1)
  +  \frac{31} {3 }  \frac{I_0' }{I_0}   +
\frac{1}{2}\frac{T''(t)}{T'(t)}
 -\frac{1}{12}  \cdot \frac{5^5e^t}{1-5^5e^t}=0.  $$
 
  This implies the formula for $\sF_1=\sum_{d=1}^\infty N_{1,d}e^{Td}$ as follows:
  \beq\label{Zinger} \sF_1 =\frac{25}{12}\cdot  (T-t)
- \ln\Big(   {I_0 }^{\frac{31}{3}}  \cdot
{T'(t)}^{\frac{1}{2}}
 \cdot{(1-5^5e^t)}^{ \frac{1}{12} }\Big) \eeq
since both side of \eqref{Zinger} are power series in $q$ with no $q^0$ term. It   is exactly what A. Zinger obtained in \cite{Zi2}. \black

\black

\begin{appendices}
 \section{Appendix: Regularity and some identities}
 \label{secregularityandsomeidentities}

\subsection{Regularity} We prove all the results involving regularities of series here.

Recall the definition of regularizable power series introduced by Zinger in \cite{Zi2}. Let
$\mathbb Q_\ft=\mathbb Q[\ft, \ft^{-1}, \ft']$. 
\begin{defi}  A power series $\cZ^*=\cZ^*(\hbar, q)\in \mathbb Q_{\ft}(\hbar)[[q]]$ is regularizable at $\hbar=0$ if there exist power series 
$$\eta=\eta(q)\in \mathbb Q_\ft[[q]] \quad {\rm and}\quad \bar\cZ^*=\bar\cZ^*(\hbar, q)\in \mathbb Q_{\ft}(\hbar)[[q]]$$ with no degree-zero term in $q$ such that $\bar\cZ^*$ is regular at $\hbar=0$ and 
$$1+\cZ^*(\hbar, q)=e^{\eta(q)/\hbar}\big(1+\bar\cZ^*(\hbar, q)\big).$$
\end{defi}    

If $\cZ^*$ is regularizable at $\hbar =0$, then $\cZ^*$ has no degree-zero term in $u$ and the regularizing pair ($\eta, \bar\cZ^*$) is unique. It is determined by 
$
\eta(u)={\rm Res}_{\hbar=0}{\ln\big(1+\cZ^*(\hbar, u)\big) }.
$

\begin{lemm}
$\displaystyle\frac{\ti \sZ_{0,1}(\ft, -\ft)}{5}\in  \mathbb Q_{\ft}(\hbar)[[q]]$ is regularizable. 
\end{lemm}
\begin{proof} From \eqref{ghost0}, we have 
\begin{eqnarray*}
1+\frac{\ti \sZ_{0,1}(\ft, -\ft)}{5}=e^{\frac{g_1\ft}{I_0\hbar}}\cdot \frac{1}{I_0}\cdot\sum_{d=0}^\infty q^d\frac{\hbar^d(5d)!}{(d!)^4\prod\limits_{m=1}^d(\ft+m\hbar)}:= e^{\frac{g_1\ft}{I_0\hbar}} (1+\overline \sZ_{0,1}).
\end{eqnarray*}
 Then $\displaystyle\frac{g_1 \ft}{I_0}$ and $\overline \sZ_{0,1}$
 has no degree-zero terms in $q$, and $\overline \sZ_{0,1}$ is regular at $\hbar=0$.
\end{proof}

\begin{lemm}\label{Z6reg} $\cZ_6\sta:=\frac{\ft^4}{5}\ti\sZ_{1,5}(\hbar,\sQ)\in \QQ(\ft,\ft')(\hbar)[[\sQ]]$ is regularizable at $\hbar =0$ \end{lemm} \begin{proof} The argument in \cite[lemm 2.3]{Zi2} applies here. 
\end{proof}
  
 Thus we can express 
     \beq\label{reg}1+\sZ_6\sta(\hbar)=e^{\eta/\hbar}(1+\bar\sZ_6\sta(\hbar)),\qquad \quad
    \eta={\rm Res}_{\hbar=0}\ln(1+\sZ_6\sta(\hbar)), \eeq
    where $\bar\sZ_6\sta\in \QQ(\ft,\ft')(\hbar)[[\sQ]]\sQ$ is regular at $\hbar=0.$

  Mimicking \cite[Lemm 2.2(i)(ii)]{Zi2}, one can show
\begin{lemm}
If $(\eta, \overline{Z}^*)$ is the regularizing pair for $Z^*$ at $\hbar=0$, then for every $a\ge 0$, 
\begin{eqnarray}\label{formula1}
\sum_{m=1}^\infty \frac{1}{m}\sum\limits_{\sum\limits_{\ell=1}^m a_\ell=m-1-a\atop a_\ell\ge 0}
 \Big(\prod_{\ell=1}^m \frac{(-1)^{a_\ell}}{a_\ell !} {\rm Res}_{\hbar=0}\{\hbar^{-a_\ell}Z^*(\hbar, q)\}\Big)=
\frac{\eta^{a+1}}{a+1}.\end{eqnarray}
\end{lemm}

 \subsection{Some identities $I$}
   Using methods in  Zinger's \cite[proof of (2.4)]{Zi2}, we want to calculate   \begin{eqnarray*}
  \bf{X}&:=&  \sum_{m=1}^\infty  \frac{1}{m!}  (\frac{\ft^4}{5})^m    \int_{\barM_{1,m}}      \hat\sP(\psi_1)\hat\sP(\psi_2)\cdots\hat\sP(\psi_m)  \\
&=&   \sum_{m=1}^\infty  \frac{1}{m!} \black  \sum_{\vec{a}=(a_1,\cdots,a_m)}    \vec{a} !   \int_{\barM_{1,m}}\psi_1^{a_1}\cdots\psi_m^{a_m}    \prod_{\ell=1}^m \frac{(-1)^{a_\ell} }{a_\ell !}
    {\rm Res}_{\hbar=0} \{ \hbar^{-a_\ell}  \sZ_6\sta(\hbar,q)  \}.
    \end{eqnarray*}

     Writing   $\bar\cZ_6\sta=\sum_{m=0}^\infty C_m\hbar^m$,  one has
      $$  {\rm Res}_{\hbar=0}\hbar^a\cZ_6\sta=\sum_{\substack{p-m=1+a\\ p,m\geq 0}}
    \frac{\eta^p}{p!} C_m +
     \begin{cases}
   \displaystyle \frac{\eta^{a+1}}{(a+1)!} & a\geq 0,\\
   0 & \text{otherwise}.
           \end{cases}  $$

    For fixed $k\geq 0$, let  $\alpha=(\alpha_1,\cdots,\alpha_k)\in \NN^k$, where
     $N=\alpha_1+\cdots+\alpha_k:=|\alpha|$. Pick     a subset of $\ZZ_{\geq 0}   $
    \begin{eqnarray}\label{fixed2}
     \quad \quad m=\{m_1=\cdots=m_{\alpha_1}, m_{\alpha_1+1}=\cdots=m_{\alpha_1+\alpha_2},\cdots, m_{\alpha_1+\cdots+\alpha_{k-1}+1}=\cdots=m_{N}\} 
     \end{eqnarray}
     be a sequence of  N  distinct numbers. We look at the coefficient of $C_m^\alpha:=C_{m_1}^{\alpha_1}C_{m_{\alpha_1+\alpha_2}}^{\alpha_2}\cdots C_{m_N}^{\alpha_N}$ in ${\bf X}$.  For each $\beta:\{1,\cdots,N\}\to \ZZ_{\geq 0}$ and a choice of  $\alpha$ above, the term $C_m^\alpha$ appears in the $n=|\beta|:=\beta_1+\cdots+\beta_N$ summand in ${\bf X}$ with the coefficient 
   \begin{align*}
   & \eta^{n-N} \prod_\ell (\beta_\ell !) \int_{\barM_{1,|\beta|}}(\prod_{s=1}^N\psi_s^{\beta_s})
   \prod_{\ell=1}^N \frac{(-1)^{\beta_\ell}}{\beta_\ell !}\frac{\eta^{m_\ell+1-\beta_\ell}}{(m_\ell+1-\beta_\ell)!}\\
   =& \frac{\eta^{\alpha\cdot m}}{ (m+1)!^\alpha } \prod_\ell (\beta_\ell !) \int_{\barM_{1,|\beta|}}
   (\prod_{s=1}^N\psi_s^{\beta_s})
   \prod_{\ell=1}^N  (-1)^{\beta_\ell} \binom{m_\ell+1}{\beta_\ell},
\end{align*}
 where $\alpha\cdot m:=\sum_{\ell}m_\ell$ and $(m+1)!^\alpha:=\prod\limits_{\ell=1}^N (m_\ell+1)!$.  The number of such choices of $\alpha$ is  $\displaystyle\binom{|\beta|}{\alpha,|\beta|-\alpha}$.

   Set $b=|\beta|$ and sum over it. We then have
    \begin{eqnarray*} 
    {\rm Coe}_{C^\alpha_m}{\bf X}&=&
  \frac{\eta^{\alpha\cdot m}}{(m+1)!^\alpha}\sum_{b=0}^\infty \frac{(-1)^b}{b!} \binom{b}{\alpha,b-N} \sum_{\substack{\beta_1+\cdots+\beta_N=b\\ \beta_\ell\geq 0}} \prod_\ell (\beta_\ell!)
    \int_{\barM_{1,b}}\psi_1^{\beta_1}\cdots \psi_N^{\beta_N} \prod_{\ell=1}^N \binom{m_\ell+1}{\beta_\ell}\\
&=&    \frac{\eta^{\alpha\cdot m}}{(m+1)!^\alpha} \frac{N!}{\alpha!}\sum_{b=N}^\infty \frac{(-1)^b}{b!} \binom{b}{N}  \Lambda_b(m_1+1,\cdots,m_N+1)\nonumber,
 \end{eqnarray*}
where  for $b\geq N$, $r_1,\cdots,r_N\geq 0$,     we set 
      \begin{eqnarray*}
      \Lambda_b(r_1,\cdots,r_N) &:=& \sum_{\substack{\beta_1+\cdots+\beta_N=b \\  r_\ell\geq \beta_\ell\geq 0 }}
    \int_{\barM_{1,b}}\psi_1^{\beta_1}\cdots \psi_N^{\beta_N} \prod_{\ell=1}^N \frac{r_\ell!}{(r_\ell-\beta_\ell)!}.
\end{eqnarray*}

\begin{lemm} Set $\Lambda_b(r_1,\cdots,r_N):=0$ if some $r_i<0$.  Then  $\forall b\geq N$, the series $\Lambda_{b+1}(r_1,\cdots,r_N)$ is equal to 
\begin{eqnarray*} 
&&r_1\Lambda_{b}(r_1-1,r_2,\cdots,r_N)+r_2\Lambda_b(r_1,r_2-1,r_3,\cdots,r_N)+\cdots  +r_N \Lambda_b(r_1,r_2,\cdots,r_{N-1},r_N-1).
\end{eqnarray*}
\end{lemm}

    The identity is a simple consequence of the string equation. Using the Lemma, we obtain
      \begin{eqnarray*} &&  \frac{N!}{\alpha!}\sum_{b=N}^\infty \frac{(-1)^b}{b!} \binom{b}{N}  \Lambda_b(r_1,\cdots,r_N)
    = \frac{1}{\alpha!}\sum_{b=N}^\infty  \frac{(-1)^b}{(b-N)!} \Lambda_b(r_1,\cdots,r_N)\\
    &=& \frac{(-1)^N}{\alpha!}\sum_{b=0}^\infty  \frac{(-1)^b}{b!} \Lambda_{b+N}(r_1,\cdots,r_N)\\
&=& \frac{(-1)^N}{\alpha!}\sum_{0\leq a_i\leq r_i}  \frac{(-1)^{|r|-|a|}}{(|r|-|a|)!}    \binom{|r|-|a|}{r_1-a_1,\cdots,r_N-a_N}    \frac{r_1!}{a_1!}\cdots\frac{r_N!}{a_N!}      \Lambda_{N}(a_1,\cdots,a_N)    \\
 &=& \frac{(-1)^N}{\alpha!}   \sum_{\substack{\beta_1+\cdots+\beta_N=N \\ r_\ell\geq \beta_\ell\geq 0 }}    \sum_{\beta_i \leq a_i\leq r_i}   \bigg((-1)^{|r|-|a|}    \frac{1}{(r_1-a_1)!\cdots (r_N-a_N)!}  \\
 &&\cdot  \frac{r_1!}{(a_1-\beta_1)!}\cdots\frac{r_N!}{(a_N-\beta_N)!} \cdot \int_{\barM_{1,N}}\psi_1^{\beta_1}\cdots \psi_N^{\beta_N} \bigg)   \\
     &=& \frac{(-1)^N}{\alpha!}   \sum_{\substack{\beta_1+\cdots+\beta_N=N \\ r_\ell\geq \beta_\ell\geq 0 }}
     (-1)^{|\beta|-|r|}   \prod_{i=1}^N \frac{r_i!}{(r_i-\beta_i)!}  \sum_{\beta_i \leq a_i\leq r_i}   (-1)^{|a|-|\beta|}   \prod_{i=1}^N \binom{r_i-\beta_i}{a_i-\beta_i}  \int_{\barM_{1,N}}\psi_1^{\beta_1}\cdots \psi_N^{\beta_N} \\
        &=&\begin{cases}
        0  &  \text{if }    r_1+r_2+\cdots+r_N \neq N, \\
          \displaystyle\frac{(-1)^N}{\alpha!}  r_1!\cdots r_N! \black \int_{\barM_{1,N}}\psi_1^{r_1}\cdots \psi_N^{r_N}  &\text{if }    r_1+r_2+\cdots+r_N = N,
         \end{cases}
\end{eqnarray*}
 where we used  
 $    \sum_{\beta_i \leq a_i\leq r_i}   (-1)^{a_i-\beta_i}  \binom{r_i-\beta_i}{a_i-\beta_i}=0 $ because $\sum_{k=0}^n(-1)^n\displaystyle\binom{n}{k}=0$. The second case implies $r_1=r_2=\cdots=r_N=1$, $k=1$ and $N=\alpha_1\in \NN$ by   \eqref{fixed2}. Therefore      we have 
  \begin{eqnarray*} {\rm Coe}_{C^\alpha_m}{\bf X}=  \frac{(-1)^N}{N!}   \int_{\barM_{1,N}}\psi_1 \cdots \psi_N  =  \frac{(-1)^N}{24 N},  
  \end{eqnarray*}
   whenever $k=1$ ($N =\alpha_1$) and $|m|=0$,  and 
     $ {\rm Coe}_{C^\alpha_m}{\bf X}=0$ otherwise.  Taking   $Res_{\hbar=0}\ln$ to \eqref{reg} we have
    \beq\label{LemmX}{\bf X} =\sum_{N=1}^\infty\frac{(-1)^N}{24N}C_0^N=-\frac{1}{24}\ln\big(1+C_0\big)=-\frac{1}{24} \ln(1+\bar{\sZ}_6\sta|_{\hbar=0})
  = -\frac{1}{24}{\rm Res}_{\hbar=0} \bl \ln(1+\bar{\sZ}_6\sta)) /\hbar \br.
  \eeq
 

 \subsection{Some identities $II$}
 
  We now  determine  ${\rm Coe}_{(\ft')^i}  \ln (1+\sZ_6\sta)$ for $i=1,2$.
 By \eqref{Z1,5}, we have 
   $$
 \sZ\sta_6=\frac{\ft^4}{5}\ti\sZ_{1,5}(\hbar) = -1  + \frac{1}{  I_0 }  exp( -\frac{\ft}{\hbar}(t-T)-\frac{ \ft'g_1}{\hbar I_0} )\cdot    \sum_{d=0}^\infty q^d \frac{  \prod\limits_{m=1}^{5d}(-5\ft+m\hbar)  \prod\limits_{m=1}^d(m+\frac{\ft'}{\hbar})}{d! \prod\limits_{m=1}^d(-\ft+m\hbar)^5}.
$$

  As $\sZ_6\sta$ has no $q$-degree zero term, so is $  \ln (1+\sZ_6\sta)$. Together with the fact that $\sZ_6\sta$ is a formal power series in $\ft'$, we know $\ln (1+\sZ_6\sta)$ is also a formal power series in $\ft'$. 
 \begin{lemm}\label{log}Let $R$ be a ring and $A_j\in R[[q]]q^j$, $j=0,1,2,3,\cdots$. Then
\begin{eqnarray*}
&&\ln (1+ A_0+A_1\ft'+A_2(\ft')^2+\cdots)\\
&=&\ln(1+A_0)+   \frac{A_1}{1+A_0} \ft' +   \big(  \frac{A_2}{1+A_0}-\frac{A_1^2}{2(1+A_0)^2}\big) (\ft')^2 +  \text{higher powers of}\ \ft'
\end{eqnarray*}
 \end{lemm}
Introduce $\sR(\hbar,\ft):=\sum\limits_{d=0}^\infty q^d  \displaystyle\frac{\prod\limits_{m=1}^{5d}(-5\ft+m\hbar)}{\prod\limits_{m=1}^d(-\ft+m\hbar)^5}$  and observe first 
 \begin{eqnarray*}&&\sum_{d=1}^\infty q^d \frac{  \prod\limits_{m=1}^{5d}(-5\ft+m\hbar)  \prod_{m=1}^d(m+\frac{\ft'}{\hbar})}{d! \prod\limits_{m=1}^d(-\ft+m\hbar)^5} = (\sR-1)  +C_1\frac{ \ft' }{\hbar}  + C_2 \frac{ (\ft')^2}{\hbar^2}   
 +  \text{higher powers of}\ \ft' , 
 \end{eqnarray*}
where $$C_1=  \sum_{d=1}^\infty q^d \frac{  \prod\limits_{m=1}^{5d}(-5\ft+m\hbar)  \sum_{m=1}^d \frac{1}{m}}{\prod\limits_{m=1}^d(-\ft+m\hbar)^5} \and C_2=  \sum_{d=2}^\infty q^d \frac{  \prod\limits_{m=1}^{5d}(-5\ft+m\hbar)  \sum_{1\leq a<b\leq d} \frac{1}{ab}}{\prod\limits_{m=1}^d(-\ft+m\hbar)^5} . $$
  Applying   Lemma \ref{log} and above observations to the last term of
  $$\ln (1+\sZ_6\sta)=-\ln I_0  -\frac{\ft}{\hbar}(t-T)-\frac{ \ft'g_1}{\hbar I_0} + \ln \Big( 1+ \sum_{d=1}^\infty q^d \frac{  \prod\limits_{m=1}^{5d}(-5\ft+m\hbar)  \prod\limits_{m=1}^d(m+\frac{\ft'}{\hbar})}{d! \prod\limits_{m=1}^d(-\ft+m\hbar)^5}  \Big),$$
  we obtain
 $$ {\rm Coe}_{(\ft')^1}   \ln (1+\sZ_6\sta) =-\frac{  g_1}{\hbar I_0} + \frac{C_1}{\hbar\sR}, \ \  \ {\rm Coe}_{(\ft')^2}   \ln (1+\sZ_6\sta) = \frac{1}{\hbar^2} \bl  \frac{C_2}{ \sR}-\frac{C_1^2}{2\sR^2} \br .$$

 For simplicity, denote  and find $$A:=\sum_{d=1}^\infty (5^5q)^d   \sum_{m=1}^d \frac{1}{m}, \and B:=\sum_{d=1}^\infty (5^5q)^d    \sum_{1\leq a<b\leq d} \frac{1}{ab}. 
 $$
   $$A(1-5^5q) = -\ln(1-5^5q) \and B(1-5^5q)= \sum_{d=2}^\infty (5^5q)^{d}  \frac{1}{d}\sum_{a=1}^{d-1}\frac{1}{a}.  $$
Then 
$$C_1
=A+ \frac{\hbar}{\ft}\cdot 2 \frac{d}{dt}A+  \text{higher powers of}\ \hbar,   \and
C_2
=B+ \frac{\hbar}{\ft}\cdot 2 \frac{d}{dt}B+  \text{higher powers of}\ \hbar,  $$


\begin{eqnarray*}
\sR(\hbar,\ft)=\sum_{d=0}^\infty (5^5q)^d +  \frac{\hbar}{\ft} \sum_{d=0}^\infty  2d  (5^5q)^d      +  \text{higher powers of}\ \hbar
=\frac{1}{1-5^5q}+  \frac{\hbar}{\ft}  \frac{d}{dt}\frac{2}{1-5^5q}  (\, mod  \, \hbar^2), 
\end{eqnarray*}
where the expansion is near $\hbar=0$.  We can now calculate the followings:
 \begin{eqnarray}\label{4coe}
  &&Coe_{(\ft')^{0}}  Res_{\hbar=0}  \ln(1+ {\sZ}_6\sta) =   Res_{\hbar=0} \bl Coe_{(\ft')^{0}} \ln(1+ {\sZ}_6\sta))  \br  =-\ft(t-T),\\
  &&  {\rm Coe}_{(\ft')^{1}}  {\rm Res}_{\hbar=0} \frac{ \ln(1+ {\sZ}_6\sta) }{\hbar}  = {\rm Res}_{\hbar=0} \frac{C_1}{\hbar^2\sR} 
=\frac{2}{\ft}\frac{d}{dt} (A (1-5^5q))= \frac{2}{\ft} \cdot \frac{5^5q}{1-5^5q},  \nonumber \\
&& {\rm Coe}_{(\ft')^{1}}  \eta ={\rm Res}_{\hbar=0} {\rm Coe}_{(\ft')^{1}}\ln(1+\sZ_6\sta )= {\rm Res}_{\hbar=0}(-\frac{  g_1}{\hbar I_0} + \frac{C_1}{\hbar\sR}) \nonumber\\
 &=&-\frac{  g_1}{  I_0} + A (1-5^5q)= -\frac{  g_1}{  I_0} -\ln(1-5^5q),   \nonumber\\
 && {\rm Coe}_{(\ft')^{2}}  \eta  ={\rm Res}_{\hbar=0} {\rm Coe}_{(\ft')^{2}}\ln(1+\sZ_6\sta )= {\rm Res}_{\hbar=0}( \frac{1}{\hbar^2} \bl  \frac{C_2}{ \sR}-\frac{C_1^2}{2\sR^2} \br)   \nonumber\\
 &=& \frac{2}{\ft} \frac{dB (1-5^5q)}{dt}  +  \frac{1}{2}\frac{4}{\ft}  A (1-5^5q) \frac{d A(1-5^5q) }{dt}   \nonumber\\& = &\frac{2}{\ft}\sum_{d=2}^\infty (5^5q)^{d}  \sum_{a=1}^{d-1}\frac{1}{a}   +\frac{2}{\ft}  \frac{5^5q}{1-5^5q} \ln(1-5^5q)
 =0 \nonumber .
 \end{eqnarray}
 
\black

 \section{Appendix: Differential relations arising from geometry}
\label{secdifferentialrelations}

We shall derive some differential relations among periods from geometry
in such a way that we keep track of the arithmetic origins of the coefficients in the relations. 
These relations are needed to prove Proposition \ref{lemmadesiredidentity}  for type B contribution studied in $g=1$ MSP localization.
We separate the discussions into this appendix since the relations obtained seem to be of independent interest.

The discussion applies to any one-parameter family of Calabi-Yau threefolds, but the focus of this work is on the quintic/mirror quintic family.
The corresponding 
differential operator is
\begin{equation*}
\mathcal{L}=\theta^{4}-\alpha \prod_{k=1}^{4} (\theta+{k\over 5})\,,
\end{equation*}
where 
$\alpha =5^5 q=5^5 e^{t}$ and $\theta=\alpha{\partial\over \partial \alpha}$.
This can be regarded as the Picard-Fuchs operator for the mirror quintic family,
or the quantum differential equation for the quintic family itself.
The properties 
are about the differential operator itself, regardless of whether we are looking at 
the A-model or B-model.
However, we shall use the language for the B-model, which is more convenient.\\

We first expand the above operator as
\begin{equation*}
\mathcal{L}_{\mathrm{hyp}}=\theta^{4}-\alpha \sum_{k=0}^{3} \sigma_{4-k}\theta^{k},
\end{equation*}
where $\sigma_{k}$ means the $k$th fundamental symmetric polynomials of the numbers
$1/5,2/5,3/5,4/5$.
For the quintic/mirror quintic case in consideration, one has
\begin{equation}\label{eqnsymmetricpolynomialsofroots}
\sigma_{1}=2\,,
\quad
\sigma_{2}={7\over 5}\,,
\quad
\sigma_{3}={2\over 5}\,,
\quad
\sigma_{4}={24\over 5^4}\,.
\end{equation}
Introducing the coordinate
$
t=\ln q$,
then the above differential operator becomes
\begin{equation}\label{eqnPFoperatorint}
\mathcal{L}_{\mathrm{hyp}}=(1-\alpha)\partial_{t}^{4}-\alpha \sum_{k=0}^{3} \sigma_{4-k}\partial_{t}^{k} \,,
\end{equation}
Introducing the notation (recall $\alpha=5^{5}q=5^{5}e^{t}$)
\begin{equation}\label{eqnCttt}
C_{ttt}={1\over 1-\alpha}={1\over 1-5^{5}q}\,,
\quad
C=\partial_{t}\ln C_{ttt}={\alpha\over 1-\alpha}={5^{5}q\over 1-5^{5}q}\,,
\end{equation}
then we can normalize the leading term of the differential operator to $1$ and get 
\begin{equation*}
\mathcal{L}:=
{1\over 1-\alpha}\mathcal{L}_{\mathrm{hyp}}=\partial_{t}^{4}
+\sum_{k=0}^{3}a_{k}\partial_{t}^{k}\,,\quad a_{k}=
- \sigma_{4-k}C\,.
\end{equation*}
The $4$ solutions of the corresponding differential equation obtained by the Frobenius method near the singular point $\alpha=0$ are denoted by
$
(I_{0},I_{1},I_{2},I_{3})$. In particular, one has
\begin{equation}\label{eqnboundaryconditionsforI0I1}
I_{0}=~_{4}F_{3}({1\over 5},{2\over 5},{3\over 5},{4\over 5};1,1,1;\alpha)\,,
\quad
I_{1}=\log ({\alpha\over 5^5})\cdot I_{0}+\cdots.
\end{equation}
Then one has
\begin{equation*}
T:={I_{1}\over I_{0}}=t+\cdots\,.
\end{equation*}

In the B-model, the operator $\mathcal{L}_{\mathrm{hyp}}$
is the Picard-Fuchs operator, and the four solutions above are the periods of the mirror quintic.
In the A-model, they are the first four coefficients in the $H$-expansion of the $I$-function $e^{H q}J$,
where $H$ is the hyperplane class of $\mathbb{P}^{4}$
and
$J$ is the $J$-function.

\subsection{Special geometry relation}

The set of periods has the special geometry structure \cite{Strominger:1990pd, Freed:1999sm}
which implies that there exists a holomorphic function $F$ (called the prepotential) of $T=I_{1}/I_{0}$,
such that
the set of periods above has the following structure
\begin{equation}\label{eqnspecialgeometry}
(I_{0},I_{1},I_{2},I_{3})=I_{0}(1,T,F_{T}, TF_{T}-2F)\,.
\end{equation}
This fact can be easily proved by using Griffiths transversality.
This structure in fact holds everywhere, but we shall only focus on a neighborhood near the point $\alpha=0$ near which we have
\begin{equation}\label{eqnprepotential}
F(T)={1\over 3!}T^{3}+\mathrm{quadratic~polynomial~in~}T+\mathcal{O}(e^{ T})\,.
\end{equation}
It follows from this structure (called the special geometry structure)
that the $4$ solutions $I_{k},k=0,1,2,3$ above, after normalization by
the fundamental period $I_{0}$, are solutions to the following equation (see \cite{Ceresole:1992su, Ceresole:1993qq})
\begin{equation*}
 \partial_{T}^2 F_{TTT}^{-1}\partial_{T}^2 ({I_{k}\over I_{0}})=0\,,
\quad k=0,1,2,3\,.
\end{equation*}

By comparing the leading terms, we can see that 
the Picard-Fuchs operator \eqref{eqnPFoperatorint} must satisfy the relation
\begin{equation}\label{eqnmatchingtwooperators}
\mathcal{L} =\mathcal{D}\,,
\end{equation}
where
\begin{equation*}
\mathcal{D}:=\left({1\over I_{0}}F_{TTT}^{-1}({\partial t\over \partial T})^4\right)^{-1} \circ \partial_{T}^2 F_{TTT}^{-1}\partial_{T}^2 \circ {1\over I_{0}}\,.
\end{equation*}\\

We now study the consequences of this identity by using the ``computing twice" trick.
We shall see that this identity 
\eqref{eqnmatchingtwooperators}
gives rise to some differential relations.\\

First,  it is easy to see that we can rewrite $\mathcal{D}$ as
\begin{eqnarray}\label{eqnoperatorD}
\mathcal{D}&=&
\left(\partial_{t}+\partial_{t}\log F_{TTT}^{-1}+\partial_{t}\log ({\partial t \over \partial T})^{3}+\partial_{t}\log {1\over I_{0}}\right)\circ
\left(\partial_{t}+\partial_{t}\log F_{TTT}^{-1}+\partial_{t}\log ({\partial t\over \partial T})^{2}+\partial_{t}\log {1\over I_{0}}\right) \nonumber\\
&&
\circ
\left(\partial_{t}+\partial_{t}\log ({\partial t \over \partial T}) +\partial_{t}\log {1\over I_{0}}\right)\circ
\left(\partial_{t}+\partial_{t}\log {1\over I_{0}}\right)\,.
\end{eqnarray}
For convenience, we introduce the following notations
\begin{eqnarray}\label{eqnck}
&&c_{1}=\partial_{t}\log {1\over I_{0}}\,,\quad 
c_{2}=\partial_{t}\log ({\partial t \over \partial T}) +\partial_{t}\log {1\over I_{0}} \,,\nonumber\\
&&c_{3}=\partial_{t}\log F_{TTT}^{-1}+\partial_{t}\log ({\partial t\over \partial T})^{2}+\partial_{t}\log {1\over I_{0}} \,,\quad 
c_{4}=\partial_{t}\log F_{TTT}^{-1}+\partial_{t}\log ({\partial t \over \partial T})^{3}+\partial_{t}\log {1\over I_{0}} \,.
\end{eqnarray} 
They satisfy the obvious relation
\begin{equation}\label{eqndifferencesofckrelation}
c_{2}-c_{1}=c_{4}-c_{3}\,.
\end{equation}

Straightforward computations show that\footnote{Hereafter $'$ means $\partial_{t}$ and similar convention is used for $''$ etc.}
\begin{equation*}
\mathcal{D}=\partial_{t}^{4}
+b_{3}\partial_{t}^{3}+b_{2}\partial_{t}^2+b_{1}\partial_{t}+b_{0}\,,
\end{equation*}
where
\begin{eqnarray}\label{eqnbkintermsofck}
 b_{3}&=&\sum_{k}c_{k}\,,\quad 
 b_{2}=\sum_{k\neq l}c_{k}c_{l}+3c_{1}'+2c_{2}'+c_{3}'\,,\nonumber\\
 b_{1}&=&\sum_{j\neq k\neq l}c_{j}c_{k}c_{l}
+(2c_{2}+2c_{3}+2c_{4})c_{1}'
+(2c_{1}+c_{3}+c_{4})c_{2}'
+
(c_{1}+c_{2})c_{3}'
+3c_{1}''+c_{2}''\,,\nonumber\\
b_{0}&=&c_{1}c_{2}c_{3}c_{4}
+(c_{2}c_{3}+c_{2}c_{4}+c_{3}c_{4})c_{1}'+
(c_{1}c_{3}+c_{1}c_{4})c_{2}'
+2c_{1}'c_{2}'+
 c_{1}c_{2}c_{3}'\nonumber\\
 &&+c_{1}'c_{3}'+c_{2}c_{1}''+c_{3}c_{1}''+c_{4}c_{1}''
+c_{1}c_{2}''+c_{1}'''\,.
\end{eqnarray}

\subsubsection{Matching degree three terms: Yukawa coupling $F_{TTT}$ in terms of period integrals}

The matching of the coefficients of $\partial_{t}^{3}$ yields $b_{3}=a_{3}$ which according to
\eqref{eqnbkintermsofck} is
\begin{equation*}
c_{1}+c_{2}+c_{3}+c_{4}=a_{3}\,.
\end{equation*}
Plugging in the relation $a_{k}=
- \sigma_{4-k}C$ and \eqref{eqnck}, we get the relation

\begin{equation}\label{eqnb3relation}
\partial_{t}\log F_{TTT}+2\partial_{t}\log I_{0}+3\partial_{t}\log {\partial T\over \partial t}=-{1\over 2}a_{3}= \sigma_{1}C\,.
\end{equation}
This gives a relation between the transcendental series $I_{0},  {\partial T\over \partial t} $ etc., and 
the rational function $C$.\\

For later use, we 
introduce furthermore
\begin{equation}\label{eqnABCY}
A=\partial_{t}\log {\partial T\over \partial t}\,,
\quad
B=\partial_{t}\log I_{0}\,,
\quad
Y=\partial_{t}\log F_{TTT}\,.
\end{equation}
The above relation is then (recall \eqref{eqnCttt})
\begin{equation*}
Y+2B+3A={1\over 2}\sigma_{1}C\,.
\end{equation*}
That is
\begin{equation*}
F_{TTT}={1\over (1-\alpha)^{\sigma_{1}\over 2}} {1\over I_{0}^2 } {1\over T'^3}\cdot const
\end{equation*}
for some constant $const$
which can be fixed to be $1$ by using the boundary conditions in \eqref{eqnboundaryconditionsforI0I1}
and
the known asymptotic behaviour of the the Yukawa coupling $F_{TTT}$
following from \eqref{eqnprepotential}. 

\subsubsection{Matching degree two terms: flatness of Gauss-Manin}

For the next coefficients, we have $b_{2}=a_{2}$. By \eqref{eqnbkintermsofck} this is
\begin{equation*}
\sum_{k\neq l}c_{k}c_{l}+3c_{1}'+2c_{2}'+c_{3}'=-\sigma_{2}C\,.
\end{equation*}
Using the $b_{3}$-relation to eliminate $c_{3}$, we can see that
the above relation becomes
\begin{equation}\label{eqnflatnessb2relation}
-A'-4B'-C'-A^2-2B^2+2BC+C^2-2AB+AC=-\sigma_{2}C\,.
\end{equation}

This relation is discussed in \cite{Lian:1994zv}
by using the Wronskian method.
It is further studied in \cite{Yamaguchi:2004bt, Hosono:2008ve, Zhou:2013hpa}.
In particular, as pointed out explicitly in \cite{Yamaguchi:2004bt, Zhou:2013hpa},
this relation is equivalent to the flatness of the Gauss-Manin connection 
on the Hodge bundle.
Note that the above differential relation about $A,B,C$ is in fact an identity without referring to the prepotential, although
in eliminating $c_{3}$ above we have used the $b_{3}$-relation whose definition does require
the existence of prepotential.
In the next section, we shall see that this relation can be naturally rephrased in terms of Wronskians, again without using the prepotential,
which offers a more intrinsic way of explaining what this relation is.

\subsubsection{Matching degree one terms: symplectic structure}

We now match the degree one terms in \eqref{eqnmatchingtwooperators}.
By straightforward calculation from \eqref{eqnbkintermsofck}, this is equivalent to
\begin{eqnarray}\label{eqnsymplecticstructureb1relation}
b_{1}=b_{2}'+ {1\over 2}b_{2}b_{3}-{1\over 8}b_{3}^3-{1\over 2}b_{3}''-{3\over 4}b_{3}b_{3}'\,.
\end{eqnarray}
This relation is studied in details in \cite{Ceresole:1992su, Ceresole:1993qq, Lian:1994zv, Almkvist:2004differential}.
In particular, as pointed out in \cite{Ceresole:1992su, Ceresole:1993qq}, this relation means that the Gauss-Manin connection is symplectic.
Note that, in this relation, the $c_{k}$'s are combined in such a way that
they do not appear individually.
Hence this relation does not give new differential relations among them.

\subsubsection{Matching degree zero terms: Picard-Fuchs for fundamental period}

Consider the degree zero terms in \eqref{eqnmatchingtwooperators}.
By using the relation $b_{3}=a_{3}$ and \eqref{eqnbkintermsofck}, we get an ODE for $c_{1}$:
\begin{eqnarray}\label{eqnb0relation}
 b_{0}
=c_{1}''' -c_{1}^4+6c_{1}'c_{1}^2-4c_{1}c_{1}''-3c_{1}'c_{1}'
+b_{3} (c_{1}''-3c_{1}c_{1}'+c_{1}^3)
+b_{2} (c_{1}'-c_{1}^2)
+b_{1}c_{1}.
\end{eqnarray}
On the other hand, the Picard-Fuchs equation
for the fundamental period $I_{0}$ reads (recall $B=\partial_{t}\log I_{0}$)
\begin{eqnarray*}
0&=&{ I_{0}''''\over I_{0}}+a_{3 } { I_{0}'''\over I_{0}}+a_{2}{ I_{0}''\over I_{0}}+a_{1}{ I_{0}'\over I_{0}}+a_{0} \,.
\end{eqnarray*}
Recall that $c_{1}=-B$, we can see that the Picard-Fuchs equation for $I_{0}$ above is exactly the relation
$
b_{0}=a_{0}$.

\subsubsection{A summary}

We remark that the $b_{0}$-relation \eqref{eqnb0relation} is always true, independent of the geometry structure.
The $b_{1}$-structure \eqref{eqnsymplecticstructureb1relation} is the symplectic structure, which can be further translated into relations on the periods,
see the discussion in the next section below. It does not refer to the prepotential, as we have pointed out earlier that it is not a relation
on the individual $c_{k}$'s. 
The $b_{2}$-relation \eqref{eqnflatnessb2relation}, the flatness condition, can also be phrased without referring to
the prepotential, which can also be seen from the explicit expression in terms of Wronskians discussed in the next section.
Finally, the $b_{3}$-relation \eqref{eqnb3relation} gives an identity that the prepotential satisfies, in case
the special geometry structure is present.

\subsection{Wronskian method}

The Wronskian method and the exterior square structure are useful in simplifications and in thinking about
what the quadratic differential polynomials in MSP calculation are.
A computational evidence is that 
the derivative of a normalized period is naturally expressed in terms of Wronskians.

It is shown in \cite{Almkvist:2004differential}
that 
for a generic order 4 differential equation $\mathcal{L}$,
the Wronskians
\begin{equation*}
M_{ab}=\det
\begin{pmatrix}
I_{a} & I_{b}\\
I_{a}' & I_{b}'
\end{pmatrix}
\end{equation*}
satisfies an order 6 differential equation.
This is the exterior square of $\mathcal{L}$.
We now review the results in \cite{Almkvist:2004differential}, which are important for the proof of the 
main result in this work.
We denote
\begin{eqnarray*}
&&u_{1}=M_{ab}\,,
\quad
u_{2}=
M_{ab}'=\det
\begin{pmatrix}
I_{a} & I_{b}\\
I_{a}'' & I_{b}''
\end{pmatrix}\,,\nonumber\\
&& u_{3}=\det
\begin{pmatrix}
I_{a} & I_{b}\\
I_{a}''' & I_{b}'''
\end{pmatrix}\,,\quad
u_{4}=\det
\begin{pmatrix}
I_{a}' & I_{b}'\\
I_{a}'' & I_{b}''
\end{pmatrix}\,, \nonumber\\
&&
u_{5}=\det
\begin{pmatrix}
I_{a}' & I_{b}'\\
I_{a}''' & I_{b}'''
\end{pmatrix}\,,\quad
u_{6}=\det
\begin{pmatrix}
I_{a}'' & I_{b}''\\
I_{a}''' & I_{b}'''
\end{pmatrix}\,.
\end{eqnarray*}

Recall that the $I_{a}$'s are annihilated by the order 4 differential operator
\begin{equation*}
\mathcal{L}=\partial^4+a_{3}\partial^3+a_{2}\partial^2+a_{1}\partial+a_{0}\,.
\end{equation*}
Then the set $\{u_{k}, k=1,2,\cdots 6\}$ satisfy the following relations
\begin{eqnarray}\label{eqnrecursiveWronskian}
u_{1}'&=&u_{2}\,,\quad 
u_{2}'=u_{3}+u_{4}\,,\quad 
u_{3}'=u_{5}-a_{1}u_{1}-a_{2}u_{2}-a_{3}u_{3}\,,\nonumber\\
u_{4}'&=&u_{5}\,,\quad 
u_{5}'=u_{6}+a_{0}u_{1}-a_{2}u_{4}-a_{3}u_{5}\,,\quad 
u_{6}'=a_{0}u_{2}+a_{1}u_{4}-a_{3}u_{6}\,.
\end{eqnarray}

\subsubsection{Symplectic structure}\label{secsymplecticstructureviaWronskian}

It is shown in \cite{Almkvist:2004differential}
that the above system gives rise to an order 6 differential equation for $u_{1}$.
The solutions are $M_{ab}\,, a,b=1,2,\cdots 4$.
When the symplectic structure \eqref{eqnsymplecticstructureb1relation}
\begin{equation}\label{eqnsymplecticstructure}
a_{1}={1\over 2}a_{2}a_{3}-{1\over 8}a_{3}^3+a_{2}'-{3\over 4}a_{3}a_{3}'-{1\over 2}a_{3}''
\end{equation}
is present, the order 6 differential equation becomes an order 5 differential equation.
We now recall the computations therein.
Define
\begin{equation*}
U=u_{1}'''+a_{3}u_{1}''+a_{2}u_{1}' +a_{1}u_{1}\,,
\end{equation*}
then
\begin{equation}\label{eqnUrelation}
U-a_{3}u_{4}=2u_{5}\,.
\end{equation}
Taking the derivative, one gets
\begin{equation*}
U'-a_{3}' u_{4}-a_{3} u_{5}=2 u_{6}+2a_{0}u_{1}-2a_{2}u_{4}-2a_{3}u_{5}\,.
\end{equation*}
Eliminating $u_{5}$ using \eqref{eqnUrelation},  we obtain 
\begin{equation}\label{eqnU'}
U'+{1\over 2}a_{3}U=u_{4}(a_{3}'+{1\over 2}a_{3}^2-2a_{2})+2a_{0}u_{1}+2u_{6}
:=C_{1}u_{4}+2a_{0}u_{1}+2u_{6}\,.
\end{equation}
Taking one more derivative and eliminating $u_{5}$ using \eqref{eqnUrelation}, we obtain 
\begin{equation*}
U''+{1\over 2}a_{3}'U+{1\over 2} a_{3}U'
=C_{1}'u_{4}+{1\over 2}C_{1} (U-a_{3}u_{4}) +2a_{0}'u_{1}+2a_{0}u_{1}'+2 (a_{0}u_{1}'+a_{1}u_{4}-a_{3}u_{6})\,.
\end{equation*}
Using 
\eqref{eqnU'}
to eliminate $u_{6}$, we obtain
\begin{eqnarray*}
&&U''+{1\over 2}a_{3}'U+{1\over 2} a_{3}U'\\
&=&C_{1}'u_{4}+{1\over 2}C_{1} (U-a_{3}u_{4})
+2a_{0}'u_{1}+4a_{0}u_{1}'+2a_{1}u_{4}
-a_{3} (U'+{1\over 2}a_{3}U-C_{1}u_{4}-2a_{0}u_{1}).
\end{eqnarray*}
Simplifying, one has
\begin{eqnarray*}
W&:=&U''+{1\over 2}a_{3}'U+{3\over 2} a_{3}U'-{1\over 2}C_{1}U+{1\over 2}a_{3}U -
2a_{0}'u_{1}-4a_{0}u_{1}'-2 a_{0}a_{3}u_{1}\nonumber \\
&=&C_{1}'u_{4}+{3\over 2}C_{1}a_{3}u_{4}
+2a_{1}u_{4} :=C_{2}u_{4}
\end{eqnarray*}
If the condition $C_{2}=0$ is satisfied, then we are done.
It turns out that this condition is exactly the symplectic structure relation
\eqref{eqnsymplecticstructure}.
Otherwise, we get 
\begin{equation*}
W'=C_{2}u_{4}'+C_{2}'u_{4}=C_{2}u_{5}+C_{2}'u_{4}
={1\over 2}C_{2} (U-a_{3}u_{4})+C_{2}'u_{4}
={1\over 2}C_{2}U+(C_{2}'-{1\over 2}C_{2}a_{3})u_{4}\,.
\end{equation*}
Hence
\begin{equation}\label{eqnorder6}
W'={1\over 2}C_{2}U+(C_{2}'-{1\over 2}C_{2}a_{3}){ W\over C_{2}}\,.
\end{equation}
This gives an order 6 ODE as desired.

\subsubsection{Simplication of Wronskians in the presence of symplectic structure}

As shown in  \cite{Almkvist:2004differential},
the symplectic structure condition \eqref{eqnsymplecticstructure} is equivalent to either of the following relations
\begin{eqnarray*}
{\partial^{2}\over \partial T^2}
{I_{3}\over I_{0}}
&=&T {\partial^{2}\over \partial T^2}
{I_{2}\over I_{0}}\,,\\
{\partial^{2}\over \partial T^2}  {I_{2}\over I_{0}}&=&
{1\over I_{0}^2( \partial_{t}T)^3} \exp (-{1\over 2}\int a_{3})\,.
\end{eqnarray*}
This condition is evidently satisfied  
if the stronger special geometry structure \eqref{eqnspecialgeometry} holds, as it should be.
In this case, the second condition above condition is nothing but the $b_{3}$-relation in \eqref{eqnb3relation}.
Note that this should be thought of as a relation on $b_{1}$ instead of one on $b_{3}$,
since in general there is no notion of prepotential but one can always talk about the condition above.

In terms of Wronskians, this is equivalent to 
\begin{eqnarray*}
W=U''+{1\over 2}a_{3}'U+{3\over 2} a_{3}U'-{1\over 2}C_{1}U+{1\over 2}a_{3}U -
2a_{0}'u_{1}-4a_{0}u_{1}'-2 a_{0}a_{3}u_{1}
=0\,,
\end{eqnarray*}
where
\begin{equation*}
C_{1}=
(a_{3}'+{1\over 2}a_{3}^2-2a_{2})\,,
\quad
U-a_{3}u_{4}=2u_{5}\,.
\end{equation*}
We now focus on the relation
\begin{equation*}
{\partial^{2}\over \partial T^2}
{I_{3}\over I_{0}}
=T {\partial^{2}\over \partial T^2}{I_{2}\over I_{0}}\,.
\end{equation*}
above.
Integrating, one gets
\begin{equation*}
{\partial\over \partial T}
{I_{3}\over I_{0}}
=T {\partial\over \partial T}{I_{2}\over I_{0}}-  {I_{2}\over I_{0}}+const
=T^{2} {\partial \over \partial T} ({1\over T} {I_{2}\over I_{0}})+const\,,
\end{equation*}
for some constant $const$.
Using the boundary conditions given in
\eqref{eqnboundaryconditionsforI0I1},\eqref{eqnspecialgeometry},\eqref{eqnprepotential}
 (in fact, only the $\log \alpha$ terms) for the periods $I_{2},I_{3}$
one can see that $const=0$.

Now since
\begin{equation*}
{\partial \over \partial T}={1\over T'}{\partial\over \partial t}\,,
\end{equation*}
the above is equivalent to
\begin{equation*}
{\partial\over \partial t} ({I_{3}\over I_{0}})=T^{2}{\partial\over \partial t}  ({1\over T}{I_{2}\over I_{0}})=
T^{2}{\partial\over \partial t}  ({I_{2}\over I_{1}})\,.
\end{equation*}
This gives 
\begin{equation}\label{eqnsymplecticstructureWronskian}
I_{0}I_{3}'-I_{0}'I_{3}=I_{1}I_{2}'-I_{1}'I_{2}\,.
\end{equation}
That is,
\begin{equation*}
\det
\begin{pmatrix}
I_{0} & I_{3}\\
I_{0}' &  I_{3}'
\end{pmatrix}
=
\det
\begin{pmatrix}
I_{1} & I_{2}\\
 I_{1}' &  I_{2}'
\end{pmatrix}\,.
\end{equation*}
This is why the order 6 differential equation \eqref{eqnorder6} for the Wronskians,
which has $6$ solutions $M_{ab}$, becomes
an order 5 differential equation.

\subsubsection{Flatness condition in terms of Wronskians in the present of symplectic structure}
\label{secflatnessWronskians}

The exterior square structure above, which is the underlying geometry of Wronskians, 
provides a nice way to express the identities we found above by matching
coefficients in \eqref{eqnmatchingtwooperators}.\\

We have done this for the symplectic structure relation in terms of
\begin{equation*}
{\partial^{2}\over \partial T^2}
{I_{3}\over I_{0}}
=T {\partial^{2}\over \partial T^2}{I_{2}\over I_{0}},
\end{equation*}
or equivalently in terms of the Wronskians
\begin{equation*}
I_{0}I_{3}'-I_{0}'I_{3}=I_{1}I_{2}'-I_{1}'I_{2}\,.
\end{equation*}

We now do this for the flatness condition without assuming the symplectic structure.
We specialize to $M_{ab}$ with $a=0,b=1$ and again
use the following convention $u_{k}, k=1,2,\cdots 6$
for the Wronskians.
One has the relations
\begin{equation*}
u_{2}'=u_{3}+u_{4}\,,
\end{equation*}
and also
\begin{equation*}
{u_{2}'\over u_{1}}=({u_{2}\over u_{1}})'+({u_{2}\over u_{1}})^2\,.
\end{equation*}
The latter is due to $u_{2}=u_{1}'$.
We write them in terms of the coefficients $c_{k}$ in \eqref{eqnck}
and the Wronskians.
From $c_{1}=-B, c_{2}=-A-B$, it is easy to see that
\begin{eqnarray*}
{u_{2}\over u_{1}}=-(c_{1}+c_{2})\,,
\quad 
({u_{2}\over u_{1}})'=-(c_{1}'+c_{2}')\,,
\quad 
{u_{3}\over u_{1}}=-2c_{1}'-c_{2}'+3c_{1}c_{2}+(c_{1}-c_{2})^2\,,
\quad 
{u_{4}\over u_{1}}=c_{1}'+c_{1}c_{2}\,,
\nonumber\\
{u_{4}+u_{3}\over u_{1}}={u_{2}'\over u_{1}}=({u_{2}\over u_{1}})'+({u_{2}\over u_{1}})^2=
-(c_{1}'+c_{2}')+(c_{1}+c_{2})^2\,,
\quad 
{u_{4}-u_{3}\over u_{1}}=3c_{1}'+c_{2}'-(c_{1}^2+c_{2}^2).
\end{eqnarray*}
Recall that the 
$b_{2}$, $b_{3}$-terms in \eqref{eqnbkintermsofck} are
\begin{equation*}
b_{2}={1\over 2}((\sum c_{k})^2-\sum c_{k}^2)+3c_{1}'+2c_{2}'+c_{3}'\,,
\quad
b_{3}=\sum c_{k}\,.
\end{equation*}
Using the relation \eqref{eqndifferencesofckrelation},
we can solve for $c_{3}, c_{4}$ as follows
\begin{equation*}
c_{3}={1\over 2}b_{3}-c_{2}\,,
\quad
c_{4}={1\over 2}b_{3}-c_{1}\,.
\end{equation*}
Plugging these into the $b_{2}$-coefficient, we have
\begin{eqnarray*}
b_{2}&=&{1\over 2}b_{3}^2-{1\over 2} (c_{1}^2+c_{2}^2+c_{3}^2+c_{4}^2)+3c_{1}'+2c_{2}'+c_{3}'\\
&=& {1\over 4}b_{3}^2-(c_{1}^2+c_{2}^2)+{1\over 2}b_{3}(c_{1}+c_{2})+{1\over 2}b_{3}'+3c_{1}'+c_{2}'\,.
\end{eqnarray*}
In terms of Wronskians, we have
\begin{equation*}
c_{1}'={u_{4}\over u_{1}} -c_{1}c_{2}\,,
\quad
2c_{1}'+c_{2}'=-{u_{3}\over u_{1}} +3c_{1}c_{2}+(c_{1}-c_{2})^2\,.
\end{equation*}
Eliminating the derivatives from the $b_{2}$-expression, we get
\begin{eqnarray*}
b_{2}&=&{1\over 4}b_{3}^2-(c_{1}^2+c_{2}^2)+{1\over 2}b_{3}(c_{1}+c_{2})+{1\over 2}b_{3}'+3c_{1}'+c_{2}'
\\
&=&
{1\over 4}b_{3}^2-(c_{1}^2+c_{2}^2)+{1\over 2}b_{3}(c_{1}+c_{2})+{1\over 2}b_{3}'
+{u_{4}\over u_{1}}-{u_{3}\over u_{1}}+2 c_{1}c_{2}+(c_{1}-c_{2})^2\\
&=&
{1\over 4}b_{3}^2+{1\over 2}b_{3}(c_{1}+c_{2})+{1\over 2}b_{3}'
+{u_{4}-u_{3}\over u_{1}}\,.
\end{eqnarray*}
Now using the relation $u_{2}/u_{1}=-(c_{1}+c_{2})$ above, this gives
\begin{equation*}
-{1\over 2}b_{3} {u_{2}\over u_{1}}+ {u_{4}-u_{3}\over u_{1}}
=b_{2}-{1\over 2}b_{3}'-{1\over 4}b_{3}^2\,.
\end{equation*}
We hence get the flatness condition in terms of Wronskians $u_{1},u_{2},u_{3},u_{4}$:
\begin{equation}\label{eqnflatnessWronskian}
( {u_{4}-u_{3}})
=(b_{2}-{1\over 2}b_{3}'-{1\over 4}b_{3}^2) u_{1}+{1\over 2}b_{3} u_{2}\,.
\end{equation}
Substituting the relations $b_{3}=a_{3}=-\sigma_{1}C=-2C$
and $b_{2}=a_{2}=-\sigma_{2}C$, recall that $C={\alpha\over  1-\alpha}$
satisfies
\begin{equation*}
C'=C^2+C\,,
\end{equation*}
we then get
\begin{equation*}
( {u_{4}-u_{3}})
= -C \left(   (\sigma_{2}-1) u_{1}+{1\over 2} \sigma_{1} u_{2}\right).
\end{equation*}
That is,
\begin{equation*}
5(1-\alpha) ( {u_{4}-u_{3}})
= -\alpha  \left( 2  u_{1}+5 u_{2}\right),
\end{equation*}
as
$
\sigma_{2}={7/ 5}
$ from \eqref{eqnsymmetricpolynomialsofroots}.
It is easy to check that this equation is the same as \eqref{eqnflatnessb2relation}, as it should be.
Note that here we did not use the $b_{1}$-relation \eqref{eqnsymplecticstructureb1relation}, hence no symplectic structure is used.

\begin{rema}
The quantity $C=\partial_{t}\ln C_{ttt}$ is independent of the $\sigma_{k}$'s.
This is in contrast with the value of the Yukawa coupling $C_{ttt}:=\int \Omega\wedge \partial_{t}^3\Omega$ which is computationally solved from
\begin{equation*}
\partial_{t}\log C_{ttt}=-{a_{3}\over 2}\,,
\end{equation*}
and is given by
\begin{equation*}
C_{ttt}=const\cdot \exp (\int -{a_{3}\over 2} )=const\cdot \exp (\int {\sigma_{1}\over 2} C )\,,
\end{equation*}
for some constant $ const $.
Hence it depends on the specific value of $a_{3}$ or equivalently $\sigma_{1}$
for the present hypergeometric case.
\end{rema}

\section{ Appendix: Antiderivatives of differential polynomials, and its application\black}
\label{sectypeBseriessimplification}

\subsection{Preparation: anti-derivatives}

The structure in the relation \eqref{eqnmatchingtwooperators} has other interesting consequences besides the differential relations we obtain in Appendix \ref{secdifferentialrelations}.
This structure is written as
\begin{equation}\label{eqnL=D}
\mathcal{L}=\mathcal{D}\,,
\end{equation}
where as in \eqref{eqnoperatorD}
\begin{equation} \label{eqnPFoperatorL}
\mathcal{L}= \partial_{t}^4+\sum_{k=0}^3 (-\sigma_{4-k}C)\partial_{t}^{k}\,,
\end{equation}
and as in \eqref{eqnck}
\begin{equation} \label{eqnPFoperatorD}
\mathcal{D}=(\partial_{t}+c_{4})(\partial_{t}+c_{3})(\partial_{t}+c_{2})(\partial_{t}+c_{1})\,.
\end{equation}

For later use, we shall denote
\begin{equation*}
\beta=1-\alpha\,,
\quad
\mathcal{D}_{k}=(\partial_{t}+c_{k})\,,
\quad k=1,2,3,4\,.
\end{equation*}

\subsubsection{Commuting differential operator actions on periods}

Some first properties of the relation \eqref{eqnmatchingtwooperators} are as follows.

\begin{lemm}\label{lemcommutingdifferentialoperators}
For the four solutions $I_{k},k=0,1,2,3$ to the Picard-Fuchs equation, one has
\begin{equation*}
\mathcal{L}_{\mathrm{hyp}}\circ \partial_{t}I_{k} =\partial_{t}^{4}I_{k}\,.
\end{equation*}
\end{lemm}
\begin{proof}
By definition,
\begin{equation*}
\mathcal{L}_{\mathrm{hyp}}\circ \partial_{t}I_{k}=[\mathcal{L}_{\mathrm{hyp}}, \partial_{t}]I_{k}+ \partial_{t}\circ  \mathcal{L}_{\mathrm{hyp}}I_{k}
=-[ \partial_{t}, \mathcal{L}_{\mathrm{hyp}}]I_{k}\,.
\end{equation*}
Now using \eqref{eqnPFoperatorL} one has
\begin{equation*}
[ \partial_{t}, \mathcal{L}_{\mathrm{hyp}}]=(\partial_{t}\beta)\partial_{t}^4+
\sum_{k=0}^3 (-\sigma_{4-k}\partial_{t}\alpha)\partial_{t}^{k}
=(\beta-1)\partial_{t}^4+
\sum_{k=0}^3 (-\sigma_{4-k} \alpha)\partial_{t}^{k}
=-\partial_{t}^4+\mathcal{L}_{\mathrm{hyp}}\,.
\end{equation*}
Therefore, one obtains
\begin{equation*}
\mathcal{L}_{\mathrm{hyp}}\circ \partial_{t}I_{k}=\partial_{t}^4 I_{k}-\mathcal{L}_{\mathrm{hyp}} I_{k}
=\partial_{t}^4 I_{k}\,.
\end{equation*}
\end{proof}

\begin{lemm}\label{lemI0L}
The following relation for the differential operators holds
\begin{equation*}
I_{0}\mathcal{L}_{\mathrm{hyp}}=\partial_{t}\circ (I_{0}\beta)\circ \mathcal{D}_{3}\circ \mathcal{D}_{2}\circ \mathcal{D}_{1}\,.
\end{equation*}
\end{lemm}
\begin{proof}
The relation 
\eqref{eqnL=D}
reads that
\begin{equation*}
\mathcal{L}_{\mathrm{hyp}}=\beta\mathcal{D}_{4}\circ \mathcal{D}_{3}\circ \mathcal{D}_{2}\circ \mathcal{D}_{1}\,,
\quad
\mathcal{D}_{4}=\partial_{t}+c_{4}
=\partial_{t}-Y-B-3A\,.
\end{equation*}
where the quantities $C,A,B,Y$ are as defined in \eqref{eqnCttt} and \eqref{eqnABCY}.
According to the $b_{3}$-relation in \eqref{eqnb3relation} and \eqref{eqnCttt}, we have
\begin{equation*}
-Y-B-3A=B-C=\partial_{t}\ln I_{0}-\partial_{t}\ln C_{ttt}
=\partial_{t}\ln (I_{0}  \beta)\,.
\end{equation*}
Therefore, we obtain
\begin{equation*}
\mathcal{D}_{4}=\partial_{t}+\partial_{t}\log (I_{0}\beta)\,.
\end{equation*}
It follows that
\begin{equation*}
(I_{0}\beta) \circ \mathcal{D}_{4}=\partial_{t}\circ (I_{0}\beta)\,.
\end{equation*}
According to 
\eqref{eqnL=D}, 
 we obtain
\begin{equation*}
I_{0}\mathcal{L}_{\mathrm{hyp}}=\partial_{t}\circ (I_{0}\beta)\circ \mathcal{D}_{3}\circ \mathcal{D}_{2}\circ \mathcal{D}_{1}\,.
\end{equation*}
\end{proof}

\begin{prop}\label{propantiderivativeI0Ik''''}
For any solution $I_{k},k=0,1,2,3$ to the Picard-Fuchs equation, one has
\begin{equation*}
\int  I_{0}\partial_{t}^4 I_{k}= (I_{0}\beta)\circ \mathcal{D}_{3}\circ \mathcal{D}_{2}\circ \mathcal{D}_{1}\circ \partial_{t}I_{k}+C_{k}\,,
\end{equation*}
for some constant $C_{k}$.
\end{prop}
\begin{proof}
Combining Lemma \ref{lemcommutingdifferentialoperators}
and Lemma \ref{lemI0L}, we obtain
\begin{equation*}
I_{0} \partial_{t}^4 I_{k}=I_{0}\mathcal{L}\circ \partial_{t}I_{k}=
\partial_{t}\circ (I_{0}\beta)\circ \mathcal{D}_{3}\circ \mathcal{D}_{2}\circ \mathcal{D}_{1}\circ \partial_{t}I_{k}\,.
\end{equation*}
The claim then follows.
\end{proof}

\begin{prop}\label{propantiderivativeT'I0Ik''''}
For any solution $I_{k},k=0,1,2,3$ to the Picard-Fuchs equation, one has
\begin{equation*}
\int T' \int  I_{0}\partial_{t}^4 I_{k}= (I_{0}\beta T')\circ  \mathcal{D}_{2}\circ \mathcal{D}_{1}\circ \partial_{t}I_{k}+C_{k}T+D_{k}\,,
\end{equation*}
for some constants $C_{k}, D_{k}$.
\end{prop}
\begin{proof}
Since 
\begin{equation*}
I_{0}\beta T' \mathcal{D}_{3}=I_{0}\beta T'  \circ (\partial_{t}+\partial_{t}\log (\beta I_{0} T' ))
=\partial_{t}\circ (I_{0}\beta T')\,,
\end{equation*}
the claim then follows  from Proposition \ref{propantiderivativeI0Ik''''}.
\end{proof}

For simplicity in what follows we shall omit the integration constants, equipped with the known boundary conditions provided in \eqref{eqnboundaryconditionsforI0I1},\eqref{eqnspecialgeometry},\eqref{eqnprepotential}.\\

By integration by parts, one has
\begin{prop}\label{propantiderivativeIkI0''''}
One has
\begin{equation*}
\int I_{k}I_{0}'''' =\int J_{k} I_{0}I_{0}''''=J_{k} \int I_{0}I_{0}'''' 
-\int J_{k}' \int I_{0}I_{0}'''' \,,
\quad
J_{k}={I_{k}\over I_{0}}\,.
\end{equation*}
In particular, when $k=1$, this is Proposition \ref{propantiderivativeT'I0Ik''''}.
When $k=2$, one has
\begin{eqnarray*}
\int J_{k}' \int I_{0}I_{0}''''&=&\int F_{TT}\cdot T' \int I_{0}I_{0}'''' 
= F_{TT}\cdot  \int T' \int I_{0}I_{0}'''' \nonumber-\int F_{TT}' \int T' \int I_{0}I_{0}''''\\
&=&  
F_{TT}\cdot 
(T' I_{0}\beta )  \mathcal{D}_{2}\circ \mathcal{D}_{1}
\partial_{t}I_{0}
-
(F_{TTT}T'^2 I_{0}\beta) \mathcal{D}_{1}\partial_{t}I_{0}\,. \black
\end{eqnarray*}
\end{prop}

\subsubsection{Anti-derivatives of degree four differential polynomials of periods}\label{diffrelation}

One can now obtain the anti-derivatives of degree four differential polynomials of periods.
We now compute some which shall be used later.
We have
\begin{equation}\label{eqnantiderivativeI0''I0''}
\int I_{0}''I_{0}''=I_{0}''I_{0}'-I_{0}'''I_{0}+\int I_{0}''''I_{0}
=I_{0}''I_{0}'-I_{0}'''I_{0}+( I_{0}\beta)\mathcal{D}_{3}\circ \mathcal{D}_{2}\circ \mathcal{D}_{1}\circ \partial_{t}I_{0}\,.
\end{equation}
\begin{equation} \label{eqnantiderivativeI1''I0''}
\int I_{1}''I_{0}''=I_{0}'I_{1}''-I_{0}I_{1}'''+\int I_{0}I_{1}''''
=I_{0}'I_{1}''-I_{0}I_{1}'''+( I_{0}\beta)\mathcal{D}_{3}\circ \mathcal{D}_{2}\circ \mathcal{D}_{1}\circ \partial_{t}I_{1}\,.
\end{equation}
For the anti-derivative of $I_{1}''I_{1}''$, one has
\begin{eqnarray}\label{eqnantiderivativeI1''I1''}
\int I_{1}''I_{1}''
&=&
I_{1}'' I_{1}'-I_{1}'''I_{1}
+\int 
T I_{1}''''I_{0}
=
I_{1}'' I_{1}'-I_{1}'''I_{1}+T \int I_{0}I_{1}''''  -  \int T'\int I_{0}I_{1}'''' \nonumber\\
&=&
I_{1}'' I_{1}'-I_{1}'''I_{1}+T ( I_{0}\beta)\mathcal{D}_{3}\circ \mathcal{D}_{2}\circ \mathcal{D}_{1}\circ \partial_{t}I_{1}- 
  (T' I_{0}\beta) \mathcal{D}_{2}\circ \mathcal{D}_{1}\circ \partial_{t}I_{1}\,.
\end{eqnarray}

\subsubsection{Relations arising from integration by parts}

There are sometimes multiple ways to compute the anti-derivatives. One can simpliy apply Proposition \ref{propantiderivativeI0Ik''''},
or alternatively combine Proposition \ref{propantiderivativeI0Ik''''} and Proposition \ref{propantiderivativeIkI0''''}.
For example, 
we have besides \eqref{eqnantiderivativeI1''I0''} the following
\begin{eqnarray*}
\int I_{1}''I_{0}''&=&I_{0}''I_{1}'-I_{0}'''I_{1}+\int I_{1}I_{0}'''' 
=I_{0}''I_{1}'-I_{0}'''I_{1}+\int T I_{0}I_{0}''''   \nonumber\\ 
&=&I_{0}''I_{1}'-I_{0}'''I_{1}+T  \int  I_{0}I_{0}''''-\int T' \int  I_{0}I_{0}''' \nonumber\\
&=&I_{0}''I_{1}'-I_{0}'''I_{1}+T (  I_{0}\beta)\mathcal{D}_{3}\circ \mathcal{D}_{2}\circ \mathcal{D}_{1}\circ \partial_{t}I_{0}
-(T' I_{0}\beta) \mathcal{D}_{2}\circ \mathcal{D}_{1}\circ \partial_{t}I_{0}\,.
\end{eqnarray*}
Taking the difference between these two equations, we have
\begin{eqnarray*}
&&I_{0}'I_{1}''-I_{0}I_{1}'''-(I_{0}''I_{1}'-I_{0}'''I_{1}) 
=\int ( I_{1}I_{0}''''-I_{1}''''I_{0}) \nonumber\\
&=&T(  I_{0}\beta)\mathcal{D}_{3}\circ \mathcal{D}_{2}\circ \mathcal{D}_{1}\circ \partial_{t}I_{0} -(T' I_{0}\beta) \mathcal{D}_{2}\circ \mathcal{D}_{1}\circ \partial_{t}I_{0}-(  I_{0}\beta)\mathcal{D}_{3}\circ \mathcal{D}_{2}\circ \mathcal{D}_{1}\circ \partial_{t}I_{1}\,.
\end{eqnarray*}
Similarly, in computing the anti-derivative of $I_{1}'''I_{0}'$, one can get the same
identity.
While directly computing the right hand side in the above seems complicated, the left hand side is easy.
That is, by integration by parts and the integration formulas in Proposition \ref{propantiderivativeI0Ik''''} and Proposition \ref{propantiderivativeIkI0''''}, we can easily produce nontrivial relations.

\subsection{Proof of an identity}

 The antiderivatives of $I_j^{(a)}I_k^{(b)}$'s in previous section admits an application in packaging loop contribution in $g=1$ MSP.  We shall prove the following identity.

\begin{prop} \label{lemmadesiredidentity}   In the following,  $I_j=I_j(y)$ in integrand and $'$ means derivative with respect to $y$. 
 \begin{align}   &   - \frac{1}{2} \   \partial_T J_2 \  \int_{-\infty}^t (I_1^{(2)}-TI_0^{(2)})^2dy  -     J_2 \Big(  \int_{-\infty}^t           (I_1^{(2)}-TI_0^{(2)}) I_0^{(2)} -  (I_1^{(3)}-TI_0^{(3)}) I_0^{(1)} dy    \Big)   \nonumber \\
 & -    ( \partial_T^{-1}  J_2)     \int_{-\infty}^t (I_0^{(2)})^2   -2    I_0^{(3)} I_0^{(1)}dy \nonumber  + \partial_t^{-2}\Big(  \big(I_3'''-tI_2'''  \big)'   I_0' +     I_2''' \big( I_1'-tI_0' \big)'  \Big)   -(T-t)\partial_t^{-1}(I_2'''I_0' )\nonumber \\
     &=\frac{1}{2}\frac{T''}{T'} + 2 \frac{I_0'}{I_0}+ \frac{\ln(1-5^5e^t)}{5}.  \nonumber \end{align}
\end{prop}

 By \eqref{eqnspecialgeometry}, we have $J_{2}=I_{2}/I_{0}=F_{T}$ and   $\partial_{T}J_{2}=F_{TT} ,  \partial_{T}^{-1}J_{2}=F$.  By   \eqref{eqnCttt} and \eqref{eqnABCY},  the R.H.S. of  the above identity is  ${1\over 2}A+2B+{1\over 5}\ln C_{ttt}^{-1}$.

The rest of this section will be devoted to proving Proposition \ref{lemmadesiredidentity}.
In what follows, we denote the left and right hand sides of the identity in Proposition \ref{lemmadesiredidentity}
by $LHS, RHS$ respectively.

\subsubsection{First simplification: anomalous term}

We first focus on the anomalous term on the left hand side: 
\begin{eqnarray*}
&&\int\int (I_{3}'''(y)-t I_{2}'''(y))' I_{0}'(y)dy
+\int\int I_{2}'''(y) (I_{1}'(y)-t I_{0}'(y))' dy -(T-t)\int^t I_{2}'''(y)I_{0}'(y)dy \nonumber \\
&=&
\int\int (I_{3}''''(y)I_{0}'(y)+I_{2}'''(y)I_{1}''(y))dy
 -T\int^t I_{2}'''(y)I_{0}'(y)dy\,.
\end{eqnarray*}

By integration by parts, one has
\begin{equation*}
\int I_{3}''''I_{0}'=I_{3}'''I_{0}'-\int I_{3}''' I_{0}''\,.
\end{equation*}
This gives
\begin{eqnarray*}
&&\int\int (I_{3}''''(y)I_{0}'(y)+I_{2}'''(y)I_{1}''(y))dy
-T\int^t I_{2}'''(y)I_{0}'(y)dy  \nonumber\\
&=&\int\int (-I_{3}'''(y)I_{0}''(y)+I_{2}'''(y)I_{1}''(y))dy
+\int^t I_{3}'''(y)I_{0}'(y)dy-T\int^t I_{2}'''(y)I_{0}'(y)dy
\,.
\end{eqnarray*}

Now applying integration by parts again, we get
\begin{equation*}
\int I_{3}'''I_{0}''=I_{3}''I_{0}''-\int I_{3}'' I_{0}'''\,.
\end{equation*}
Hence
\begin{equation*}
\int I_{3}'''I_{0}''={1\over 2}I_{3}''I_{0}''+ {1\over 2}\int (I_{3}'''I_{0}'' -I_{3}'' I_{0}''')\,.
\end{equation*}
Similarly, 
\begin{equation*}
\int I_{2}'''I_{1}''={1\over 2}I_{2}''I_{1}''+ {1\over 2}\int (I_{2}'''I_{1}'' -I_{2}'' I_{1}''')\,.
\end{equation*}
Therefore,
\begin{equation*}
\int (-I_{3}'''I_{0}''+I_{2}'''I_{1}'')=-{1\over 2}I_{3}''I_{0}''+{1\over 2}I_{2}''I_{1}''- {1\over 2}\int (I_{3}'''I_{0}'' -I_{3}'' I_{0}''')+ {1\over 2}\int (I_{2}'''I_{1}'' -I_{2}'' I_{1}''')\,.
\end{equation*}\\

To compute 
\begin{equation*}
- {1\over 2} (I_{3}'''I_{0}'' -I_{3}'' I_{0}''')+ {1\over 2} (I_{2}'''I_{1}'' -I_{2}'' I_{1}''')
=-{1\over 2}u_{6}^{(03)}+{1\over 2} u_{6}^{(12)}\,,
\end{equation*}
we use the Wronskian method discussed in Section \ref{secsymplecticstructureviaWronskian} above.
Here $u_{6}^{(03)}$ means $u_{6}$ with $(a,b)=(0,3)$ and 
$u_{6}^{(12)}$ is similar.
We recall that 
the two $u_{1}$'s are equal,
hence so are
\begin{equation*}
u_{2}=u_{1}'\,,\quad u_{2}'=u_{3}+u_{4}
\end{equation*}
and 
\begin{equation*}
U=a_{3}u_{4}+2u_{5}\,,
\quad
C_{1}u_{4}+2a_{0}u_{1}+2u_{6}\,,
\end{equation*}
where
\begin{equation*}
C_{1}=a_{3}'+{1\over 2}a_{3}^2-2a_{2}={4\over 5}C\,.
\end{equation*}
The difference 
\begin{equation*}
\Delta u_{6}=u_{6}^{(03)}-u_{6}^{(12)}=M_{03}-M_{12}
\end{equation*}
can be found by 
using the Wronskian method as follows.
Observe that all of the equations involved are linear. To figure out $\Delta u_{6}$, we 
recall
\begin{equation*}
u_{6}'=a_{0}u_{2}+a_{1}u_{4}-a_{3}u_{6}\,.
\end{equation*}
This gives
\begin{equation*}
\Delta (C_{1}u_{4}+2a_{0}u_{1}+2u_{6})
=C_{1}\Delta u_{4}+2 \Delta u_{6}=0\,,
\end{equation*}
 and
 \begin{equation*}
\Delta (u_{6}'
-a_{0}u_{2}-a_{1}u_{4}+a_{3}u_{6})
=(\Delta u_{6})'-a_{1}(\Delta u_{4})
+a_{3}\Delta u_{6}=0\,.
\end{equation*}
 It follows that
 \begin{equation*}
( \Delta u_{6})'+{a_{1}} {2\over C_{1}}\Delta u_{6}+a_{3}\Delta u_{6}=0\,.
 \end{equation*}
 Plugging in
 $C_{1}={4\over 5}C, a_{1}=-{2\over 5}C, a_{3}=-2C$
, we obtain  
 \begin{equation*}
 \Delta u_{6}=k C'=kC(C+1)\,.
 \end{equation*}
 for some constant $k$.
 This constant can be fixed by looking at the boundary conditions
 of $- {1\over 2} (I_{3}'''I_{0}'' -I_{3}'' I_{0}''')+ {1\over 2} (I_{2}'''I_{1}'' -I_{2}'' I_{1}''')$
 obtained by the first few terms of the series
obeying \eqref{eqnspecialgeometry} and \eqref{eqnprepotential}.
This fixes $k=-\frac{1}{5}$.

\begin{rema}
To minimize the computations, we can also prove the identity as follows.
Note that the same reasoning as above tells us that
\begin{equation*}
\Delta u_{3}+\Delta u_{4}=0\,,
\end{equation*}
\begin{equation*}
(\Delta u_{3})'
=(\Delta u_{4})'-a_{3}\Delta u_{3}
=(-\Delta u_{3})'-a_{3}\Delta u_{3}\,.
\end{equation*}
Solving the latter differential equation, we see that 
$\Delta u_{3}=a / \beta$
for some constant $a$.
Then we get
$\Delta u_{4}=-a / \beta
$
and hence
\begin{equation*}
\Delta u_{6}=-{C_{1}\over 2} \Delta u_{4}= {C_{1}\over 2}{a\over \beta}
={2\over 5}C {a\over\beta}\,.
\end{equation*}
Hence the constant $k$ in the previous discussion is related to $a$ by $k={2\over 5}a$.
But now computing $\Delta u_{3}$ or $\Delta u_{4}$ is much easier:
one has
\begin{eqnarray*}
\Delta u_{3}&=& - {1\over 2} (I_{3}'''I_{0} -I_{3} I_{0}''')+ {1\over 2} (I_{2}'''I_{1} -I_{2} I_{1}''')\,,\\
\Delta u_{4}&=& - {1\over 2} (I_{3}''I_{0}' -I_{3}' I_{0}'')+ {1\over 2} (I_{2}''I_{1}' -I_{2}' I_{1}'').
\end{eqnarray*}
Due to the asymptotic behaviour of $1/\beta$, this amounts to computing the constant term.
We consider $\Delta u_{4}=-{a / \beta}=-a+\mathcal{O}(q)$ for simplicity.
Now knowing that the first term $I_{0}$ is $1$ from \eqref{eqnboundaryconditionsforI0I1} is enough to fix $-a={1/ 2}$.
This gives $k=-{1/5}$.
\end{rema}

Therefore, we obtain 
\begin{equation*}
\int (-I_{3}'''I_{0}''+I_{2}'''I_{1}'')=-{1\over 2}I_{3}''I_{0}''+{1\over 2}I_{2}''I_{1}'' -{1\over 5}C\,.
\end{equation*}
It follows then the anomalous term considered above is given by
\begin{eqnarray*}
&&\int (-{1\over 2}I_{3}''(y) I_{0}''(y) +{1\over 2}I_{2}''(y) I_{1}''(y))dy
+\int^t I_{3}'''(y)I_{0}'(y)dy-T\int^t I_{2}'''(y)I_{0}'(y)dy
\,.
\end{eqnarray*}

Now the left hand side of the desired identity in Proposition \ref{lemmadesiredidentity} is simplified into
\begin{eqnarray*}
LHS&=&
-{1\over 2}F_{TT}\int^{t} (I_{1}''(y)-T(t)  I_{0}''(y))^2 dy \nonumber\\
&&
-F_{T}\int^{t} (I_{1}''(y)-T(t)  I_{0}''(y))  I_{0}''(y)dy+F_{T}\int^t (I_{1}'''(y)-T(t)  I_{0}'''(y))I_{0}'(y)  dy \nonumber \\
&&
-F\int^t   I_{0}''(y)^2 dy+2F\int^t I_{0}'''(y)I_{0}'(y)dy \nonumber \\
&&+\int (-{1\over 2}I_{3}''I_{0}''+{1\over 2}I_{2}''I_{1}'')
' +\int^t I_{3}'''I_{0}'-T\int^t I_{2}'''I_{0}' +
\int -{1\over 5}C \nonumber \\
&=&
-{1\over 2}F_{TT}\int^{t} (I_{1}''(y)-T(t)  I_{0}''(y))^2 dy \nonumber\\
&&
-2 F_{T}\int^{t} (I_{1}''(y)-T(t)  I_{0}''(y))  I_{0}''(y)dy -  3F\int^t   I_{0}''(y)^2 dy\nonumber \\
&&
+F_{T} (I_{1}''I_{0}'- T I_{0}''I_{0}')+ 2FI_{0}''I_{0}' \nonumber \\
&&+\int (-{1\over 2}I_{3}''I_{0}''+{1\over 2}I_{2}''I_{1}'')
 +\int^t I_{3}'''I_{0}'-T\int^t I_{2}'''I_{0}'
+
\int -{1\over 5}C 
\,.
\end{eqnarray*}

\subsubsection{Taking the derivative}

We can apply Proposition \ref{propantiderivativeI0Ik''''}, Proposition \ref{propantiderivativeIkI0''''} and integral by parts
to simply the single integrals involving $I_{0}$ in the quantity $LHS$.
But the double integral seems to be out of reach. Both this and the $\log C_{ttt}^{-1}=\int C$ term on the right hand side
of the expected identity in  Proposition \ref{lemmadesiredidentity} suggest that one takes
the derivative of the identity.

The derivative is then computed to be
\begin{eqnarray*}
&&\partial_{t}LHS
=
-{1\over 2}F_{TTT}T'\int^{t} (I_{1}''(y)-T(t)  I_{0}''(y))^2 dy\\
&& - F_{TT}T' \int^{t} (I_{1}''(y)-T(t)  I_{0}''(y))  I_{0}''(y)dy 
-F_{T}T'\int^{t} I_{0}''(y)^2 dy
\nonumber\\
&&
-T'\int I_{2}''' I_{0}'
-{1\over 2}F_{TT} (I_{1}''-T I_{0}'')^2-F_{T} (I_{1}''-T I_{0}'')I_{0}''- F I_{0}''^2
+I_{0}' \left(F_{T} (I_{1}''I_{0}'- T I_{0}'') + 2FI_{0}'' \right)'\nonumber \\
&&+ (-{1\over 2}I_{3}''I_{0}''+{1\over 2}I_{2}''I_{1}'')
+I_{3}'''I_{0}'-T I_{2}'''I_{0}'
-{1\over 5}C
\,.
\end{eqnarray*}

After plugging in \eqref{eqnantiderivativeI0''I0''},\eqref{eqnantiderivativeI1''I0''},\eqref{eqnantiderivativeI1''I1''}, 
the first two terms above reduce to expressions only involving $\mathcal{D}_{2},\mathcal{D}_{1}$
thanks to Proposition \ref{propantiderivativeI0Ik''''}, Proposition \ref{propantiderivativeIkI0''''}:
\begin{eqnarray*} 
\int ( I_{1}''(y)-T(t) I_{0}''(y))^2 dy
&=& T' I_{0} (I_{1}''-T I_{0}'')
-
  (T' I_{0}\beta) \mathcal{D}_{2}\circ \mathcal{D}_{1}\circ \partial_{t}I_{1}
+ T  (T' I_{0}\beta) \mathcal{D}_{2}\circ \mathcal{D}_{1}\circ \partial_{t}I_{0}
\,,\\
\int ( I_{1}''(y)-T(t) I_{0}''(y))I_{0}''(y) dy
&=& I_{0}I_{0}''T' -   (T' I_{0}\beta )\mathcal{D}_{2}\circ \mathcal{D}_{1}\circ \partial_{t}I_{0}
\,.
\end{eqnarray*}

We also apply integration by parts, Proposition \ref{propantiderivativeI0Ik''''} and Proposition \ref{propantiderivativeIkI0''''}
to simplify the other integrals.
This then gives a differential polynomial in the generators $A,B,C$ displayed in \eqref{eqnABCY}.
Below are the details.

\begin{eqnarray*}
-F_{T}T'\int^{t} I_{0}''(s)^2 ds  
&=&-F_{T}T'( I_{0}'' I_{0}'-I_{0}'''I_{0} )-F_{T}T' (  I_{0}\beta)\mathcal{D}_{3}\circ \mathcal{D}_{2}\circ \mathcal{D}_{1}\circ \partial_{t}I_{0}\,,\\
-T'\int I_{2}''' I_{0}'  
&=&T' [-I_{2}'' I_{0}' +I_{2}'I_{0}''-I_{2}I_{0}'''+  F_{T}  (  I_{0}\beta)\mathcal{D}_{3}\circ \mathcal{D}_{2}\circ \mathcal{D}_{1}\circ \partial_{t}I_{0} \nonumber\\
&&
-F_{TT} (T'I_{0}\beta) \mathcal{D}_{2}\circ \mathcal{D}_{1}\circ \partial_{t}I_{0} +(F_{TTT}T'^2I_{0}\beta)    \mathcal{D}_{1}\circ \partial_{t}I_{0}   ]
\,.
\end{eqnarray*}

\subsubsection{Completion of the proof of Proposition \ref{lemmadesiredidentity}}

Plugging in all of the above expressions for the anti-derivatives, we obtain
 \begin{eqnarray*}
\partial_{t}LHS
&=&
-{1\over 2}F_{TTT}T'  
 T' I_{0} (I_{1}''-T I_{0}'')- F_{TT}T'      I_{0}I_{0}''T' 
-F_{T}T'( I_{0}'' I_{0}'-I_{0}'''I_{0} )\nonumber\\
&&
-{1\over 2}F_{TT} (I_{1}''-T I_{0}'')^2-F_{T} (I_{1}''-T I_{0}'')I_{0}''- F I_{0}''^2\nonumber\\
&&
+I_{0}'  F_{TT}T' (I_{1}''- T I_{0}'') +I_{0}' F_{T} (I_{1}''-TI_{0}'')'
+ 2F_{T} T' I_{0}'  I_{0}'' +2F I_{0}'  I_{0}'''\nonumber\\
&&
+T'  (I_{2}'''I_{0}-I_{2}'' I_{0}' +I_{2}'I_{0}''-I_{2}I_{0}''')+ (-{1\over 2}I_{3}''I_{0}''+{1\over 2}I_{2}''I_{1}'')
+(I_{3}'''I_{0}'- I_{2}'''I_{1}')
\nonumber\\
&&
+{1\over 2}F_{TTT}T'
  (T' I_{0}\beta) \mathcal{D}_{2}\circ \mathcal{D}_{1}\circ \partial_{t}I_{1}\nonumber\\
&&
-{1\over 2}F_{TTT}T' T  (T' I_{0}\beta) \mathcal{D}_{2}\circ \mathcal{D}_{1}\circ \partial_{t}I_{0}
  +T'(F_{TTT} T'^2I_{0}\beta)    \mathcal{D}_{1}\circ \partial_{t}I_{0}   -{1\over 5}C
\nonumber\\
 &=&0+{1\over 2}F_{TTT}T'
  (T' I_{0}\beta) \mathcal{D}_{2}\circ \mathcal{D}_{1}\circ \partial_{t}I_{1} -{1\over 2}F_{TTT}T' T  (T' I_{0}\beta) \mathcal{D}_{2}\circ \mathcal{D}_{1}\circ \partial_{t}I_{0}
 \nonumber \\
&& +T'(F_{TTT} T'^2I_{0}\beta)    \mathcal{D}_{1}\circ \partial_{t}I_{0} -{1\over 5}C  \,.
\end{eqnarray*}

Keeping simplifying, we obtain 
\begin{eqnarray*}
\partial_{t}LHS  
&=&
+{1\over 2}F_{TTT}T'
  (T' I_{0}\beta)  [T'''I_{0}+3 T'' I_{0}'+3 T' I_{0}''\nonumber\\
  &&
 -(A+2B) (T'' I_{0}+2T' I_{0}')
 +(AB+B^2-B') T' I_{0}  ]  \nonumber\\
&& +T'(F_{TTT} T'^2\beta)    (I_{0}'' I_{0}-I_{0}' I_{0}')-{1\over 5}C\nonumber\\
&=&({1\over 2}A'+B')
+B' -{1\over 5}C\,.
\end{eqnarray*}

This 
 is identical to
\begin{equation*}
\partial_{t}RHS=\partial_{t}
\left({1\over 2}A+2B+{1\over 5}\ln C_{ttt}^{-1}\right)={1\over 2}A'+2B'-{1\over 5}C\,.
\end{equation*}

The identity $LHS=RHS$
then follows from the above equation $\partial_{t}LHS=\partial_{t}RHS$,
and the fact that they have the same boundary conditions which follow from \eqref{eqnboundaryconditionsforI0I1},\eqref{eqnspecialgeometry} and \eqref{eqnprepotential}.

\end{appendices}


\begin{thebibliography}{ABCD}


 









\bibitem[AZ]{Almkvist:2004differential}
G. Almkvist and W. Zudilin, \emph{Differential equations, mirror maps and
  zeta values},  Mirror symmetry V,  481-515 (2006).


\bibitem[BCOV]{BCOV} M.~Bershadsky,
S.~Cecotti, H.~Ooguri and C.~Vafa, {\em Holomorphic Anomalies in
Topological Field Theories}, Nucl.Phys. B {\bf 405}  279-304 (1993);   {\em Kodaira-Spencer Theory of Gravity and Exact Results for Quantum
String Amplitudes}, Comm. Math. Phys. Volume {\bf 165}, no. 2,
311-427 (1994).







\bibitem[CDFLL]{Ceresole:1992su}
A. Ceresole, R.~D'Auria, S.~Ferrara, W.~Lerche, and J.~Louis,
  \emph{{Picard-Fuchs equations and special geometry}}, Int.J.Mod.Phys.
  \textbf{A8}  79-114 (1993).

\bibitem[CDFLLR]{Ceresole:1993qq}
A. Ceresole, R.~D'Auria, S.~Ferrara, W.~Lerche, J.~Louis, and T.~Regge,
  \emph{{Picard-Fuchs equations, special geometry and target space duality}},
  Mirror symmetry II, 281-353 (1997).


\bibitem[CJ]{CJ} A. Collino, M. Jinzenji,
``On the structure of the small quantum cohomology rings of projective hypersurfaces,''
Comm. Math. Phys. {\bf 206}, no. 1, 157-183 (1999).

 
\bibitem[CK1]{CK1} I. Ciocan-Fontanine and B. Kim, \emph{Moduli stacks of stable toric quasimaps}, Advances in Math. 
{\bf 225},  no. 6, 3022-3051 (2010).

\bibitem[CK2]{CK2}  I. Ciocan-Fontanine and B. Kim, \emph{Quasimap Wall-crossings and Mirror Symmetry}, arXiv:1611.05023.  

\bibitem[CL]{CL} H.-L. Chang and J. Li,  \emph{Gromov-Witten invariants of stable maps with fields},
Int. Math. Res. Not.  {\bf 18}, 4163-4217 (2012).

 
\bibitem[CL1]{CL1} H.-L.  Chang and Jun Li, \emph{An algebraic proof of the hyperplane property of the genus one GW-invariants of quintics}, J. Diff. Geom. \black {\bf 100}, no. 2, p 251-299 (2015).



 \bibitem[CLL]{CLL} H.-L. Chang, J. Li, W.-P. Li,  \emph{Witten's top Chern classes via cosection localization},  Invent. Math. {\bf 200}, no 3, 1015-1063   (2015).


 
 \bibitem[CLLL1]{CLLL1} H.-L. Chang, J. Li, W.-P. Li, C.-C. Melissa Liu,
 \emph{Mixed-Spin-P fields of Fermat quintic polynomials},   math.AG. arXiv:1505.07532 . 

\bibitem[CLLL2]{CLLL2} H.-L. Chang, J. Li, W.-P. Li, C.-C. Melissa Liu,
 \emph{An effective theory of GW and FJRW invariants of quintics Calabi-Yau manifolds},    arXiv:1603.06184 (2016).
 
 \bibitem[CLLL3]{CLLL3} H.-L. Chang, J. Li, W.-P. Li, C.-C. Melissa Liu,
 \emph{A survey on mixed spin P-fields},  Chin. Ann. Math. Ser. B {\bf 38}, no. 4, 869-882  (2017).

 













\bibitem[F]{Freed:1999sm}
D. Freed, \emph{Special {K}\"ahler manifolds}, Comm. Math. Phys.
  \textbf{203} , no.~1, 31-52 (1999). 
  
  
\bibitem[FJR1]{FJR1} H.-J. Fan, T. J. Jarvis, Y.-B. Ruan, 
 \emph{The Witten equation, mirror symmetry, and quantum singularity theory},
Ann. of Math (2) {\bf 178}, no. 1, 1-106  (2013).

\bibitem[FJR2]{FJR2} H-J Fan, T. J. Jarvis and Y.-B. Ruan, \emph{The Witten equation and its virtual fundamental cycle}, math.AG. arXiv:0712.4025.





\bibitem[Gi1]{Gi1} A. Givental, \emph{ The mirror formula for quintic threefolds.} Northern California Symplectic Geometry Seminar, 49-62, Amer. Math. Soc. Transl. Ser. 2, 196, Amer. Math. Soc., Providence, RI, (1999).

\bibitem[Gi2]{Gi2} A. Givental, 
``Equivariant Gromov-Witten invariants,''  
math.berkeley.edu/~giventh/papers/eqv.pdf.

\bibitem[GP]{GP} T. Graber, R. Pandharipande,   \emph{Localization of virtual classes}, Invent. Math. {\bf 135}, no. 2, 487-518 (1999).









\bibitem[H]{Hosono:2008ve}
S. Hosono, \emph{{BCOV ring and holomorphic anomaly equation}}, Adv.Stud.Pure Math. {\bf 59}, 79-110 (2008).














 


\bibitem[KL]{KL} Y.H. Kiem and J. Li,  \emph{Localized virtual cycle by cosections},
J. Amer. Math. Soc. {\bf 26}, no. 4, 1025-1050   (2013).  



\bibitem[KLh]{KLh} B. Kim and H. Lho, \emph{Mirror Theorem for Elliptic Quasimap Invariants},
  arXiv:1506.03196.




  






\bibitem[LY]{Lian:1994zv}
B.~Lian and S.-T. Yau, \emph{{Arithmetic properties of mirror map and
  quantum coupling}}, Commun.Math.Phys. \textbf{176}, 163-192  (1996).


\bibitem[LZ]{LZ} J. Li and A. Zinger, \emph{On the Genus-One Gromov-Witten Invariants of Complete Intersections}, J. Diff. Geom. {\bf 82}, no. 3, 641-690  (2009).

\bibitem[LLY]{LLY} B. Lian, K.F. Liu and S.-T. Yau, \emph{Mirror principle. I }  Surveys in differential geometry: differential geometry inspired by string theory, 405-454,  Surv. Differ. Geom. 5, Int. Press, Boston, MA, (1999).








\bibitem[MP]{MP} D. Maulik,  and R. Pandharipande,  {\em A topological view of Gromov-Witten theory}, 
Topology  {\bf 45}, no. 5, 887-918  (2006).







\bibitem[P]{P} R. Pandharipande,
``Rational curves on hypersurfaces [after A. Givental]," S\'{e}minaire Bourbaki 848, 50\`{e}me ann\'{e}e, 1997-1998.

\bibitem[PZ]{PZ} A. Popa and A. Zinger,
``Mirror symmetry for closed, open, and unoriented Gromov-Witten invariants,''
Adv. Math. {\bf 259}, 448-510  (2014).
 
 
\bibitem[S]{Strominger:1990pd}
A. Strominger, \emph{{Special Geometry}}, Commun.Math.Phys. \textbf{133}, 163-180  (1990) .
  
\bibitem[VZ]{VZ} R. Vakil and A. Zinger, \emph{A Desingularization of the Main Component of the Moduli Space of 
Genus-One Stable Maps into $\PP^n$}, Geom. Topol. 12, no. 1, 1-95 (2008).


 




\bibitem[YY]{Yamaguchi:2004bt}
S. Yamaguchi and S.-T. Yau, \emph{{Topological string partition
  functions as polynomials}}, JHEP (2004) no.7, 047, 20pp.


\bibitem[Z]{Zhou:2013hpa} 
  J.~Zhou,
  \emph{Differential Rings from Special K\"ahler Geometry},
  arXiv:1310.3555 [hep-th].



\bibitem[Zi2]{Zi2} A. Zinger ``The Reduced Genus One Gromov-Witten Invariants of Calabi-Yau Hypersurfaces",  JAMS. {\bf 22},  no.3, 691-737 (2009). 






 

\end{thebibliography}
\end{document}